\documentclass{article}
\usepackage{psfrag}
\usepackage[parfill]{parskip}
\usepackage{float, graphicx}
\usepackage[]{epsfig}
\usepackage{amsmath, amsthm, amssymb}
\usepackage{epsfig}
\usepackage{verbatim}
\usepackage{multicol}
\usepackage{url}
\usepackage{latexsym}
\usepackage{mathrsfs}
\usepackage[colorlinks, bookmarks=true]{hyperref}
\usepackage{graphicx}
\usepackage{amsmath}
\usepackage{enumerate}
\usepackage[normalem]{ulem}
\usepackage{bm}
\usepackage{stmaryrd}
\usepackage{tikz}
\usepackage[T1]{fontenc}

\usepackage{cite}

\usepackage{color}

\makeatletter
\def\thm@space@setup{%
  \thm@preskip=\parskip \thm@postskip=0pt
}
\makeatother

\numberwithin{equation}{section}

\newcommand{\N}{\mathbb{N}}
\newcommand{\Y}{\mathbb{Y}}
\newcommand{\R}{\mathbb{R}}

\newcommand{\bflambda}{\pmb{\lambda}}
\newcommand{\bfmu}{\pmb{\mu}}

\newcommand{\bfx}{\bm{x}}
\newcommand{\bfy}{\bm{y}}

\renewcommand{\cal}{\mathcal}

\newcommand{\cC}{{\cal C}}
\newcommand{\cD}{\cal D}
\newcommand{\cE}{{\cal E}}
\newcommand{\cF}{{\cal F}}
\newcommand{\cG}{{\cal G}}

\newcommand{\cK}{{\cal K}}
\newcommand{\cL}{{\cal L}}

\newcommand{\cJ}{{\cal J}}

\newcommand{\cP}{{\cal P}}

\newcommand{\cS}{{\mathcal S}}

\newcommand\cT{{\mathcal T}}

\newcommand{\fb}{{\frak b}}
\newcommand{\fc}{{\frak c}}
\newcommand{\fd}{{\frak d}}

\newcommand{\fm}{{\frak m}}

\newcommand{\fp}{{\frak p}}
\newcommand{\fG}{{\frak G}}

\newcommand{\bml}{{\bm{l}}}
\newcommand{\bmp}{{\bm p}}

\newcommand{\bmx}{{\bm{x}}}
\newcommand{\bmy}{{\bm{y}}}

\newcommand{\bmeta}{{\bm \eta}}

\newcommand{\rd}{{\rm d}}

\newcommand{\ri}{\mathrm{i}}

\newcommand{\bC}{{\mathbb C}}

\newcommand{\bE}{\mathbb{E}}

\newcommand{\bH}{\mathbb{H}}
\newcommand{\bP}{\mathbb{P}}
\newcommand{\bQ}{\mathbb{Q}}
\newcommand{\bR}{{\mathbb R}}

\newcommand{\bT}{\mathbb T}

\newcommand{\bZ}{\mathbb{Z}}

\newcommand{\bY}{\mathbb{Y}}

\newcommand{\la}{\lambda}

\DeclareMathOperator{\supp}{supp}

\DeclareMathOperator{\sgn}{sgn}
\DeclareMathOperator{\OO}{O}
\DeclareMathOperator{\oo}{o}

\renewcommand{\Im}{\mathop{\mathrm{Im}}}

\newcommand{\deq}{\mathrel{\mathop:}=} 
\newcommand{\eqd}{=\mathrel{\mathop:}} 
\renewcommand{\leq}{\leqslant}
\renewcommand{\geq}{\geqslant}

\newcommand{\del}{\partial}

\newcommand{\qq}[1]{[\![{#1}]\!]}

\newcommand{\beq}{\begin{equation}}
\newcommand{\eeq}{\end{equation}}

\theoremstyle{plain} 
\newtheorem{theorem}{Theorem}[section]
\newtheorem*{theorem*}{Theorem}
\newtheorem{lemma}[theorem]{Lemma}
\newtheorem*{lemma*}{Lemma}
\newtheorem{corollary}[theorem]{Corollary}
\newtheorem*{corollary*}{Corollary}
\newtheorem{proposition}[theorem]{Proposition}
\newtheorem*{proposition*}{Proposition}
\newtheorem{assumption}[theorem]{Assumption}
\newtheorem*{assumption*}{Assumption}
\newtheorem{claim}[theorem]{Claim}

\newtheorem{definition}[theorem]{Definition}
\newtheorem*{definition*}{Definition}

\newtheorem*{example*}{Example}
\newtheorem{remark}[theorem]{Remark}

\newtheorem*{remark*}{Remark}
\newtheorem*{remarks*}{Remarks}

\def\author#1{\par
    {\centering{\authorfont#1}\par\vspace*{0.05in}}
}

\def\titlefont{\fontsize{13}{15}\bfseries\boldmath\selectfont\centering{}}
\def\authorfont{\fontsize{13}{15}}

\let\affiliationfont\rhfont

\def\address#1{\par
    {\centering{\affiliationfont#1\par}}\par\vspace*{11pt}
}

\def\body{
\setcounter{footnote}{0}
\def\thefootnote{\alph{footnote}}
\def\@makefnmark{{$^{\rm \@thefnmark}$}}
}

\def\title#1{
    \thispagestyle{plain}
    \vspace*{-14pt}
    \vskip 79pt
    {\centering{\titlefont #1\par}}%
    \vskip 1em
}

\newcommand{\bmla}{{\bm{\lambda}}}
\newcommand{\bmmu}{{\bm{\mu}}}

\newcommand{\bmnu}{{\bm \nu}}

\newcommand{\fcov}{{\mathfrak{cov}}}

\newcommand{\cov}{{\rm{cov}}}

\newcommand{\Sym}{\mathrm{Sym}}

\newcommand{\Sp}{\mathrm{Sp}}

\begin{document}
\title{Law of Large Numbers and Central Limit Theorems by Jack Generating Functions}

\vspace{1.2cm}

 \author{Jiaoyang Huang}
\address{Harvard University\\
   E-mail: jiaoyang@math.harvard.edu}

~\vspace{0.3cm}

\begin{abstract}
In a series of papers \cite{MR3361772,BG1,BG2} by Bufetov and Gorin, Schur generating functions as the Fourier transforms on the unitary group $U(N)$, are introduced to study the asymptotic behaviors of random $N$-particle systems. We introduce and study the Jack generating functions of random $N$-particle systems. In special cases, this can be viewed as the Fourier transforms on the Gelfand pairs $(GL_N(\bR), O(N))$, $(GL_N(\bC), U(N))$ and $(GL_N(\bH), Sp(N))$. Our main results state that the law of large numbers  and the central limit theorems for such particle systems, is equivalent to certain conditions on the germ at unity of their Jack generating functions. Our main tool is the Nazarov-Sklyanin operators \cite{MR3141546}, which have Jack symmetric functions as their eigenfunctions. As applications, we derive asymptotics of Jack characters, prove law of large numbers and central limit theorems for the Littlewood-Richardson coefficients of zonal polynomials, and show that the fluctuations of the height functions of a general family of nonintersecting random walks are asymptotically equal to
those of the pullback of the Gaussian free field on the upper half plane.

\end{abstract}


\begingroup
\hypersetup{linkcolor=black}
 \tableofcontents
\endgroup

\date{\today}

\vspace{-0.7cm}

\section{Introduction}

\subsection{Overview}

During the eighteenth century, De Moivre was the first to show that the Gaussian (or Normal) distribution
describes the fluctuations of the sums of independent binomial variables. Soon after, Laplace generalized
his result to more general independent variables. Gauss advertised this central limit theorem by showing
that it allows the evaluation of the errors in a large class of systems characterized by independence. One classical approach for the proof of the central limit is via the Fourier transform. The \emph{inversion theorem}, which was first demonstrated by L{\'e}vy, states that the probability measure is uniquely determined by its Fourier transform. Gaussian distribution is characterized by the derivative at origin of the logarithm of its Fourier transform, or equivalently the cumulants.

The theory of Fourier transform extends naturally to Gelfand pairs $(G,K)$, where $G$ is a locally compact group, and $K$ is a compact subgroup of $G$, and the classical Fourier transform corresponds to the Gelfand pair $(\bR, \{0\})$. For Gelfand pairs $(G,K)$, the zonal spherical functions, or positive definite spherical functions, corresponding to irreducible unitary representations of $G$ of class one, play the role of the multiplicative characters in the classical Fourier transform. In this paper, we generalize the proof of the central limit via the Fourier transform, to the setting of Gelfand pairs $(GL_N(\bR), O(N))$, $(GL_N(\bC), U(N))$ and $(GL_N(\bH), Sp(N))$, corresponding to orthogonal, unitary and symplectic groups. In these settings, the dual space, i.e., the set of zonal spherical functions is parametrized by $N$-tuples of nonnegative integers. Probability measures on the dual space can be viewed as random $N$-particle systems on $\bZ_{>0}$. In other words, we study random $N$-particle systems on $\bZ_{>0}$ via the Fourier transform. Our main results state that for a measure on the dual space, the germ at unity of its Fourier transform contains complete information about the law of large numbers and the central limit theorems.

For the Gelfand pairs $(GL_N(\bR), O(N))$, $(GL_N(\bC), U(N))$ and $(GL_N(\bH), Sp(N))$, their zonal spherical functions correspond to the zonal polynomials, Schur polynomials, and symplectic zonal polynomials respectively. They are special cases of the Jack polynomials, which are a family of symmetric functions dependent on a positive parameter $\theta$, introduced by Henry Jack in \cite{MR0289462,MR0306166}. Up to multiplicative constants, for $\theta=1$, Jack polynomials coincide with Schur polynomials, for $\theta=1/2$, they coincide with zonal polynomials, for $\theta=2$, they coincide with symplectic zonal polynomials. The zonal spherical functions correspond to the zonal polynomials, Schur polynomials, and symplectic zonal polynomials are special cases Jack characters, as introduced in \cite{Cuenca}. For random $N$-particle systems, we introduce ``Fourier transforms'' using Jack characters, called the Jack generating functions. When $\theta=1/2, 1, 2$, they are the Fourier transforms on the Gelfand pairs $(GL_N(\bR), O(N))$, $(GL_N(\bC), U(N))$ and $(GL_N(\bH), Sp(N))$. 

%
%
%


In a series of papers \cite{MR3361772,BG1,BG2} by Bufetov and Gorin, Schur generating functions are introduced to study the asymptotic behaviors of $N$-particle systems, which can be viewed as a version of Fourier transform on the unitary group, or the Gelfand pair $(GL_n(\bC), U(n))$. They prove that the law of large numbers and central limit Theorems of the $N$-particle systems hold if and only if partial derivatives of the logarithm of the Schur generating function at unity converge. In this paper, using the Jack generating functions, we prove that, the same as in \cite{MR3361772,BG1,BG2}, the law of large numbers and central limit theorems of the $N$-particle systems can be extracted from the asymptotic behavior of the Jack generating function at unity. 

One key idea in \cite{MR3361772,BG1,BG2} is to apply certain differential operators diagonalized by Schur polynomials to the Schur generating functions, which leads to polynomial formulas relating the expectations of the products of the moments, with partial derivatives of the logarithm of the Schur generating functions at unity. 
The idea of using differential operators on representation spaces for groups that act in an eigen-fashion on each of the irreducible components has a long history. In the context of Lie groups, a natural set of such differential operators is provided by the center of the universal enveloping algebra of the corresponding Lie algebra, the Casimir operators. The usage of Casimir operators in probabilistic problems first appeared in the pioneer works of Kerov and Olshanski \cite{MR1288389,MR1204294}. These ideas were further developed to derive Gaussian free field asymptotics in \cite{MR3263029,MR3349009}. The $q$-analog of the Casimir operators lead Macdonald to construct the famous Macdonald polynomials, which are eigenfunctions for an algebra of commuting self adjoint operators. The probabilistic pointview of the Macdonald version of Casimir operators was used by Borodin and his collaborators to construct and study the Macdonald processes in \cite{MR3152785,MR3449217}. Its degeneration leads to the derivation of the Gaussian free field asymptotics of multilevel $\beta$-Jacobi random matrix ensembles in \cite{MR3385342}.


%

In our setting, we use the Nazarov-Sklyanin operators, originated in the work of Nazarov and Sklyanin \cite{MR3141546}, which have Jack symmetric functions (in infinitely many variables) as their eigenfunctions, and the eigenvalues are given by the moments of the Perelomov-Popov measures, see Section \ref{s:PPM}. Fortunately, the law of large numbers and central limit theorems for the Perelomov-Popov measures are equivalent to those of the original probability measures. The law of large numbers and central limit theorems can be accessed by applying the Nazarov-Sklyanin operators to the Jack generating functions, which leads to polynomial formulas expressing the expectations of the products of the moments of the Perelomov-Popov measures in terms of partial derivatives of the logarithm of the Jack generating functions at unity.  As a consequence, our methods give  new proofs of some results in \cite{MR3361772,BG1,BG2}.  The Nazarov-Sklyanin operators have been previously used by Moll to prove the law of large numbers and the central limit theorems of the Jack measures in \cite{Moll}, and more generally, any quantum integrable system. The Jack measures in \cite{Moll} are different from the probability measures we study,  where the number of particles are fixed.


As applications, our theorems lead to the proofs of the law of large numbers and central limit theorems of a variety of $N$-particle systems, including  Littlewood-Richardson coefficients for the zonal polynomials, and a rich family of discrete nonintersecting random walks. The same as for continuous $\beta$-ensembles, in all our examples, the law of large numbers do not depend on the parameter $\theta$, and the covariance structure in the central limit theorems is also $\theta$ independent up to a factor $\theta$. As applications of our inverse theorems, we derive asymptotics of the the Jack characters at unity.  It is worth mentioning that the asymptotics of Jack characters at unity, are the same as those of schur characters up to a factor $\theta$.

\subsection{Formulation}

Random $N$-particle systems on $\bZ_{>0}$ can be identified with probability measures on Young diagrams with at most $N$ rows. We recall that Young diagrams are defined as infinite sequences $\bflambda = (\lambda_1, \lambda_2, \dots)$ of nonnegative, nonincreasing integers, $\la_1\geq \la_2\geq \la_3\geq\cdots\geq 0$, with only finitely many nonzero parts. The length of $\bflambda$, defined as its number of nonzero parts, is denoted by $\ell(\bflambda)$. We can write a Young diagram as a finite sequence $\bflambda = (\lambda_1, \ldots, \lambda_N)$ as long as $N \geq \ell(\bflambda)$. For each $N\in\N$, we denote the set of Young diagrams of length $\leq N$ by  $\Y(N)$, and the set of all Young diagrams by $\Y := \cup_{N\geq 0}{\Y(N)}$.

Let $N\in \bZ_{>0}$ and $M_N$ be a probability measure on $\bY(N)$.  For any $\bmla=(\la_1,\la_2,\cdots,\la_N)\in \bY(N)$, with $\la_1\geq \la_2\geq \cdots\geq \la_N$, we define
\begin{align}\label{e:defyi}
y_i=\frac{\la_i}{\theta N}-\frac{i-1}{N},\quad i=1,2,\cdots, N.
\end{align}
Then the probability measure $M_N$ induces a random measure on $\bR$,
\begin{align}\label{def:mula}
\mu[\bmla]=\frac{1}{N}\sum_{i=1}^N\delta_{y_i},
\end{align}
which is called the empirical density.
The goal of this paper is to study the law of large numbers and the central limit theorems of the random measure $\mu[\bmla]$ as $N$ goes to infinity. The law of large numbers refer to the phenomenon that the sequence of the empirical densities $\mu[\bmla]$ converges weakly to a deterministic measure.  The central limit theorems refer to the phenomenon that the fluctuations of the random measures around its mean are asymptotically Gaussian, i.e. with polynomial test functions
\begin{align*}
N\int x^k (\rd \mu[\bmla]-\bE[\rd \mu[\bmla] ) =\sum_{i=1}^{N} y_i^k-\bE\left[\sum_{i=1}^{N} y_i^k\right],
\end{align*}
converge weakly to a Gaussian random variable. The scaling is different from the classical central limit theorem, where the fluctuation would grow as $N^{1/2}$ in the system. This is because the interactions of the $N$-particle systems we study are non-local and of log-gas type, and for them, typically, the fluctuations do not grow with the size of the system.

The precise formulation of the law of large numbers and central limit theorems of the measure $M_N$ is given in the following sense. 
\begin{definition}\label{def:LLNCLT}
Let $\{M_N\}_{N\geq 1}$ be a sequence of probability measures on $\bY(N)$,  and $\mu[\bmla]$ as defined in \eqref{def:mula}. We say $M_N$ (or $\mu[\bmla]$) satisfies the law of large numbers, if there exists a collection of numbers $\fp_1,\fp_2,\fp_3,\cdots$, the collection of random variables 
\begin{align*}
\int x^k \rd \mu[\bmla],\quad k=1,2,3,\cdots,
\end{align*}
converges, as $N\rightarrow \infty$, in the sense of moments, 
\begin{align*}
\lim_{N\rightarrow \infty}\bE \left[\prod_{i=1}^r \int x^{k_r} \rd \mu[\bmla]\right] = \prod_{i=1}^r\fp_{k_i},
\end{align*}
for any index $r\geq1$ and $k_1,k_2,\cdots, k_r\geq 1$.

We say $M_N$ (or $\mu[\bmla]$) satisfies central limit theorems,
if it satisfies law of large numbers, and the collection of random variables
\begin{align*}
\xi^{(k)}[\bmla]\deq N\left(\int x^k\rd \mu[\bmla]-\bE\int x^k\rd \mu[\bmla]\right),\quad k=1,2,3,\cdots
\end{align*} 
converges, as $N\rightarrow\infty$, in the sense of moments, to the Gaussian vector with zero mean, variance
\begin{align*}
\lim_{N\rightarrow \infty}\cov(\xi^{(k)}[\bmla],\xi^{(l)}[\bmla])=\fcov_{k,l},
\end{align*}
and for any index $r\geq 3$ and $k_1,k_2,\cdots, k_r\geq 1$, the $r$-th cumulant vanishes,
\begin{align*}
\lim_{N\rightarrow \infty}\kappa_{k_1,k_2,\cdots,k_r}(\xi^{(k)}[\bmla],k=1,2,3,\cdots)=0.
\end{align*}

\end{definition}

\subsection{Jack generating functions}


Let $\bm x=(x_1, x_2,x_3,\cdots)$ be an infinite set of indeterminates, and $\Sym$ the graded algebra of symmetric functions in infinitely many variables over $\bR$. We use the following notations for certain symmetric functions indexed by partitions $\bmla$:
\begin{itemize}
\item We denote the monomial symmetric functions $m_{\bmla}(\bmx)$, which is defined as the sum of all monomials $\bm{x}^{{\bmla_\sigma}}$ where ${\bmla_\sigma}$ ranges over all distinct permutations of $\bmla$.
\item We denote the power sum symmetric functions $p_{\bmla}(\bmx)$, as
\begin{align*}
p_{\bmla}(\bmx)=\prod_{1\leq i\leq \ell({\bmla})}p_{\la_i}(\bmx),\quad p_n(\bmx)=x_1^n+x_2^n+x_3^n+\cdots.
\end{align*}
\end{itemize}
Both the set of monomial symmetric functions $\{m_\bmla(\bmx): \bmla\in \bY\}$, and the set of power sum symmetric functions $\{p_\bmla(\bmx): \bmla\in \bY\}$ form homogeneous bases of $\Sym$. The algebra $\Sym$ can be identified with the algebra in the variables $\bmp=(p_1,p_2,p_3,\cdots)$, i.e. $\Sym=\bR[p_1,p_2, p_3,\cdots]$. There exists non-negative coefficients $c_{\bmla\bmmu}$ such that 
\begin{align}\label{def:c}
p_\bmla(\bmx)=\sum_{\bmmu\in \bY}c_{\bmla\bmmu}m_{\bmmu}(\bmx).
\end{align}
The coefficients $c_{\bmla\bmmu}$ are nonzero only if $|\bmla|=|\bmmu|$ and $\ell(\bmla)\geq\ell(\bmmu)$. Similarly there exists coefficients $\tilde c_{\bmla\bmmu}$ such that \begin{align}\label{def:tc}
m_\bmla(\bmx)=\sum_{\bmmu\in \bY}\tilde c_{\bmla\bmmu}p_{\bmmu}(\bmx).
\end{align}
The coefficients $\tilde c_{\bmla\bmmu}$ are nonzero only if $|\bmla|=|\bmmu|$ and $\ell(\bmla)\geq\ell(\bmmu)$.
If $\bmla$ has $m_i=m_i(\bmla)$ parts equal to $i$, then we write 
\begin{align*}
z_{\bmla}=\prod_{1\leq i\leq \la_1} i^{m_i(\bmla)}\prod_{1\leq i\leq \la_1}m_i(\bmla)!.
\end{align*}
Let $\theta>0$ be a parameter (indeterminante), and let $\bQ(\theta)$ denote the field of all rational functions of $\theta$ with rational coefficients. Define a scalar product $\langle\cdot, \cdot\rangle$ on the vector space $\Sym \otimes \bQ(\theta)$ of all symmetric functions over the field $\bQ(\theta)$ by the condition
\begin{align}\label{e:scalarP}
\langle p_{\bmla}(\bmx), p_{\bmmu}(\bmx)\rangle =\delta_{\bmla\bmmu}z_{\bmla} \theta^{-\ell(\bmla)}.
\end{align}

We recall that a partial order on Young diagrams is obtained by declaring $\bflambda \preceq \bfmu$ if $|\bflambda| = |\bfmu|$ and $\la_1+\la_2+\cdots+\la_i\leq \mu_1+\mu_2+\cdots+\mu_i$ for all $i$. The following fundamental result is due to Macdonald, which characterizes Jack symmetric functions.
\begin{theorem}
There are unique symmetric functions $J_{\bmla}(\bmx; \theta)\in \Sym\otimes \bQ(\theta)$, where $\bmla$ ranges over all Young diagrams, satisfying the following three conditions:
\begin{enumerate}
\item (orthogonality) $\langle J_{\bmla}(\bmx; \theta), J_{\bmmu}(\bmx;\theta)\rangle=0$ if $\bmla\neq \bmmu$, where the scalar product is given by \eqref{e:scalarP}.
\item (triangularity) We have
\begin{align*}
J_{\bmla}(\bmx; \theta)=\sum_{\bmmu}v_{\bmla \bmmu}(\theta)m_{\bmmu}(\bmx),
\end{align*} 
where $v_{\bmla\bmmu}(\theta)=0$ unless $\bmmu\preceq \bmla$.
\item (normalization) The leading term of $J_{\bmla}(\bmx;\theta)$ is given by $x_1^{\la_1}x_2^{\la_2}\cdots x_{\ell(\bmla)}^{\la_{\ell(\bmla)}}$.
\end{enumerate}
\end{theorem}
If we set all but finitely many variables equal to $0$ (say $x_{N+1}=x_{N+2}=\cdots=0$) in $J_\bmla(\bmx,\theta)$, then we obtain a polynomial $J_\bmla(x_1,x_2,\cdots, x_N;\theta)$.  
Jack symmetric functions are the eigenfunctions of the Nazarov-Sklyanin operators, which are differential operators on $\bR[p_1,p_2,p_3,\cdots]$, see Section \ref{s:NSop}. From now on, we will write the Jack symmetric functions in terms of the variables $\bmp=(p_1,p_2,p_3,\cdots)$.

\begin{definition}
Let $\{M_N\}_{N\geq 1}$ be a sequence of probability measures on $\bY(N)$. The \textit{Jack generating function of $M_N$} is given by
\begin{equation}\label{e:gener}
F_{N}(\bmp; \theta) := \sum_{\bflambda\in\Y(N)}
M_N(\bflambda)\frac{J_{\bflambda}(\bmp; \theta)}{J_{\bflambda}(1^N; \theta)}.
\end{equation}
\end{definition}
For any $\bflambda\in\Y(N)$ we have $J_{\bflambda}(1^N; \theta) > 0$ by Theorem \ref{t:J1N}, and therefore $(\ref{e:gener})$ does not have vanishing denominators. The sum in \eqref{e:gener} may be infinite; thus $F_N(\bmp;\theta)$ may not be an element of $\Sym$. However, it can be justified as an element in the completion $\widehat \Sym$.  By an abuse of notation, we use $\bmp=1^N$ to mean that we set $x_1=x_2=\cdots=x_N=1$ and $x_{N+1}=x_{N+2}=\cdots=0$. With this notation, it is easy to see that $F_N(\bmp;\theta)|_{\bmp=1^N}=1$.

The main results of this paper show the law of large numbers and central limit theorems of the measure $M_N$ can be extracted from the
asymptotic behaviors of the Jack generating functions $F_N(\bmp;\theta)$ at unity.

\subsection{Related results}
The most natural and studied measure on Young diagrams is the
Plancherel measure for the symmetric group,  which is the push-forward of the uniform distribution
on the symmetric group by the Robinson-Schensted-Knuth (RSK) correspondence. The limit shape for
the Plancherel measure  was first obtained by Vershik and Kerov \cite{MR0480398} and Logan
and Shep \cite{MR1417317}. The central limit theorems for the Plancherel measure was announced in the short note \cite{MR1204294}, where Kerov outlined the scheme of the proof. The detailed exposition of Kerov's central limit theorem can be found in \cite{MR1708561,MR2059361}

A deformation of Plancherel measure,
linked to the celebrated Jack polynomials, is the Jack deformation of Plancherel measure. It has appeared recently in several research papers \cite{MR2143199,MR2098845,MR2607329,MR2227852}. 
The law of large numbers and the central limit theorems of Jack deformation of Plancherel measure were recently proven in \cite{MR3498866}, using certain polynomiality result for the structure constant of Jack characters. A more general family of Jack deformed probability measures on Young diagrams, arising from decomposing certain Jack characters into irreducible characters was studied in \cite{Mac}, for which the limit shape and global fluctuation were understood.

Random Young diagrams are studied as discrete analogues of eigenvalues of the Gaussian Unitary Ensemble (GUE).  The empirical eigenvalue distributions of the GUE, or more generally, the Wigner matrices, converge weakly to the Semi-circle distribution. This statement follows from computing the moments of the eigenvalues, and dates back to the original work of Wigner \cite{MR0077805,MR0083848}. The study of the central limit theorems  for the linear statistics of the eigenvalues of random matrices can be traced back to \cite {MR650926},  in the context of Wishart matrices. The proof is based on a careful analysis of the moment method, see \cite[Chapter 2]{MR2760897} for an exposition. Another approach to the central limit theorems of the linear statistics is via analyzing the resolvent \cite{MR1411619,MR2189081}. This approach is further generalized to study sample covariance matrices \cite{MR2040792}, sparse random matrices \cite{MR2605074,MR2953145}, heavy tailed random matrices \cite{MR3210147}, and band matrices \cite{MR3392508}.

A family of $N$-particle systems called the $\beta$-ensembles ($\beta$ log-gas) is intensively studied by the mathematical community, which is a probability distribution on $N$-tuples of reals $x_1<x_2<\cdots<x_N$ with density proportional to
\begin{align*}
\prod_{1\leq i< j\leq N}(x_j-x_i)^\beta \prod_{i=1}^N e^{-NV(x_i)}\rd x_i,
\end{align*}
where $V$ is the potential. 
When $\beta = 1, 2,4$ and the potential $V$ is quadratic, this is the joint eigenvalue density for the Gaussian Orthogonal, Unitary, or Symplectic ensembles. The law of large numbers of the empirical density of the $\beta$-ensembles with general potential $V$ was established in \cite{MR1327898,MR1465640}. In the breakthrough paper \cite{MR1487983} by Johansson, the loop equations were introduced to the mathematical community, and used to prove the central limit theorems of the liner statistics of the $\beta$-ensembles. Further developments have led to establishing central limit theorems in more general settings, see \cite{borot-guionnet2,MR3010191,MR3063494,MR3468249,KrSh}. The discrete beta ensembles are introduced and studied in \cite{MR3668648}. The law of large numbers for such systems can be established by the same methods as for the continuous case, see \cite{MR2483725}. The  proof of the Gaussian fluctuation in \cite{MR3668648} is more remarkable and relies on an appropriate discrete version of the loop equations,
which originates in the work of Nekrasov and his collaborators \cite{Nek_PS,Nek_Pes,Nekrasov}. For the two-dimensional $\beta$ log-gas, the law of large numbers is well-known, see \cite{MR1485778}. The central limit theorems have been recently proven in \cite{MR2361453,MR2817648} for $\beta=2$ and in \cite{LeSe,RoBo} for general $\beta>0$. 

The discrete Markov chains in Section \ref{s:nrw}, are $\beta$ analogues of the nonintersecting random walks \cite{MR1887625}, which corresponds to $\beta=2$. The central limit theorems for the global fluctuations of the nonintersecting random walks were obtained by various methods. For the fully-packed initial data, the central limit theorems were established in \cite{MR3148098} by the technique of determinantal point processes, in \cite{MR3552537,MR3263029} by computations in the universal enveloping algebra of
$U(N)$ and in \cite{Dui, MR3556288} by employing finite term recurrence relations of orthogonal polynomials. For general initial data, the law of large numbers and the central limit theorems were recently addressed in \cite{MR3361772,BG1} with Schur generating functions.
Our results give a new proof of these results based on the Jack generating functions, and we recognize the fluctuations of the height functions as
those of the pullback of the Gaussian free field on the upper half plane.

\subsection{Organization of the paper}

In Section \ref{s:main}, we formulate our main results linking the law of large numbers and central limit theorems to the asymptotic behaviors of the Jack generating functions.  In Section \ref{s:PPM}, we recall the definition of the Perelomov-Popov measures, and show that the the law of large numbers and central limit theorems of the empirical density is equivalent to those of the Perelomov-Popov measures. In Section \ref{subs:NSOP}, we recall the Nazarov-Sklyanin Operators from \cite{MR3141546}, which have Jack symmetric functions (in infinitely many variables) as their eigenfunctions, and eigenvalues are given by the moments of the Perelomov-Popov measures. Section \ref{s:proof} presents the proofs of our main results in Sectoin \ref{s:main}. In Section \ref{s:applications}, we present applications of our main results, including the Littlewood-Richardson coefficients for zonal polynomials, and ensembles of nonintersecting random walks. We collect some basic properties of the Jack polynomials and the skew Jack polynomials in Section \ref{s:Jack}. In Section \ref{s:finiteJGF}, we define the Jack generating functions using Jack polynomials in $N$ variables, and prove that our Assumptions on the asymptotic behaviors of the Jack generating functions are equivalent to their counterparts in \cite{BG1}. 




\noindent\textbf{Acknowledgement.} I am thankful to Prof. Gorin for many enlightening discussions. I am indebted to Cesar Cuenca for introducing me to the discrete dynamics related to Jack symmetric functions which are studied in Section \ref{s:nrw}, and explaining to me his work on asymptotics of Jack characters. I also want to thank Prof. Borodin,  Alex Moll and Prof. Petrov for helpful remarks on the draft.

\section{Main Results}\label{s:main}

Let $\{M_N\}_{N\geq 1}$ be a sequence of probability measures on $\bY(N)$, with Jack generating function $F_N(\bmp;\theta)$ as defined in \eqref{e:gener}. $F_N(\bmp;\theta)$ is an element in the completion $\widehat{\Sym}$. We can rewrite it in terms of the power sum symmetric functions, 
\begin{align}\label{e:gener2}
F_N(\bmp;\theta)=\sum_{\bmla\in \bY}a_{M_N}(\bmla)p_{\bmla}.
\end{align}
We make the following assumption on the analyticity of the Jack generating function. Roughly speaking, it implies that in a neighborhood of $x_1=x_2=\cdots=x_N=1$, $x_{N+1}=x_{N+2}=\cdots=0$, \eqref{e:gener2} converges absolutely and analytic.
\begin{assumption}\label{asup:infinite}
There exists a small positive real $\varepsilon > 0$ such that for any $N\in\bZ_{>0}$, the function
\begin{align*}
F_N(\bmp;\theta)=\sum_{\bmla\in \bY}a_{M_N}(\bmla)p_{\bmla},
\end{align*}
satisfies,
\begin{align*}
\sum_{\bmla\in \bY}| a_{M_N}(\bmla)|N^{\ell(\bmla)}|x|^{|\bmla|},
\end{align*}
converges absolutely on the disk $\{z\in \bC, |z|< 1+\varepsilon\}$.
\end{assumption}
Under Assumption \ref{asup:infinite}, the derivatives of the logarithm of the Jack generating functions are summable. We postpone the proof of the following proposition to Section \ref{s:finiteJGF}.

\begin{proposition}\label{p:summable}
Given $F_N(\bmp;\theta)$ satisfies Assumption \ref{asup:infinite}. Then for any fixed $r\geq 1$, the infinite series 
\begin{align}\label{e:summable}
\sum_{i_1,i_2,\cdots,i_r\geq 1}\left|\left.\frac{\del^{r}\ln F_N(\bmp;\theta)}{\del p_{i_1}\del p_{i_2}\cdots \del p_{i_{r}}}\right|_{\bmp=1^N}\right||x|^{i_1+i_2+\cdots+i_r},
\end{align}
converges absolutely on the disk $\{z\in\bC: |z|<1+\varepsilon\}$.

\end{proposition}


\subsection{Law of large numbers}

Let $\{M_N\}_{N\geq 1}$ be a sequence of probability measures on $\bY(N)$, with Jack generating functions $F_N(\bmp;\theta)$ as defined in \eqref{e:gener}. 
We assume that the Jack generating functions $F_N(\bmp;\theta)$ satisfy Assumption \ref{asup:infinite} and the following assumption.
\begin{assumption}[Law of large numbers]\label{a:LLN}
We assume that the Jack generating functions $F_N(\bmp;\theta)$ satisfy Assumption \ref{asup:infinite}, and
\begin{enumerate}
\item For any $k\geq 1$,
\begin{align}\label{e:LLN1}
\lim_{N\rightarrow \infty}\left.\frac{1}{N}\sum_{i\geq 1}k!{i\choose k}\frac{\del \ln F_N(\bmp;\theta)}{\del p_i}\right|_{\bmp=1^N}= \fc_k.
\end{align}
\item For any $r\geq 2$ and indices $k_1,k_2,\cdots, k_r\geq 1$,
\begin{align}\label{e:LLN3}
\lim_{N\rightarrow \infty}\frac{1}{N}\sum_{i_1,i_2,\cdots, i_r\geq 1}{i_1\choose k_1}{i_2\choose k_2}\cdots {i_r\choose k_r} \left.\frac{\del^r \ln F_N(\bmp;\theta)}{\del p_{i_1}\del p_{i_2}\cdots \del p_{i_r}}\right|_{\bmp=1^N}=0.
\end{align}
\end{enumerate}
\end{assumption}

Under Assumption \ref{a:LLN}, we define a function 
\begin{equation}\label{eqn:Frho}
U_{\mu}(z) := \sum_{k=1}^{\infty}{\frac{\fc_k}{(k-1)!}(z - 1)^{k-1}},
\end{equation}
where the coefficients $\{\fc_k\}_{k\geq 1}$ are from Assumption \ref{a:LLN}. $U_\mu(z)$ may not converge, and should be viewed as a formal power series of $(z-1)$.

\begin{theorem}\label{thm:LLN}
Let $\{M_N\}_{N\geq 1}$ be a sequence of probability measures on $\bY(N)$, with Jack generating functions $F_N(\bmp;\theta)$ as defined in \eqref{e:gener}. We assume the Jack generating functions $F_N(\bmp;\theta)$ satisfy Assumption \ref{asup:infinite} and \ref{a:LLN}. Then, the random measures $\mu[\bmla]$ as defined in \eqref{def:mula} satisfy the law of large numbers in the sense of Definition \ref{def:LLNCLT}, and the moments are characterized by
\begin{equation}\label{eqn:moments}
\lim_{N\rightarrow\infty}\int_{\bR}{x^k \rd\mu[\bmla](x)} =[w^{-1}]\frac{1}{(k+1)}\left(\frac{(w+1)}{\theta}U_{\mu}(w+1)+\frac{1}{w}\right)^{k+1}\sum_{a\geq 0}(-w)^a,
\end{equation}
where the formal power series $U_\mu(z)$ is defined in \eqref{eqn:Frho}.
\end{theorem}

\subsection{Central limit theorems for one level}

Let $\{M_N\}_{N\geq 1}$ be a sequence of probability measures on $\bY(N)$, with Jack generating functions $F_N(\bmp;\theta)$ as defined in \eqref{e:gener}. 
We assume that the Jack generating functions $F_N(\bmp;\theta)$ satisfy Assumption \ref{asup:infinite} and the following assumption.
\begin{assumption}[Central limit theorems]\label{a:CLT}
We assume that the Jack generating functions $F_N(\bmp;\theta)$ satisfy Assumption \ref{asup:infinite}, and
\begin{enumerate}
\item For any $k\geq 1$,
\begin{align}\label{e:CLT1}
\lim_{N\rightarrow \infty}\left.\frac{1}{N}\sum_{i\geq 1}k!{i\choose k}\frac{\del \ln F_N(\bmp;\theta)}{\del p_i}\right|_{\bmp=1^N}= \fc_k.
\end{align}
\item 
For any $k,l\geq 1$,
\begin{align}\label{e:CLT2}
\lim_{N\rightarrow \infty}\sum_{i,j\geq 1}k!l!{i\choose k}{j\choose l}\left.\frac{\del^2 \ln F_N(\bmp;\theta)}{\del p_i\del p_j}\right|_{\bmp=1^N}= \fd_{k,l}.
\end{align}
\item 
For any $r\geq 3$ and indices $k_1,k_2,\cdots,k_r\geq 1$,
\begin{align}\label{e:CLT5}
\lim_{N\rightarrow \infty}\sum_{i_1,i_2,\cdots, i_r\geq 1}{i_1\choose k_1}{i_2\choose k_2}\cdots {i_r\choose k_r} \left.\frac{\del^r \ln F_N(\bmp;\theta)}{\del p_{i_1}\del p_{i_2}\cdots \del p_{i_r}}\right|_{\bmp=1^N}=0.
\end{align}
\end{enumerate}


\end{assumption}

Under Assumption \ref{a:CLT}, we define functions 
\begin{equation}\label{eqn:UV}
U_{\mu}(z) \deq \sum_{k=1}^{\infty}{\frac{\fc_k}{(k-1)!}(z - 1)^{k-1}},\quad
V_\mu(z,w)\deq \sum_{k,l=1}^{\infty}\frac{\fd_{k,l}}{(k-1)!(l-1)!}(z-1)^{k-1}(w-1)^{l-1},
\end{equation}
where the coefficients $\{\fc_k\}_{k\geq 1}$, $\{\fd_{k,l}\}_{k,l\geq 1}$ are from Assumption \ref{a:CLT}. $U_\mu(z)$ and $V_\mu(z)$ may not converge, and should be viewed as formal power series of $(z-1)$ and $(w-1)$.

\begin{theorem}\label{t:clt}
Let $\{M_N\}_{N\geq 1}$ be a sequence of probability measures on $\bY(N)$, with Jack generating functions $F_N(\bmp;\theta)$ as defined in \eqref{e:gener}. 
We assume that the Jack generating functions $F_N(\bmp;\theta)$ satisfies Assumption \ref{asup:infinite} and \ref{a:CLT}.
Then, the random measures $\mu[\bmla]$ as defined in \eqref{def:mula} satisfy the central limit theorems in the sense of Definition \ref{def:LLNCLT}, and the covariant structure is characterized by
\begin{align}\begin{split}\label{e:cltcov}
\lim_{N\rightarrow\infty}\cov(\xi^{(k)}[\bmla], \xi^{(l)}[\bmla])
&=[z^{-1}w^{-1}]
\left(\sum_{a\geq 1}\frac{a}{\theta}z^{a-1}w^{-a-1} +\frac{V_\mu(z+1,w+1)}{\theta^2}\right)\\
&\phantom{{}={}}\left( \frac{1}{z} +\frac{ (1 + z)}{\theta}U_{\mu}(1 + z) \right)^{k}
\left( \frac{1}{w} +\frac{ (1 + w)}{\theta}U_{\mu}(1 + w) \right)^{l},
\end{split}\end{align}
where
\begin{align*}
\xi^{(k)}[\bmla]\deq N\left(\int x^k\rd \mu[\bmla]-\bE\int x^k\rd \mu[\bmla]\right),\quad k=1,2,3,\cdots,
\end{align*} 
and the formal power series $U_\mu(z)$ and $V_\mu(z,w)$ are defined in \eqref{eqn:UV}.
\end{theorem}

\subsection{Central limit theorems for several levels}\label{s:mulclt}
Let $\{M_N\}_{N\geq 1}$, $\{\fm_N^{(1)}\}_{N\geq 1}, \{\fm_N^{(2)}\}_{N\geq 1}, \{\fm_N^{(3)}\}_{N\geq 1},\cdots $ be sequences of probability measures, where $M_N$, $\fm_N^{(1)}, \fm_N^{(2)}, \fm_N^{(3)},\cdots$ are probability measures on $\Y(N)$. We denote the  Jack generating functions of $M_N$, $\fm_N^{(1)}, \fm_N^{(2)}, \fm_N^{(3)},\cdots $ by $F_N(\bmp;\theta)$, $g_N^{(1)}(\bmp;\theta), g_N^{(2)}(\bmp;\theta), g_N^{(3)}(\bmp;\theta),\cdots$ respectively.
For $t\geq 1$, we define the coefficients $\fp_N^{(t)}(\bmla,\bmmu)$, for $\bmla,\bmmu\in \Y(N)$ via the decomposition 
\begin{align}\label{e:deffp}
g_N^{(t)}(x_1,x_2,\cdots, x_N;\theta)\frac{J_\bmla(x_1,x_2,\cdots, x_N;\theta)}{J_\bmla(1^N;\theta)}
=\sum_{\bmmu\in \Y(N)}\fp_N^{(t)}(\bmla,\bmmu)
\frac{J_\bmmu(x_1,x_2,\cdots, x_N;\theta)}{J_\bmmu(1^N;\theta)}.
\end{align}
By the definition \eqref{e:deffp} of the coefficients $\fp_N^{(t)}(\bmla,\bmmu)$, we have for $\bmla,\bmmu\in \bY(N)$ and $t\geq1$,
\begin{align}\label{e:deffpinfinite}
g_N^{(t)}(\bmp;\theta)\frac{J_\bmla(\bmp;\theta)}{J_\bmla(1^N;\theta)}
=\sum_{\bmmu\in \Y(N)}\fp_N^{(t)}(\bmla,\bmmu)
\frac{J_\bmmu(\bmp;\theta)}{J_\bmmu(1^N;\theta)}+\Sp(J_\bmeta(\bmp;\theta): \bmeta\not\in \bY(N)),
\end{align}
where $\Sp(J_\bmeta(\bmp;\theta): \bmeta\not\in \bY(N))$  is the linear span of Jack symmetric functions $J_\bmeta(\bmp;\theta)$ with $\bmeta\not\in \bY(N)$ in $\Sym$. We define a Markov chain $\bmla^{(0)}, \bmla^{(1)},\bmla^{(2)},\cdots \in \bY(N)$, with initial distribution
\begin{align*}
\bP(\bmla^{(0)}=\bmla)=M_N(\bmla),
\end{align*} 
and transition probability,
\begin{align}\label{e:chain}
\bP(\bmla^{(t)}=\bmmu|\bmla^{(t-1)}=\bmla)=\fp_N^{(t)}(\bmla, \bmmu),\quad t=1,2,3,\cdots.
\end{align}
%
%
%
We define a sequence of functions by
\begin{align*}
H_N^{(t)}(\bmp;\theta)\deq F_N(\bmp;\theta)\prod_{s=1}^t g_N^{(s)}(\bmp;\theta), \quad t=0,1,2,\cdots.
\end{align*}
Thanks to \eqref{e:deffpinfinite}, we have 
\begin{align*}\begin{split}
&\phantom{{}={}}\sum_{\bmla\in \bY(N)}\bP(\bmla^{(t)}=\bmla)\frac{J_\bmla(\bmp;\theta)}{J_{\bmla}(1^N;\theta)}
=\sum_{\bmla,\bmmu\in \bY(N)}\bP(\bmla^{(t-1)}=\bmmu)\fp_{N}^{(t)}(\bmmu,\bmla)\frac{J_\bmla(\bmp;\theta)}{J_{\bmla}(1^N;\theta)}\\
&=g_N^{(t)}(\bmp;\theta)\sum_{\bmmu\in \bY(N)}\bP(\bmla^{(t-1)}=\bmmu)\frac{J_\bmmu(\bmp;\theta)}{J_\bmmu(1^N;\theta)}+\Sp(J_\bmeta(\bmp;\theta):\bmeta\notin \bY(N))
\end{split}\end{align*}
We notice that the linear subspace $\Sp(J_\bmla(\bmp;\theta): \bmla\not\in \bY(N))$ is an ideal of $\Sym$, the algebra of symmetric functions in infinitely many variables. By repeating the above procedure, we get
\begin{align*}
\sum_{\bmla\in \bY(N)}\bP(\bmla^{(t)}=\bmla)\frac{J_\bmla(\bmp;\theta)}{J_{\bmla}(1^N;\theta)}=H_N^{(t)}(\bmp;\theta)+\Sp(J_\bmla(\bmp;\theta): \bmla\not\in \bY(N)).
\end{align*}
The Jack generating function of $\bP(\bmla^{(t)}=\bmla)$ and $H_N^{(t)}(\bmp;\theta)$ differ by an element in $\Sp(J_\bmla(\bmp;\theta): \bmla\not\in \bY(N))$. In the following, we will call $H_N^{(t)}(\bmp;\theta)$ Jack generating functions.

%
%
%

\begin{theorem}\label{t:mulclt}
Let $\{M_N\}_{N\geq 1}$, $\{\fm_N^{(1)}\}_{N\geq 1}, \{\fm_N^{(2)}\}_{N\geq 1}, \{\fm_N^{(3)}\}_{N\geq 1},\cdots$ be a sequence of probability measures on $\bY(N)$, with Jack generating functions $F_N(\bmp;\theta)$,$g_N^{(1)}(\bmp;\theta), g_N^{(2)}(\bmp;\theta), g_N^{(3)}(\bmp;\theta), \cdots $. We assume that for any time $\tau\geq 0$, the Jack generating functions $H^{(\lfloor N\tau\rfloor)}_N(\bmp;\theta)$ satisfy Assumption \ref{asup:infinite} and \ref{a:CLT}, we denote the corresponding limit functions \eqref{eqn:UV} by $U_\mu^{(\tau)}(z)$ and $V_\mu^{(\tau)}(z,w)$. Then, the collection of random variables
\begin{align*}
\xi^{(k)}[\bmla^{(\lfloor N\tau\rfloor)}]\deq N\left(\int x^k\rd \mu[\bmla^{(\lfloor N\tau\rfloor)}]-\bE\int x^k\rd \mu[\bmla^{(\lfloor N\tau\rfloor)}]\right),  k=1,2,3,\cdots,\quad \tau\geq 0,
\end{align*} 
converges in the sense of moments, as $N\rightarrow\infty$, to the Gaussian vector with zero mean and covariance
\begin{align}\begin{split}\label{e:mulcltcov}
\lim_{N\rightarrow\infty}\cov(\xi^{(k)}[\bmla^{(\lfloor N\tau\rfloor)}], \xi^{(l)}[\bmla^{(\lfloor N\sigma\rfloor)}])
&=[z^{-1}w^{-1}]
\left(\sum_{a\geq 1}\frac{a}{\theta}z^{a-1}w^{-a-1} +\frac{V^{(\tau\wedge \sigma)}_\mu(z+1,w+1)}{\theta^2}\right)\\
&\times\left( \frac{1}{z} +\frac{ (1 + z)}{\theta}U^{(\tau)}_{\mu}(1 + z) \right)^{k}
\left( \frac{1}{w} +\frac{ (1 + w)}{\theta}U^{(\sigma)}_{\mu}(1 + w) \right)^{l}.
\end{split}\end{align}
\end{theorem}

\subsection{Necessary conditions for central limit theorems}

The following theorems invert Theorem \ref{t:clt}. Roughly speaking, we show
that if the central limit theorems hold for the random measures $\{M_N\}_{N\geq 1}$, then the partial derivatives of the logarithm of their Jack generating functions at unity converge.


\begin{theorem}\label{t:invclt}
Let $\{M_N\}_{N\geq 1}$ be a sequence of probability measures on $\bY(N)$, with Jack generating function $F_N(\bmp;\theta)$ as defined in \eqref{e:gener}.
We assume the  Jack generating functions $F_N(\bmp;\theta)$ satisfy Assumption \ref{asup:infinite},
there exists a collection of numbers $\fp_1,\fp_2,\fp_3,\cdots$, the collection of random variables 
\begin{align*}
\int x^k \rd \mu[\bmla],\quad k=1,2,3,\cdots,
\end{align*}
converges, as $N\rightarrow \infty$, in the sense of moments, 
\begin{align*}
\lim_{N\rightarrow \infty}\bE \left[ \int x^{k} \rd \mu[\bmla]\right] = \fp_{k},
\end{align*}
and the collection of random variables
\begin{align*}
\xi^{(k)}[\bmla]\deq N\left(\int x^k\rd \mu[\bmla]-\bE\int x^k\rd \mu[\bmla]\right),\quad k=1,2,3,\cdots,
\end{align*} 
converges, as $N\rightarrow\infty$, in the sense of moments, to the Gaussian vector with zero mean, variance
\begin{align*}
\lim_{N\rightarrow \infty}\cov(\xi^{(k)}[\bmla],\xi^{(l)}[\bmla])=\fcov_{k,l},
\end{align*}
and for any $r\geq 3$, the $r$-th cumulants satisfy
\begin{align*}
\lim_{N\rightarrow \infty}\kappa_{k_1,k_2,\cdots,k_r}(\xi^{(k)}[\bmla],k=1,2,3,\cdots)=0.
\end{align*}
Then the Jack generating functions $F_N(\bmp;\theta)$ satisfy Assumption \ref{a:CLT}, with $\{\fc_{k}\}_{k\geq 1}$ and $\{\fd_{k,l}\}_{k,l\geq 1}$ characterized by
\begin{align}\label{e:invllnc}
\fp_k =[w^{-1}]\frac{1}{(k+1)}\left(\frac{(w+1)}{\theta}U_{\mu}(w+1)+\frac{1}{w}\right)^{k+1}\sum_{a\geq 0}(-w)^a,
\end{align}
for any $k\geq 1$, and 
\begin{align}\begin{split}\label{e:invcltcov}
\fcov_{k,l}
&=[z^{-1}w^{-1}]
\left(\sum_{a\geq 1}\frac{a}{\theta}z^{a-1}w^{-a-1} +\frac{V_\mu(z+1,w+1)}{\theta^2}\right)\\
&\times\left( \frac{1}{z} +\frac{ (1 + z)}{\theta}U_{\mu}(1 + z) \right)^{k}
\left( \frac{1}{w} +\frac{ (1 + w)}{\theta}U_{\mu}(1 + w) \right)^{l},
\end{split}\end{align}
for any $k,l\geq 1$, where the formal power series $U_\mu(z)$ and $V_{\mu}(z,w)$ are defined in \eqref{eqn:UV}.

\end{theorem}

\section{Nazarov-Sklyanin Operators}\label{s:NSop}

In this section we recall the Nazarov-Sklyanin operators from \cite[Section 6]{MR3141546}. We denote $p_k$ and $p_k^*$ the operators,
\begin{align*}
p_k=p_k,\quad p_k^*=\frac{k\del}{\theta \del p_k},\quad k\geq 1.
\end{align*}
We introduce the infinite matrix $L$,
\begin{align}\label{e:Lmat}
L=\left[
\begin{array}{cccccc}
0 & p_1 & p_2 & p_3 & p_4 & \cdots\\
p_1^* & 1(\theta^{-1}-1) & p_1 & p_2 & p_3  & \cdots\\
p_2^*  & p_1^* & 2(\theta^{-1}-1) & p_1 & p_2   & \cdots\\
p_3^* & p_2^*  & p_1^* & 3(\theta^{-1}-1) & p_1   & \cdots\\
p_4^* & p_3^* & p_2^*  & p_1^* & 4(\theta^{-1}-1)   & \cdots\\
\vdots & \vdots & \vdots & \vdots & \vdots & \ddots \\
\end{array}
\right],
\end{align}
where the rows and columns are labelled by the indices $i,j=0,1,2,\cdots$. It is convenient to set $p_{-k}=p_k^*$ for $k\geq 1$. We also use the convention $p_0=0$. Then 
\begin{align*}
L=[L_{ij}]_{i,j=0}^{\infty},\quad L_{ij}=j(\theta^{-1}-1)\delta_{ij}+p_{j-i}.
\end{align*}
\begin{remark}
For our convenience, the matrix $L$ in \eqref{e:Lmat} is slightly different from the \emph{Lax matrix} defined in \cite[Section 6]{MR3141546}, which is the submatrix $[L]_{i,j=1}^{\infty}$.
\end{remark}

The following is the main theorem of \cite[Theorem 2]{MR3141546}
\begin{theorem}\label{t:mainT}
We define the operator 
\begin{align*}
I(u;\theta)=(u-L)^{-1}_{00},
\end{align*}
where the inverse to the infinite matrix $u-L$ is regarded as a formal power series in $1/u$:
\begin{align}\label{e:inv}
(u-L)_{00}^{-1}=\frac{1}{u}+\frac{I^{(1)}}{u^2}+\frac{I^{(2)}}{u^3}+\cdots,\quad I^{(k)}=(L^k)_{00}.
\end{align}
The Jack symmetric functions $J_{\bmla}(\bmp;\theta)$ are eigenfunctions of $I(u;\theta)$,
\begin{align}\label{e:IJ}
I(u;\theta)J_{\bmla}(\bmp;\theta) =\frac{1}{u+\ell(\bmla)} \prod_{i=1}^{ \ell(\bmla)}\frac{u+i-\la_i/\theta}{u+i-1-\la_i/\theta}J_\bmla(\bmp;\theta).
\end{align}
\end{theorem}

We note that the operator $I^{(k)}$ is an infinite sum of operators acting on $\bR[p_1,p_2,\cdots]$. However, only finite number of the summands do not vanish on any subspace of $\bR[p_1,p_2,\cdots]$ of a fixed degree in $x_1,x_2,\cdots$. This can be easily seen by rewriting the operator $I^{(k)}$ explicitly as
\begin{align}\label{e:Ik}
I^{(k)}=\sum_{i_1,i_2,\cdots, i_{k-1}\geq 0}L_{0i_1}L_{i_1i_2}\cdots L_{i_{k-2}i_{k-1}}L_{i_{k-1}0}.
\end{align} 

%

\subsection{Perelomov-Popov Measures}\label{s:PPM}

For any $\bmla\in \bY(N)$,  we have $\ell(\bmla)=N$. We define the empirical density and the Perelomov-Popov measure respectively
\begin{align}\label{def:PP}
\mu[\bmla]=\frac{1}{N}\sum_{i=1}^N\delta_{y_i},\quad \mu_{PP}[\bmla]=\frac{1}{N}\sum_{i=1}^N\prod_{j:j\neq i}\frac{y_i-y_j+1/N}{y_i-y_j}\delta_{y_i},
\end{align}
where
\begin{align*}
y_i=\frac{\la_i}{\theta N}-\frac{i-1}{N},\quad i=1,2,\cdots, N.
\end{align*}
The above definition of Perelomov-Popov measure first appears in \cite[Section 1.4]{MR3361772},  which originates from the work of Perelomov and
Popov \cite{MR0236308} on the centers of universal enveloping algebras of classical Lie groups. The empirical density and the Perelomov-Popov measure are closely related. In the rest of this section, we prove that the the law of large numbers and central limit theorems of the empirical density, in the sense of Definition \ref{def:LLNCLT}, is equivalent to those of the Perelomov-Popov measures.

\begin{lemma}\label{l:STPP}
The Stieltjes transform of $\mu_{PP}[\bmla]$ is given by
\begin{align}\label{e:STPP}
\int \frac{\rd \mu_{PP}[\bmla](x)}{z-x} 
=\prod_{i=1}^N\left(1+\frac{1}{N}\frac{1}{z-y_i}\right)-1.
\end{align}
\end{lemma}
\begin{proof}
The Stieltjes transform of $\mu_{PP}[\bmla]$ is given by
\begin{align*}\begin{split}
\int \frac{\rd \mu_{PP}[\bmla](x)}{z-x} 
&=\frac{1}{N}\sum_{i=1}^N\prod_{j:j\neq i}\frac{y_i-y_j+1/N}{y_i-y_j}\frac{1}{z-y_i}\\
&=\sum_{i=1}^N\prod_{j:j\neq i}\frac{y_i-y_j+1/N}{y_i-y_j}\left(1-\frac{y_i-z+1/N}{y_i-z}\right)\\
&=\prod_{i=1}^N\frac{z-y_i+1/N}{z-y_i}-1.
\end{split}\end{align*}

\end{proof}

Lemma \ref{l:STPP} connects the empirical density $\mu[\bmla]$ with its Perelomov-Popov measure $\mu_{PP}[\bmla]$. It is used in the next lemma to derive the relations between their moments.
\begin{lemma}\label{l:change}
We denote the moments of $\mu[\bmla]$ and $\mu_{PP}[\bmla]$,
\begin{align*}
c^{(k)}[\bmla]=\int x^k \rd \mu[\bmla](x),\quad
c^{(k)}_{PP}[\bmla]=\int x^k \rd \mu_{PP}[\bmla](x),\quad k=0,1,2,3,\cdots.
\end{align*}
Fix any $k\geq 1$. There exists $k+1$ multivariate polynomials $P_0,P_1,\cdots, P_{k} $, 
\begin{align}\label{e:tock}
c^{(k)}_{PP}[\bmla]=\sum_{r=0}^{k}\frac{1}{N^{r}}P_r(c^{(1)}[\bmla], c^{(2)}[\bmla],\cdots, c^{(k-r)}[\bmla]).
\end{align}
There exists a multivariate polynomial $Q$,
\begin{align}\label{e:tosk}
c^{(k)}[\bmla]=Q(c^{(1)}_{PP}[\bmla], c^{(2)}_{PP}[\bmla],\cdots, c^{(k)}_{PP}[\bmla])+\sum_{r=1}^{k}\frac{(-1)^{r+1}}{r+1}{k+1 \choose r}\frac{c^{(k-r)}[\bmla]}{N^{r}}.
\end{align}
\end{lemma}

\begin{proof}
For the proof of \eqref{e:tock}, we expand both sides of \eqref{e:STPP} as a formal power series in $1/z$. The lefthand side of \eqref{e:STPP} is the moment generating function of $\rd \mu_{PP}[\bmla]$
\begin{align}\label{e:left1}
\int \frac{\rd \mu_{PP}[\bmla](x)}{z-x} 
=\frac{1}{z}+\frac{c^{(1)}_{PP}[\bmla]}{z^2}+\frac{c^{(2)}_{PP}[\bmla]}{z^3}+\frac{c^{(3)}_{PP}[\bmla]}{z^4}\cdots.
\end{align}
For the righthand side of \eqref{e:STPP}, we have
\begin{align}\begin{split}\label{e:right1}
\prod_{i=1}^N\left(1+\frac{1}{N}\frac{1}{z-y_i}\right)-1
&=\prod_{i=1}^N\left(1+\frac{1}{Nz}+\frac{y_i}{Nz^2}+\frac{y_i^2}{Nz^3}+\cdots\right)-1\\
&=\sum_{r\geq 1}\sum_{s_1, s_2,\cdots,s_r\geq 0}\sum_{1\leq i_1<i_2<\cdots <i_r\leq N}\frac{y_{i_1}^{s_1}y_{i_2}^{s_2}\cdots y_{i_r}^{s_r}}{N^r z^{s_1+s_2+\cdots +s_r+r}}.
\end{split}\end{align}
By comparing the coefficient of $1/z^{k+1}$ in \eqref{e:left1} and \eqref{e:right1}, we get
\begin{align*}
c^{(k)}_{PP}[\bmla]=\sum_{r\geq 1}\sum_{s_1+s_2+\cdots+s_r=k+1-r,\atop s_1, s_2,\cdots, s_r\geq 0}\sum_{1\leq i_1<i_2<\cdots<i_r\leq N}\frac{y_{i_1}^{s_1}y_{i_2}^{s_2}\cdots y_{i_r}^{s_r}}{N^r}.
\end{align*}
The claim \eqref{e:tock} follows from rewriting the monomial symmetric polynomials in terms of the power-sum symmetric polynomials.

For the proof of \eqref{e:tosk}, we rewrite \eqref{e:STPP} as
\begin{align}\label{e:logSTPP}
\ln\left(1+\int\frac{\rd \mu_{PP}[\bmla](x)}{z-x}\right)
=\sum_{i=1}^N\ln\left(1+\frac{1}{N}\frac{1}{z-y_i}\right).
\end{align}
Similarly we expand both sides of \eqref{e:logSTPP} as formal power series in $1/z$. For the lefthand side of \eqref{e:logSTPP}, we have
\begin{align}\label{e:left2}\begin{split}
\ln\left(1+\int\frac{\rd \mu_{PP}[\bmla](x)}{z-x}\right)
&=\ln\left(1+\frac{1}{z}+\frac{c^{(1)}_{PP}[\bmla]}{z^2}+\frac{c^{(2)}_{PP}[\bmla]}{z^3}+\frac{c^{(3)}_{PP}[\bmla]}{z^4}\cdots\right)\\
&=-\sum_{r\geq 1}\frac{(-1)^{r}}{r}\left(\frac{1}{z}+\frac{c^{(1)}_{PP}[\bmla]}{z^2}+\frac{c^{(2)}_{PP}[\bmla]}{z^3}+\frac{c^{(3)}_{PP}[\bmla]}{z^4}\cdots\right)^r\\
&=-\sum_{r\geq 1}\frac{(-1)^{r}}{r}\sum_{i_1,i_2,\cdots,i_r\geq 0}\frac{c_{PP}^{(i_1)}[\bmla]c_{PP}^{(i_2)}[\bmla]\cdots c_{PP}^{(i_r)}[\bmla]}{z^{i_1+i_2+\cdots+i_r+r}},
\end{split}\end{align}
where we make the convention $c_{PP}^{(0)}=1$. For the righthand side of \eqref{e:logSTPP}, we have
\begin{align}\label{e:right2}\begin{split}
&\phantom{{}={}}\sum_{i=1}^N\ln\left(1+\frac{1}{N}\frac{1}{z-y_i}\right)
=\sum_{i=1}^N\ln\left(1+\frac{1}{Nz}+\frac{y_i}{Nz^2}+\frac{y_i^2}{Nz^3}+\cdots\right)\\
&=\left(\frac{1}{z}+\frac{c^{(1)}[\bmla]}{z^2}+\frac{c^{(2)}[\bmla]}{z^3}+\cdots\right)-\sum_{r\geq 2}\frac{(-1)^r}{r}\sum_{i=1}^N\left(\frac{1}{Nz}+\frac{y_i}{Nz^2}+\frac{y_i^2}{Nz^3}+\cdots\right)^r\\
&=\left(\frac{1}{z}+\frac{c^{(1)}[\bmla]}{z^2}+\frac{c^{(2)}[\bmla]}{z^3}+\cdots\right)-\sum_{r\geq 2}\frac{(-1)^r}{r}\sum_{i_1,i_2,\cdots, i_r\geq 0}\frac{c^{(i_1+i_2+\cdots+i_r)}[\bmla]}{N^{r-1}z^{i_1+i_2+\cdots+i_r+r}}.
\end{split}\end{align}
By comparing the coefficient of $1/z^{k+1}$ in \eqref{e:left2} and \eqref{e:right2}, we get
\begin{align*}
c^{(k)}[\bmla]=-\sum_{r\geq 1}\frac{(-1)^r}{r}\sum_{i_1+i_2+\cdots+i_r=k+1-r,\atop 
i_1,i_2,\cdots, i_r\geq 0}c^{(i_1)}_{PP}[\bmla]c^{(i_2)}_{PP}[\bmla]\cdots c^{(i_r)}_{PP}[\bmla]+\sum_{r\geq 1}\frac{(-1)^{r+1}}{r+1}{k+1 \choose r}\frac{c^{(k-r)}[\bmla]}{N^{r}},
\end{align*}
and the claim \eqref{e:tosk} follows.
\end{proof}

In terms of moments, the empirical density $\mu[\bmla]$ and Perelomov-Popov measure $\mu_{PP}[\bmla]$ are equivalent, i.e. one uniquely determine the other. As an easy consequence, $\mu[\bmla]$ satisfies the law of large numbers or the central limit theorems, in the sense of Definition \ref{def:LLNCLT}, if and only if $\mu_{PP}[\bmla]$ satisfies the law of large numbers or the central limit theorems.

\begin{lemma}\label{l:CLTtoCLT}
In the sense of Definition \ref{def:LLNCLT}$,  \mu[\bmla]$ satisfies the law of large numbers if any only if $\mu_{PP}[\bmla]$ satisfies the law of large numbers; $\mu[\bmla]$ satisfies the central limit theorems if and only if $\mu_{PP}[\bmla]$ satisfies the central limit theorems.
\end{lemma}
\begin{proof}
We denote the moments of $\mu[\bmla]$ and $\mu_{PP}[\bmla]$
\begin{align*}
c^{(k)}[\bmla]=\int x^k \rd \mu[\bmla](x),\quad
c^{(k)}_{PP}[\bmla]=\int x^k \rd \mu_{PP}[\bmla](x),\quad k=0,1,2,3,\cdots.
\end{align*}

If $\mu[\bmla]$ satisfies the law of large numbers, in the sense of Definition \ref{def:LLNCLT}, there exists  a collection of numbers $\fc^{(1)}, \fc^{(2)},\fc^{(3)}, \cdots,$ in the sense of moments
\begin{align*}
\lim_{N\rightarrow \infty} c^{(i)}[\bmla]=\fc^{(i)}, \quad i=1,2,3,\cdots.
\end{align*}
Thanks to \eqref{e:tock}, we have
\begin{align}\begin{split}\label{e:cklaw}
\lim_{N\rightarrow \infty}c^{(k)}_{PP}[\bmla]
&=\lim_{N\rightarrow \infty}\left(\sum_{r=0}^k\frac{1}{N^{r}}P_r(c^{(1)}[\bmla], c^{(2)}[\bmla],\cdots, c^{(k-r)}[\bmla])\right)=P_0(\fc^{(1)}, \fc^{(2)},\cdots, \fc^{(k)}),
\end{split}\end{align}
in the sense of moments, and thus $\mu_{PP}[\bmla]$ satisfies the law of large numbers.

If $\mu_{PP}[\bmla]$ satisfies the law of large numbers in the sense of Definition \ref{def:LLNCLT}, there exists  a collection of numbers $\fc_{PP}^{(1)}, \fc_{PP}^{(2)},\fc^{(3)}_{PP}, \cdots,$ in the sense of moments
\begin{align*}
\lim_{N\rightarrow \infty} c_{PP}^{(i)}[\bmla]=\fc_{PP}^{(i)}, \quad i=1,2,3,\cdots.
\end{align*}
We assume that for any $1\leq l\leq k-1$, there exists $\fc^{(l)}$ such that in the sense of moments
\begin{align*}
\lim_{N\rightarrow \infty} c^{(l)}[\bmla]=\fc^{(l)}.
\end{align*}
Then thanks to \eqref{e:tosk}, we have
\begin{align}\begin{split}\label{e:sklaw}
\lim_{N\rightarrow \infty}c^{(k)}[\bmla]
&=\lim_{N\rightarrow \infty}\left(Q(c^{(1)}_{PP}[\bmla], c^{(2)}_{PP}[\bmla],\cdots, c^{(k)}_{PP}[\bmla])+\sum_{r=1}^{k}\frac{(-1)^{r+1}}{r+1}{k+1 \choose r}\frac{s_{k-r}[\bmla]}{N^{r}}\right)\\
&=Q(\fc^{(1)}_{PP}, \fc^{(2)}_{PP},\cdots, \fc^{(k)}_{PP}),
\end{split}\end{align}
in the sense of moments, and thus $\mu_{PP}[\bmla]$ satisfies the law of large numbers.

If $\mu[\bmla]$ satisfies the central limit theorems in the sense of Definition \ref{def:LLNCLT}, then it also satisfies the law of large numbers and from the discussion above, $\mu_{PP}[\bmla]$ satisfies the law of large numbers. Moreover, for any $k\geq 1$, 
\begin{align}\label{e:jointc}
N(c^{(1)}[\bmla]-\bE[c^{(1)}[\bmla]]), N(c^{(2)}[\bmla]-\bE[c^{(2)}[\bmla]]),\cdots, N(c^{(k)}[\bmla]-\bE[c^{(k)}[\bmla]]),
\end{align}
are asymptotically Gaussian as $N\rightarrow \infty$. From \eqref{e:tock}, we have
\begin{align}\begin{split}\label{e:Nsk2}
Nc^{(k)}_{PP}[\bmla]
&=\sum_{i=1}^kN(c^{(i)}[\bmla]-\bE[c^{(i)}[\bmla]])\del_{i}P_0(\bE[c^{(1)}[\bmla]], \bE[c^{(2)}[\bmla]],\cdots, \bE[c^{(k)}[\bmla]])\\
&+NP_0(\bE[c^{(1)}[\bmla]], \bE[c^{(2)}[\bmla]],\cdots, \bE[c^{(k)}[\bmla]])+P_1(\fc^{(1)},\fc^{(2)},\cdots,\fc^{(k-1)})+o_{\bP}(1),
\end{split}\end{align}
Then, by subtracting mean from both sides of \eqref{e:Nsk2},  we conclude that $N(c^{(k)}_{PP}[\bmla]-\bE[c^{(k)}_{PP}[\bmla]])$ is a linear combination of \eqref{e:jointc}
\begin{align*}\begin{split}
N(c^{(k)}_{PP}[\bmla]-\bE[c^{(k)}_{PP}[\bmla]])
&=\sum_{i=1}^kN(c^{(i)}[\bmla]-\bE[c^{(i)}[\bmla]])\del_{i}P_0(\bE[c^{(1)}[\bmla]], \bE[c^{(2)}[\bmla]],\cdots, \bE[c^{(k)}[\bmla]])+o_{\bP}(1).
\end{split}\end{align*}
Thus $\mu[\bmla]$ satisfies the central limit theorems.

If $\mu_{PP}[\bmla]$ satisfies the central limit theorems in the sense of Definition \ref{def:LLNCLT}, then it also satisfies the law of large numbers and from the discussion above, $\mu[\bmla]$ satisfies the law of large numbers. Moreover, for any $k\geq 1$, 
\begin{align}\label{e:jointcpp}
N(c^{(1)}_{PP}[\bmla]-\bE[c^{(1)}_{PP}[\bmla]]), N(c^{(2)}_{PP}[\bmla]-\bE[c^{(2)}_{PP}[\bmla]]),\cdots, N(c^{(k)}_{PP}[\bmla]-\bE[c^{(k)}_{PP}[\bmla]]),
\end{align}
are asymptotically Gaussian as $N\rightarrow \infty$. 
From \eqref{e:tosk}, we have 
\begin{align}\begin{split}\label{e:Nsk}
Nc^{(k)}[\bmla]
&=\sum_{i=1}^kN(c^{(i)}_{PP}[\bmla]-\bE[c^{(i)}_{PP}[\bmla]])\del_{i}Q(\bE[c^{(1)}_{PP}[\bmla]], \bE[c^{(2)}_{PP}[\bmla]],\cdots, \bE[c^{(k)}_{PP}[\bmla]])\\
&+NQ(\bE[c^{(1)}_{PP}[\bmla]], \bE[c^{(2)}_{PP}[\bmla]],\cdots, \bE[c^{(k)}_{PP}[\bmla]])+\frac{1}{2}(k+1)\fc^{(k-1)}+o_{\bP}(1).
\end{split}\end{align}
Then, by subtracting mean from both sides of \eqref{e:Nsk},  we conclude that $N(c^{(k)}[\bmla]-\bE[c^{(k)}[\bmla]])$ is a linear combination of \eqref{e:jointcpp}
\begin{align*}
N(c^{(k)}[\bmla]-\bE[c^{(k)}[\bmla]])
&=\sum_{i=1}^kN(c^{(i)}_{PP}[\bmla]-\bE[c^{(i)}_{PP}[\bmla]])\del_{i}Q(\bE[c^{(1)}_{PP}[\bmla]], \bE[c^{(2)}_{PP}[\bmla]],\cdots, \bE[c^{(k)}_{PP}[\bmla]])+o_{\bP}(1).
\end{align*}
Thus $\mu_{PP}[\bmla]$ satisfies the central limit theorems. 
 
%

\end{proof}

\begin{remark}\label{r:change}
Due to the nature of the equation \eqref{e:STPP}, it is in fact easier to see the relation between $\mu[\bmla]$ and $\mu_{PP}[\bmla]$ using Stieltjes transform, viewed as formal power series in $1/z$. If both $\mu[\bmla]$ and $\mu_{PP}[\bmla]$ satisfy the law of large numbers, then the coefficients of their Stieltjes transform viewed as formal power series in $1/z$ are stochastically bounded. We can rewrite \eqref{e:STPP} as
\begin{align}\label{e:mppst}\begin{split}
\int \frac{\rd \mu_{PP}[\bmla](x)}{z-x}
&=e^{\sum_{i=1}^N \ln\left(1+\frac{1}{N}\frac{1}{z-y_i}\right)}-1\\
&=e^{\frac{1}{N}\sum_{i=1}^N \frac{1}{z-y_i}+\OO_{\bP}\left(\frac{1}{N}\right)}-1
=e^{\int \frac{\rd \mu[\bmla](x)}{z-x}+\OO_{\bP}\left(\frac{1}{N}\right)}-1.
\end{split}\end{align}
$\mu[\bmla]$ converges weakly to $\mu$, if and only if $\mu_{PP}[\bmla]$ converges weakly to $\mu_{PP}$. We denote the Stieltjes transforms of $\mu$ and $\mu_{PP}$ by $m_\mu(z)$ and $m_{\mu_{PP}}(z)$ respectively, and they  satisfy the relation,
\begin{align}\label{e:limitlaw}
m_{\mu_{PP}}(z)=e^{m_\mu(z)}-1.
\end{align}
If we further assume that $\mu[\bmla]$ and $\mu_{PP}[\bmla]$ satisfy the central limit theorems, we denote 
\begin{align*}
\Delta m_{PP}[\bmla](z)=N\int \frac{\rd \mu_{PP}[\bmla](x)-\bE\rd \mu_{PP}[\bmla](x)}{z-x},\quad 
\Delta m[\bmla](z)=N\int \frac{\rd \mu[\bmla](x)-\bE\rd \mu[\bmla](x)}{z-x}.
\end{align*}
The coefficients of $\Delta m_{PP}[\bmla](z)$ and $\Delta m[\bmla](z)$ viewed as formal power series in $1/z$ are stochastically bounded.
We expand \eqref{e:STPP} to a higher order
\begin{align}\begin{split}\label{e:higherorder}
\int \frac{\rd \mu_{PP}[\bmla](x)}{z-x}
&=e^{\frac{1}{N}\sum_{i=1}^N \frac{1}{z-y_i}-\frac{1}{2N^2}\sum_{i=1}^N \frac{1}{(z-y_i)^2}+\OO_{\bP}\left(\frac{1}{N^2}\right)}-1\\
&=\exp\left(\bE\left[\int \frac{\rd \mu[\bmla](x)}{z-x}\right]\right)e^{\frac{1}{N}\Delta m[\bmla](z)+\frac{1}{2N}\del_z m_\mu(z)+\oo_{\bP}\left(\frac{1}{N}\right)}-1\\
&=\exp\left(\bE\left[\int \frac{\rd \mu[\bmla](x)}{z-x}\right]\right)-1+\oo_{\bP}\left(\frac{1}{N}\right).
\end{split}\end{align}
And it follows by taking expectation on both sides of \eqref{e:higherorder},
\begin{align}\label{e:hexpect}
\bE\left[\int \frac{\rd \mu_{PP}[\bmla](x)}{z-x}\right]=\exp\left(\bE\left[\int \frac{\rd \mu[\bmla](x)}{z-x}\right]\right)-1+\oo\left(\frac{1}{N}\right).
\end{align}
By taking difference of \eqref{e:higherorder} and \eqref{e:hexpect}, we get
\begin{align*}\begin{split}
\frac{1}{N}\Delta m_{PP}[\bmla](z)
&=\exp\left(\bE\left[\int \frac{\rd \mu[\bmla](x)}{z-x}\right]\right)\left(e^{\frac{1}{N}\Delta m[\bmla](z)+\frac{1}{2N}\del_zm_\mu(z)+o_{\bP}\left(\frac{1}{N}\right)}-1\right).
\end{split}\end{align*}
We can rearrange the above expression, 
\begin{align*}
\Delta m_{PP}[\bmla](z)=e^{m_\mu(z)}\Delta m[\bmla](z)+\frac{1}{2}\del_z m_\mu(z)e^{m_\mu(z)}+o_{\bP}(1),
\end{align*}
and 
\begin{align}\label{e:Deltam}
\cov\langle\Delta m[\bmla](z_1),\Delta m[\bmla](z_2)\rangle=e^{-m_\mu(z_1)}e^{-m_\mu(z_2)}\cov\langle\Delta m_{PP}[\bmla](z_1),\Delta m_{PP}[\bmla](z_2)\rangle+o(1).
\end{align}
The relation \eqref{e:Deltam} will be used later to derive the covariance structures of the moments of  $\mu[\bmla]$ from those of $\mu_{PP}[\bmla]$.
\end{remark}

\subsection{Eigenvalues of the Nazarov-Sklyanin operators}
\label{subs:NSOP}
For any $\bmla\in \bY(N)$, we recall that 
\begin{align*}
y_i=\frac{\la_i}{\theta N}-\frac{i-1}{N},\quad i=1,2,\cdots, N.
\end{align*}
We can rewrite \eqref{e:IJ} as 
\begin{align}\begin{split}\label{e:IJ1}
I(u;\theta)J_{\bmla}(\bmp;\theta) 
=&\frac{1}{u+N} \prod_{i=1}^{ N}\frac{u+i-\la_i/\theta}{u+i-1-\la_i/\theta}J_\bmla(\bmp;\theta)\\
=&\frac{1}{u+N}\prod_{i=1}^N\left(1+\frac{1}{u-Ny_i}\right)J_\bmla(\bmp;\theta).
\end{split}\end{align}
We take $u=Nz$ for some $z\in \bC\setminus \bR$. Using \eqref{e:STPP}, we can rewrite \eqref{e:IJ1} as
\begin{align}\label{e:PPid}
N(z+1)I(Nz;\theta)J_{\bmla}(\bmp;\theta)
=\left(1+\int\frac{\rd \mu_{PP}[\bmla](x)}{z-x}\right) J_\bmla(\bmp;\theta).
\end{align} 
Since the operator $I(Nz;\theta)$ is defined by the formal power series \eqref{e:inv}, the lefthand side of \eqref{e:PPid} is given by
\begin{align*}\begin{split}
&\phantom{{}={}}N(z+1)I(Nz;\theta)J_{\bmla}(\bmp;\theta)
=(z+1)\left(\frac{1}{z}+\frac{I^{(1)}}{Nz^2}+\frac{I^{(2)}}{N^2z^3}+\cdots\right)J_\bmla(\bmp;\theta)\\
&=\left(1+\left(1+\frac{I^{(1)}}{N}\right)\frac{1}{z}+\left(\frac{I^{(1)}}{N}+\frac{I^{(2)}}{N^2}\right)\frac{1}{z^2}
+\cdots\right)J_\bmla(\bmp;\theta).
\end{split}\end{align*}
The righthand side of \eqref{e:PPid} is given by
\begin{align*}\begin{split}
\phantom{{}={}}&\left(1+ \int\frac{\rd \mu_{PP}[\bmla](x)}{z-x}\right) J_\bmla(\bmp;\theta)\\
=&\left(1+\frac{1}{z}\int \rd\mu_{PP}[\bmla](x)+\frac{1}{z^2}\int x\rd \mu_{PP}[\bmla](x)+\frac{1}{z^3}\int x^2\rd \mu_{PP}[\bmla](x) +\cdots\right) J_\bmla(\bmp;\theta).
\end{split}\end{align*}
It follows by comparing both sides of \eqref{e:PPid}, for any $k\geq 1$, we have
\begin{align}\label{t:momentPP}
\left(\frac{I^{(k)}}{N^{k}}+\frac{I^{(k+1)}}{N^{k+1}}\right)J_{\bmla}(\bmp;\theta)
=\left(\int x^{k}\rd \mu_{PP}[\bmla](x) \right)J_{\bmla}(\bmp;\theta),
\end{align}
We can repeatedly apply \eqref{t:momentPP}, and get the following Theorem.
\begin{theorem}\label{t:moment}
The joint moments of the Perelomov-Popov measure $\mu_{PP}[\bmla]$ satisfies: for any $r\geq 1$ and numbers $k_1,k_2,\cdots, k_r\geq 1$, we have
\begin{align*}
\prod_{j=1}^r\left(\frac{I^{(k_j)}}{N^{k_j}}+\frac{I^{(k_j+1)}}{N^{k_j+1}}\right)J_{\bmla}(\bmp;\theta)
=\prod_{j=1}^r\left(\int x^{k_j}\rd \mu_{PP}[\bmla](x) \right)J_{\bmla}(\bmp;\theta).
\end{align*}
Let $\{M_N\}_{N\geq 1}$ be a sequence of probability measures on $\bY(N)$, with Jack generating functions $F_N(\bmp;\theta)$ as defined in \eqref{e:gener}, then 
\begin{align*}
\left.\prod_{j=1}^r\left(\frac{I^{(k_j)}}{N^{k_j}}+\frac{I^{(k_j+1)}}{N^{k_j+1}}\right)F_N(\bmp;\theta)\right|_{\bmp=1^N}=\bE\left[\prod_{j=1}^r\int x^{k_j}\rd \mu_{PP}[\bmla](x) \right].
\end{align*}
\end{theorem}

\begin{remark}
Let $\tilde F_N(\bmp;\theta)$ be a function which differs from $ F_N(\bmp;\theta)$ by an element in the linear span $\Sp(J_\bmla(\bmp;\theta): \bmla\not\in \bY(N))$, i.e. 
\begin{align}\label{e:diffsp}
\tilde F_N(\bmp;\theta)= F_N(\bmp;\theta)+\Sp(J_\bmla(\bmp;\theta): \bmla\not\in \bY(N)).
\end{align}
Since the Nazarov-Sklyanin Operators preserve the subspace $\Sp(J_\bmla(\bmp;\theta): \bmla\not\in \bY(N))$, and $J_{\bmla}(1^N;\theta)=0$ for any $\bmla\not\in \bY(N)$, we have
\begin{align*}
\left.\prod_{j=1}^r\left(\frac{I^{(k_j)}}{N^{k_j}}+\frac{I^{(k_j+1)}}{N^{k_j+1}}\right)F_N(\bmp;\theta)\right|_{\bmp=1^N}
=\left.\prod_{j=1}^r\left(\frac{I^{(k_j)}}{N^{k_j}}+\frac{I^{(k_j+1)}}{N^{k_j+1}}\right)\tilde F_N(\bmp;\theta)\right|_{\bmp=1^N}.
\end{align*}
One can use $\tilde F_N(\bmp;\theta)$ as the Jack generating function of $M_N$, and compute the expectation of the joint moments of the Perelomov-Popov measure $\mu_{PP}[\bmla]$. In most of our applications in Section \ref{s:applications}, there are no closed forms for the Jack generating functions $F_N(\bmp;\theta)$. However, there exist functions $\tilde F_N(\bmp;\theta)$ with easy form such that \eqref{e:diffsp} holds. The law of large numbers and central limit theorems follow from analyzing $\tilde F_N(\bmp;\theta)$ instead.
\end{remark}

%
%
%

\section{Proof of Main Results} \label{s:proof}
In this section we give proofs of our main results in Section \ref{s:main}.
Let $\{M_N\}_{N\geq 1}$ be a sequence of probability measures on $\bY(N)$, with Jack generating functions $F_N(\bmp;\theta)$ as defined in \eqref{e:gener}. Then the expectation of the joint moments of the Perelomov-Popov measure $\mu_{PP}[\bmla]$, as defined \eqref{def:PP}, can be extracted from the Jack generating functions by applying the Nazarov-Sklyanin operators,
\begin{align}\label{e:momentPP}
\left.\prod_{j=1}^r\left(\frac{I^{(k_j)}}{N^{k_j}}+\frac{I^{(k_j+1)}}{N^{k_j+1}}\right)F_N(\bmp;\theta)\right|_{\bmp=1^N}=\bE\left[\prod_{j=1}^r\int x^{k_j}\rd \mu_{PP}[\bmla](x) \right].
\end{align}
We have proven in Section \ref{s:PPM} that the law of large numbers and the central limit theorems of the empirical density $\mu[\bmla]$ are equivalent to those of the Perelomov-Popov measure. The asymptotic questions boil down to analyze the lefthand side of \eqref{e:momentPP}. In Section \ref{s:Tl}, we show that the lefthand side of \eqref{e:momentPP} decomposes into a finite sum of product of terms in the form
\begin{align}\label{e:dpterms}
\sum_{i_1,i_2,\cdots, i_r\geq 1}{i_1\choose k_1}{i_2\choose k_2}\cdots {i_r\choose k_r} \left.\frac{\del^r \ln F_N(\bmp;\theta)}{\del p_{i_1}\del p_{i_2}\cdots \del p_{i_r}}\right|_{\bmp=1^N}.
\end{align}
We remark that Assumptions \ref{a:LLN} and \ref{a:CLT} are the asymptotic behaviors of these terms \eqref{e:dpterms}. The law of large numbers in Section \ref{subs:LLN}, directly follows from the decomposition. In Section \ref{subs:CLT} and \ref{subs:mulCLT}, using \eqref{e:momentPP}, we show higher order cumulants of the moments of the Perelomov-Popov measure \ref{def:PP} vanish asymptotically. Thus, the one level and multi-level central limit theorems follow. In Section \ref{subs:invCLT}, we prove that Assumption \ref{a:LLN} and \ref{a:CLT} are necessary for the law of large numbers and the central limit theorems.

\subsection{Some Notations}\label{s:SN}

For any positive integer $r$, we denote $\qq{r}\deq \{1,2,\cdots,r\}$. Given any set $\Omega$, we denote $\cP(\Omega)$ the set of partitions of the set $\Omega$. 

For a sequence of numbers $\{a_N\}_{N
\geq1}$, we say it has an $N$-degree (at most) $d$, if 
\begin{align*}
\lim_{N\rightarrow \infty}\frac{a_N}{N^d}
\end{align*}
exists.
We say it has an $N$-degree less than $d$, if 
\begin{align*}
\lim_{N\rightarrow \infty}\frac{a_N}{N^d} =0.
\end{align*}

We define the infinite Toeplitz matrix $[T_{ij}]_{i,j=0}^\infty$ such that $T_{ij}=q_{i-j}$, and its symbol
\begin{align}\label{e:defTw}
T(w)=\sum_{i=-\infty}^{\infty}q_i w^{i},
\end{align}
viewed as a formal series of $w$.
There is an easy relation between the inverse of the infinite Toeplitz matrix $[T_{ij}]_{i,j=0}^\infty$ and its symbol $T(w)$.
\begin{lemma}\label{l:invertT}
Let $[T_{ij}]_{i,j=0}^\infty$ be an infinite Toeplitz matrix with symbol $T(w)$ as defined in \eqref{e:defTw}. We assume that the symbol $T(w)$ converges absolutely on $\{w\in \bC:1<|w|<1+\varepsilon\}$, then
\begin{align}\label{e:invertT}
(z-T)^{-1}_{00}=\exp\left(-\frac{1}{2\pi \ri}\int_{|w|=1+\varepsilon/2}\ln(z-T(w))\frac{\rd w}{w}\right),
\end{align}
as a formal power series of $1/z$. 
\end{lemma}
\begin{proof}
The lemma follows from Wiener-Hopf Factorization. Fix any $1<r<1+\varepsilon$, and let $\bT_r=\{w\in \bC:|w|=r\}$. The Hardy space $H^2(\bT_r)$ is a closed subspace of $L^2(\bT_r)$. $L^2(\bT_r)$ consists of bi-infinite sequences indexed by $\bZ$, i.e., $f(w)=\sum_{i\in \bZ}a_i w^i$, and $H^2(\bT_r)$ consists of infinite sequences indexed by $\bZ_{\geq 0}$, i.e. $f(w)=\sum_{i\in \bZ_{\geq 0}} a_i w^i$. The norm is given by $\|f(z)\|_{L^2(\bT_r)}=\sum_{i\in \bZ} |a_i|^2 r^2$. We denote the projection from $L^2(\bT_r)$ to $H^2(\bT_r)$ by $\pi$. The Toeplitz operator $[T_{ij}]_{i,j=0}^\infty$ can 
be identified with the operator $\pi\circ T(w)$ on the Hardy space $H^2(\bT_r)$, where $T(w)$ is the multiplication operator. $\pi\circ T(w)$ is a bounded operator,
\begin{align*}
\|\pi\circ T(w)\|_{H^2(\bT_r)\rightarrow H^2(\bT_r)}
\leq \sup_{w\in \bT_r}|T(w)|^2.
\end{align*} 
Therefore for $z$ large enough, $z-T$ is invertible. $(z-T)^{-1}$ is well-defined as
\begin{align*}
(z-T)^{-1}=\frac{1}{z}+\frac{T}{z^2}+\frac{T^2}{z^3}+\cdots.
\end{align*}
Fix $z$ large enough, we can write $\ln(z-T(w))$ as a power series of $w$,
\begin{align*}
\ln(z-T(w))=\ln(z)-\sum_{i\geq 1}\frac{T(w)^i}{i z^i}\eqd a_0+a_{+}(w)+a_{-}(w),
\end{align*}
where $a_0$ is the constant term,
\begin{align}\label{def:a0}
a_0=[w^0]\ln(z-T(w))=\frac{1}{2\pi \ri}\int_{|w|=1+\varepsilon/2}\ln(z-T(w))\frac{\rd w}{w},
\end{align}
$a_{+}(w)$ collects terms with positive degrees, and $a_{-}(w)$ collects terms with negative degrees. We notice that 
\begin{align*}
\pi\circ e^{a_-(w)}\circ\pi = \pi\circ e^{a_-(w)},\quad \pi\circ e^{-a_-(w)}\circ\pi = \pi\circ e^{-a_-(w)},
\end{align*}
where $e^{a_-(w)}$ and $e^{-a_-(w)}$ are multiplication operators on $L^2(\bT_r)$. With those relations, the inverse of $z-\pi \circ T(w)$ is explicitly given by
\begin{align*}
(z-\pi\circ T(w))^{-1}=(\pi\circ (z-T(w)))^{-1}
=(\pi\circ e^{a_0}e^{a_+(w)}e^{a_-(w)})^{-1}
=e^{-a_0}e^{-a_+(w)}\circ \pi \circ e^{-a_-(w)}.
\end{align*}
Therefore,
\begin{align*}
(z-T)^{-1}_{00}=\langle 1, (z-\pi\circ T(w))^{-1} 1\rangle
=\langle 1, e^{-a_0}e^{-a_+(w)}\circ \pi \circ e^{-a_-(w)}1\rangle
=e^{-a_0},
\end{align*}
and the claim \eqref{e:invertT} follows by combining the above expression with \eqref{def:a0}.
\end{proof}


Fix a Jack generating function $F_N(\bmp;\theta)$, satisfying Assumption \ref{asup:infinite}.
We define for $i\in \bZ_{>0}$,
\begin{align*}
q_i(F_N)=\frac{1}{N}\frac{i}{\theta}\left.\frac{\del\ln F_N(\bmp;\theta)}{\del p_{i}}\right|_{\bmp=1^N},
\end{align*}
for $i=0$, $q_i(F_N)=0$, and for $i\in \bZ_{<0}$, $q_i(F_N)=1$. For $i,j\in \bZ_{>0}$, we define
\begin{align*}
q_{i,j}(F_N)=\left.\frac{\del^2\ln F_N(\bmp;\theta)}{\del p_{i}\del p_j}\right|_{\bmp=1^N}.
\end{align*}

We define the infinite Toeplitz matrix $[(T_{F_N})_{ij}]_{i,j=0}^\infty$ such that $(T_{F_N})_{ij}=q_{i-j}(F_N)$.
The symbol of $[(T_{F_N})_{ij}]_{i,j=1}^\infty$ is given by
 \begin{align}\label{def:TMN}
 T_{F_N}(w)=\frac{w}{\theta}U_{F_N}(w)+\sum_{i\geq 1}\frac{1}{w^i},\quad U_{F_N}(w)=\sum_{i\geq 1}\left(\frac{i}{N}\left.\frac{\del\ln F_N(\bmp;\theta)}{\del p_{i}}\right|_{\bmp=1^N}\right)w^{i-1}.
 \end{align}
 Thanks to Proposition \ref{p:summable}, $U_{F_N}(w)$ converges absolutely on $\{w\in \bC: |w|<1+\varepsilon\}$. Thus we can rewrite it as
 \begin{align*}
 U_{F_N}(w)=\sum_{k\geq 1}\frac{\del_w^{k-1}U_{F_N}(1)}{(k-1)!}(w-1)^{k-1},
 \end{align*}
which converges uniformly on $\{w\in \bC: |w-1|< \varepsilon\}$. We notice that by Assumption \ref{a:LLN}
 \begin{align*}
 \del_w^{k-1}U_{F_N}(1)
 =\sum_{i\geq 1}k!{i\choose k}\left(\frac{1}{N}\left.\frac{\del\ln F_N(\bmp;\theta)}{\del p_{i}}\right|_{\bmp=1^N}\right)\rightarrow \fc_k,
 \end{align*}
as $N\rightarrow \infty$. Thus as formal power series, the coefficients converge
 \begin{align}\begin{split}\label{e:UMNlimit}
  \lim_{N\rightarrow \infty}U_{F_N}(w)&=\sum_{k\geq 1}\frac{c_k}{(k-1)!}(w-1)^{k-1}=U_\mu(w),\\
  \lim_{N\rightarrow \infty}T_{F_N}(w)&=T_\mu(w)\deq \frac{w}{\theta}U_\mu(w)+\sum_{i\geq 1}\frac{1}{w^i},
\end{split} \end{align}
 where $U_\mu(w)$ is defined in \eqref{eqn:Frho}.
 
 We define the function,
 \begin{align*}
 V_{F_N}(z,w)=\sum_{i,j\geq 1}ij q_{i,j}(F_N)z^{i-1}w^{j-1}.
 \end{align*}
 Thanks to Proposition \ref{p:summable}, $V_{F_N}(z,w)$ converges absolutely on $\{(z,w)\in \bC: |z|,|w|<1+\varepsilon\}$. Thus we can rewrite it as
 \begin{align*}
 V_{F_N}(z,w)=\sum_{k,l\geq 1}\frac{\del_z^{k-1}\del_w^{l-1}V_{F_N}(1,1)}{(k-1)!(l-1)!}(z-1)^{k-1}(w-1)^{l-1},
 \end{align*}
 which converges uniformly on $\{(z,w)\in \bC: |z-1|,|w-1|<\varepsilon\}$.
 We notice that by Assumption \ref{a:CLT}
 \begin{align*}
\del_z^{k-1}\del_w^{l-1}V_{F_N}(1,1)
 =\sum_{i\geq 1}k!l!{i\choose k}{j\choose l}q_{i,j}(F_N)\rightarrow \fd_{k,l},
 \end{align*}
as $N\rightarrow \infty$. Thus as a formal power series, we have
 \begin{align}\label{e:VMNlimit}
  \lim_{N\rightarrow \infty}V_{F_N}(z,w)=\sum_{k,l\geq 1}\frac{d_{k,l}}{(k-1)!(l-1)!}(z-1)^{k-1}(w-1)^{l-1}=V_\mu(z,w),
 \end{align}
 where $V_\mu(z,w)$ is defined in \eqref{eqn:UV}.

\subsection{Technical lemmas}\label{s:Tl}
In Section \ref{s:Tl}, we show that the lefthand side of \eqref{e:momentPP} decomposes into a finite sum of product of terms in the form
\begin{align}\label{e:dptermscopy}
\sum_{i_1,i_2,\cdots, i_s\geq 1}{i_1\choose k_1}{i_2\choose k_2}\cdots {i_s\choose k_s} \left.\frac{\del^s \ln F_N(\bmp;\theta)}{\del p_{i_1}\del p_{i_2}\cdots \del p_{i_s}}\right|_{\bmp=1^N},
\end{align}
whose asymptotic behaviors are our Assumptions \ref{a:LLN} and \ref{a:CLT}.

To precisely state the results, we need to introduce some notations. Given a positive integer $r$, and indices $k_1,k_2,\cdots, k_r\geq 1$. We view $j_1,j_2,\cdots,j_{k_1+k_2+\cdots+k_r}$ as formal symbols and define 
\begin{itemize}
\item Let $\cK=(k_1,k_2,\cdots, k_r)$, we define an index function $\sigma_\cK: \qq{k_1+k_2+\cdots+k_r}\mapsto \qq{r}$, $\sigma(t)=s$ where $k_1+k_2+\cdots+k_{s-1}<t\leq k_1+k_2+\cdots+k_{s}$.
\item A sign pattern is an element $\cS=(\cS_1,\cS_2,\cdots, \cS_{k_1+k_2+\cdots+k_r})$ in $\cS(\cK)=\{-,0,+\}^{k_1+k_2+\cdots+k_r}$. 
\item Given a sign pattern $\cS\in \cS(\cK)$, we define $\cD(\cS)$ the set of pairings $\cD=\{(s_1,t_1), (s_2,t_2), \cdots, (s_v,t_v)\}$, such that $|\cD|=v\geq 0$, $1\leq s_u<t_u\leq k_1+k_2+\cdots+k_r$, $\cS_{s_u}=-$, $\cS_{t_u}=+$ for $1\leq u\leq v$, and $s_1, s_2,\cdots, s_v, t_1,t_2,\cdots, t_v$ are distinct.  We also identify $\cD$ with the set $\{s_1, s_2,\cdots, s_v, t_1,t_2,\cdots, t_v\}$.
\item Given a sign pattern $\cS\in \cS(\cK)$, and a  pairing $\cD\in \cD(\cS)$, we define the set of partitions $\cP(\cS, \cD)=\cP(\{t: \cS_t=-, t\not\in \cD\})$.

\item Given a sign pattern $\cS\in \cS(\cK)$ and a pairing $\cD\in \cD(\cS)$, we define the set of conditions $\cC(\cK, \cS, \cD)$ over $\{j_t\in \bZ_{>0}: \cS_t=+ \text{ or }\cS_t=-, t\not\in \cD\}$: for $1\leq t\leq k_1+k_2+\cdots+k_r$ with $\cS_t=+$, if $t=k_1+k_2+\cdots +k_{s-1}+1$ for some $1\leq s\leq r$,
\begin{align}\label{e:cond1}
j_t+\sum_{v: t<v,\atop\cS_v=+, v\notin \cD} j_v+\sum_{v:u<t<v\atop (u,v)\in \cD}j_v=\sum_{v: t<v,\atop \cS_v=-, v\notin\cD}j_v,
\end{align}
otherwise if $t\not= k_1+k_2+\cdots +k_{s-1}+1$ for any $1\leq s\leq r$,
\begin{align}\label{e:cond2}
j_t+\sum_{v: t<v,\atop\cS_v=+, v\notin \cD} j_v+\sum_{v:u<t<v\atop (u,v)\in \cD}j_v\leq\sum_{v: t<v,\atop \cS_v=-, v\notin\cD}j_v.
\end{align}

\item Given a triple $(\cS, \cD, \cP)$, where $\cS\in \cS(\cK)$ is a sign pattern, $\cD\in \cD(\cS)$ is a pairing, and $\cP\in \cP(\cS, \cD)$ is a partition, we  define a multi-edge hypergraph $\cG(\cK,\cS,\cD, \cP)$ on the vertex set $\{1,2,\cdots,r\}$. For any pair $(s,t)\in \cD$, we add an edge  $\{{\sigma_\cK(s)},{\sigma_\cK(t)}\}$. For any set $V\in \cP$, we add an edge $\{{\sigma_\cK(t)}: t\in V\}$. The connected components of $\cG(\cK,\cS, \cD, \cP)$ induce a partition of $\{1,2,\cdots,r\}$.

\end{itemize}

\begin{theorem}\label{t:decompose}
Fix a Jack generating function $F_N(\bmp;\theta)$, satisfying Assumption \ref{asup:infinite}. For any positive integer $r$ and indices $k_1,k_2,\cdots,k_r\geq 1$, 
\begin{align*}
\left.I^{(k_1)}I^{(k_2)}\cdots I^{(k_r)}F_N\right|_{\bmp=1^N}
\end{align*}
is a sum over terms $\cE(\cK,\cS,\cD, \cP)$
\begin{align}\label{e:sum}
N^{|\{t: \cS_t=+, t\not\in \cD\}|}\sum_{\{j_t\in \bZ_{>0}: \cS_t=-,t\not\in \cD\}}Q_{\cK,\cS, \cD}((j_t: \cS_t=-, t\not\in \cD))\prod_{V\in \cP}\left.\left(\prod_{t\in V} \del p_{j_t}\right)(\ln F_N)\right|_{\bmp=1^N},
\end{align}
where $\cK=(k_1,k_2,\cdots,k_r)$, $\cS\in \cS(\cK)$ is a sign pattern, $\cD\in \cD(\cS)$ is a pairing, $\cP\in \cP(\cS, \cD)$ is a partition,
and $Q_{\cK,\cS, \cD}$ is a polynomial over $\{j_t: \cS_t=-, t\not\in \cD\}$, with 
\begin{align}\label{e:degQKSD}
\deg(Q_{\cK,\cS, \cD})=k_1+k_2+\cdots+k_r-r.
\end{align}
\end{theorem}
The proof of Theorem \ref{t:decompose} follows from Claim \ref{c:firststep} and \ref{c:existtQ}.

\begin{claim}\label{c:firststep}
Fix a Jack generating function $F_N(\bmp;\theta)$, satisfying Assumption \ref{asup:infinite}. For any positive integer $r$ and indices $k_1,k_2,\cdots,k_r\geq 1$, let $\cK=(k_1,k_2,\cdots, k_r)$, then
\begin{align*}
\frac{1}{F_N}I^{(k_1)}I^{(k_2)}\cdots I^{(k_r)}F_N
\end{align*}
is a sum over triples $(\cS, \cD, \cP)$, where $\cS\in \cS(\cK)$ is a sign pattern, $\cD\in \cD(\cS)$ is a pairing, and $\cP\in \cP(\cS, \cD)$ is a partition,
\begin{align*}
\sum_{\cC(\cK,\cS, \cD)}\tilde Q_{\cS, \cD}((j_t: \cS_t=+ \text{ or } \cS_t=-, t\not\in \cD))\prod_{t: \cS_t=+\atop t\not\in \cD}p_{j_t}\prod_{V\in \cP}\left(\prod_{t\in V}\del p_{j_t}\right)(\ln F_N),
\end{align*}
where $\tilde Q_{\cS, \cD}$ is a polynomial over $\{j_t: \cS_t=+ \text{ or } \cS_t=-, t\not\in \cD\}$ given by
\begin{align}\label{e:tQSD}
\tilde Q_{\cS, \cD}((j_t: \cS_t=+ \text{ or } \cS_t=-, t\not\in \cD))
=\prod_{t: \cS_t=0}(1-\theta^{-1})\left(\sum_{v: t<v,\atop v\notin \cD} \cS_vj_v+\sum_{v:u<t<v\atop (u,v)\in \cD}\cS_vj_v\right)\prod_{t:\cS_t=-}\frac{j_t}{\theta}.
\end{align}
\end{claim}

\begin{claim}\label{c:existtQ}
Fix a Jack generating function $F_N(\bmp;\theta)$, satisfying Assumption \ref{asup:infinite}. For any positive integer $r$ and indices $k_1,k_2,\cdots,k_r\geq 1$, let $\cK=(k_1,k_2,\cdots, k_r)$, $\cS\in \cS(\cK)$ a sign pattern, $\cD\in \cD(\cS)$ a pairing, and $\cP\in \cP(\cS, \cD)$ a partition. Then
there exists a polynomial $Q_{\cK,\cS,\cD}$ over $\{j_t: \cS_t=-, t\not\in \cD \}$, such that on $\{j_t\in \bZ_{>0}: \cS_t=-, t\not\in \cD \}$,
\begin{align}\label{e:existtQ}
\sum_{\{j_t\in \bZ_{>0}: \cS_t=+\} \atop\text{satisfies } \cC(\cK, \cS,\cD)}\tilde Q_{\cS, \cD}((j_t: \cS_t=+ \text{ }\mathrm{ or }\text{ } \cS_t=-, t\not\in \cD))=Q_{\cK, \cS,\cD}((j_t: \cS_t=-, t\not\in \cD )),
\end{align}
where the polynomial $\tilde Q_{\cS, \cD}$ is as defined in \eqref{e:tQSD}, and 
\begin{align*}
\deg(Q_{\cK,\cS, \cD})=k_1+k_2+\cdots+k_r-r.
\end{align*}
\end{claim}

\begin{proof}[Proof of Claim \ref{c:firststep}]
We prove the case when $r=1$, i.e. $\cK=(k)$, the general cases follow from the same argument. We recall the definition of $I^{(k)}$ from \eqref{e:Ik},
\begin{align}\label{e:Ikcopy}
I^{(k)}=\sum_{i_1,i_2,\cdots, i_{k-1}\geq 0}L_{0i_1}L_{i_1i_2}\cdots L_{i_{k-2}i_{k-1}}L_{i_{k-1}0},
\end{align} 
If $j\geq i$, $L_{ij}=\delta_{ij}j(\theta^{-1}-1)+(1-\delta_{ij})p_{j-i}$ is a multiplication operator; if $j<i$, $L_{ij}=p_{j-i}=\theta^{-1}(i-j) \del/ \del{p_{i-j}}$ is a differential operator.
We divide the summands of \eqref{e:Ik} into groups according to their sign patterns, which is an element in $\{-,0,+\}^k$. We recall the sign function $\sgn(i)=+$ if $i>0$, $\sgn(i)=-$ if $i<0$ and $\sgn(i)=0$ if $i=0$. Then \eqref{e:Ikcopy} is a sum of the following terms over all $3^k$ sign patterns $\cS=(\cS_1,\cS_2,\cdots, \cS_k)\in\{-,0,+\}^k$,
\begin{align}\label{e:signsum}
\sum_{i_1,i_2,\cdots, i_{k-1}\geq 0,\atop
(\sgn(i_1-i_0), \sgn(i_2-i_1), \cdots,\sgn(i_k-i_{k-1}))=\cS
}\prod_{1\leq t\leq k,\atop\cS_t=0}i_t(\theta^{-1}-1)\prod_{1\leq t\leq k,\atop\cS_t\neq 0}p_{i_{t}-i_{t-1}},
\end{align}
where $i_0=i_k=0$.

We denote $j_t=|i_t-i_{t-1}|$ for $1\leq t\leq k$, and rewrite the last factor of \eqref{e:signsum}
\begin{align}\label{e:lastterm}
\prod_{1\leq t\leq k,\atop\cS_t\neq 0}p_{i_{t}-i_{t-1}}
=\prod_{1\leq t\leq k,\atop\cS_t\neq 0}p_{\cS_tj_t}.
\end{align}
We recall that if $i+j\neq 0$, $p_i$ and $p_j$ commute i.e. $p_ip_j=p_jp_i$. If $i>0$, then $p_{-i}p_i=p_ip_{-i}+i/\theta$. By repeatedly using those commutative relations, we can rewrite \eqref{e:signsum} such that multiplication operators are followed by differential operators, 
\begin{align}\begin{split}\label{e:lastterm2}
\prod_{1\leq t\leq k,\atop\cS_t\neq 0}p_{\cS_tj_t}
&=\sum_{\cD\in \cD(\cS)}\prod_{(s,t)\in \cD}\frac{j_s{\bm1}_{j_s=j_t}}{\theta}\prod_{1\leq t\leq k,\atop\cS_t=+, t\notin \cD}p_{j_t}\prod_{1\leq t\leq k,\atop\cS_t=-, t\notin \cD}p_{-j_t}\\
&=\sum_{\cD\in \cD(\cS)}\prod_{(s,t)\in \cD}\frac{j_s{\bm1}_{j_s=j_t}}{\theta}
\prod_{1\leq t\leq k,\atop\cS_t=-, t\notin \cD}\frac{j_t}{\theta}
\prod_{1\leq t\leq k,\atop\cS_t=+, t\notin \cD}p_{j_t}\prod_{1\leq t\leq k,\atop\cS_t=-, t\notin \cD}\del p_{j_t},
\end{split}\end{align}
where $\cD(\cS)$ is the set of pairings, in the following form $\{(s_1,t_1), (s_2,t_2), \cdots, (s_v,t_v)\}$, where $v\geq 0$, $1\leq s_u<t_u\leq k$, $\cS_{s_u}=-$, $\cS_{t_u}=+$ for $1\leq u\leq v$, and $s_1, s_2,\cdots, s_v, t_1,t_2,\cdots, t_v$ are distinct. In this case, we also identify $\cD$ with the set $\{s_1, s_2,\cdots, s_v, t_1,t_2,\cdots, t_v\}$.

By combining \eqref{e:signsum} and \eqref{e:lastterm2} together, \eqref{e:Ikcopy} is a sum of the following terms over pairs $(\cS, \cD)$, where $\cS$ is a sign pattern and  $\cD\in \cD(\cS)$ is a pairing,
\begin{align}\begin{split}\label{e:SPterm}
\sum_{j_1,j_2,\cdots,j_k\geq 0,\atop \cS_1 j_1+\cS_2j_2+\cdots+\cS_k j_k=0}
&\prod_{1\leq t\leq k, \atop \cS_t=0}\bm1_{j_t=0} \prod_{1\leq t\leq k, \atop \cS_t\neq 0}\bm1_{j_t\geq 1}\prod_{1\leq t\leq k}\bm1_{\cS_t j_t+\cS_{t+1}j_{t+1}+\cdots+
\cS_kj_k\leq 0}\prod_{(s,t)\in \cD}{\bm1}_{j_s=j_t}\\
&\prod_{1\leq t\leq k,\atop \cS_t=0}(1-\theta^{-1})(\cS_tj_t+\cS_{t+1}j_{t+1}+\cdots+\cS_kj_k)\prod_{1\leq t\leq k,\atop \cS_t=-}\frac{j_t}{\theta}\prod_{1\leq t\leq k,\atop\cS_t=+, t\notin \cD}p_{j_t}\prod_{1\leq t\leq k,\atop\cS_t=-, t\notin \cD}\del p_{j_t}.
\end{split}\end{align}
For $1\leq t\leq k$ such that $\cS_t=0$, we have $j_t=0$. And  for $(s,t)\in \cD$, we have $j_s=j_t$. We can view \eqref{e:SPterm} as an expression in terms of $\{j_t: \cS_t=+\text{ or } \cS_t=-, t\not\in \cD\}$. The conditions $\cS_1 j_1+\cS_2j_2+\cdots+\cS_k j_k=0$ and $\cS_t j_t+\cS_{t+1}j_{t+1}+\cdots+
\cS_kj_k\leq 0$ for $1\leq t\leq k$ are equivalent to:
for $1\leq t\leq k$ with $\cS_t=+$, if $t=1$,
\begin{align}\label{e:condset1}
j_t+\sum_{v: t<v,\atop\cS_v=+, v\notin \cD} j_v+\sum_{v:u<t<v\atop (u,v)\in \cD}j_v=\sum_{v: t<v,\atop \cS_v=-, v\notin\cD}j_v,
\end{align}
otherwise if $t\not=1$, 
\begin{align}\label{e:condset2}
j_t+\sum_{v: t<v,\atop\cS_v=+, v\notin \cD} j_v+\sum_{v:u<t<v\atop (u,v)\in \cD}j_v\leq\sum_{v: t<v,\atop \cS_v=-, v\notin\cD}j_v.
\end{align}
The constraints \eqref{e:condset1} and \eqref{e:condset2} define the condition set $\cC(\cK, \cS, \cD)$, over $\{j_t\in \bZ_{>0}: \cS_t=+\text{ or } \cS_t=-, t\not\in \cD\}$. With these notations, we can rewrite \eqref{e:SPterm} as
\begin{align}\label{e:sum2}
\sum_{\cC(\cK,\cS,\cD)}\tilde Q_{\cS, \cD}((j_t: \cS_t=+ \text{ or }\cS_t=-, t\not\in \cD))\prod_{1\leq t\leq k,\atop\cS_t=+, t\notin \cD}p_{j_t}\prod_{1\leq t\leq k,\atop\cS_t=-, t\notin \cD}\del p_{j_t}.
\end{align}
where $\tilde Q_{\cS, \cD}$ is a polynomial given by
\begin{align*}
\tilde Q_{\cS, \cD}((j_t: \cS_t=+ \text{ or } \cS_t=-, t\not\in \cD))
=\prod_{t: \cS_t=0}(1-\theta^{-1})\left(\sum_{v: t<v,\atop v\notin \cD} \cS_vj_v+\sum_{v:u<t<v\atop (u,v)\in \cD}\cS_vj_v\right)\prod_{t:\cS_t=-}\frac{j_t}{\theta}.
\end{align*}
This is \eqref{e:tQSD}. Finally we notice that for any $j\in \bZ_{>0}$,
\begin{align*}
\del p_{j}( F_N )=\del p_{j} (e^{\ln F_N})=F_N  \del p_{j}( \ln F_N),  
\end{align*}
and it follows by induction,
\begin{align}\label{e:sum3}
\prod_{1\leq t\leq k,\atop\cS_t=-, t\notin \cD}\del p_{j_t}F_N
=F_N\sum_{\cP\in \cP(\cS, \cD)}\prod_{V\in \cP}\left(\prod_{t\in V}\del p_{j_t}\right)(\ln F_N),
\end{align}
where $\cP(\cS,\cD)=\cP(\{t: \cS_t=-,t\not\in \cD\})$. The claim \ref{c:firststep} follows from combining \eqref{e:sum2} and \eqref{e:sum3}.

\end{proof}
 
The proof of Claim \ref{c:existtQ} relies on the following elementary lemma.
\begin{lemma}\label{lem:basic1}
For any $k\geq 1$ and a  polynomial $P(x_1,x_2,\cdots, x_k)$, there exists a single variable polynomial $Q_P$, with $\deg Q_P=\deg P+k$, such that for any non-negative integer $y$,
\begin{align}\label{e:plytoply}
\sum_{x_1,x_2,\cdots, x_k\geq 1\atop x_1+x_2+\cdots+x_k\leq y}P(x_1,x_2,\cdots,x_k)=Q_P(y),
\end{align}
where the sum is over integers. Moreover, there exists a single variable polynomial $\Delta Q_P$, with $\deg \Delta Q_P=\deg P+k-1$, such that for any positive integer $y$,
\begin{align}\label{e:plytoply2}
\sum_{x_1,x_2,\cdots, x_k\geq 1\atop x_1+x_2+\cdots+x_k=y}P(x_1,x_2,\cdots,x_k)=\Delta Q_P(y)\deq Q_P(y)-Q_P(y-1),
\end{align}
where the sum is over integers.
\end{lemma}
\begin{proof}
The second statement \eqref{e:plytoply2} follows directly from \eqref{e:plytoply}. In the following, we prove \eqref{e:plytoply} by induction on $k$. For $k=1$, it is well known that, for any non-negative integer $d$, there exists a polynomial $Q_d(x)$, with $\deg Q_d = d+1$, such that
\begin{align}\label{e:k=1}
\sum_{x=1}^y {x^d} = Q_d(y), 
\end{align}
for any non-negative integer $y$. We assume that the statement \eqref{e:plytoply} holds with $k-1$, then
\begin{align*}
\sum_{x_1,x_2,\cdots, x_k\geq 1\atop x_1+x_2+\cdots+x_k\leq y}P(x_1,x_2,\cdots,x_k)
=\sum_{x_1,x_2,\cdots, x_{k-1}\geq 1\atop x_1+x_2+\cdots+x_{k-1}\leq y}\sum_{x_k=1}^{y-(x_1+x_2+\cdots+x_{k-1})}P(x_1,x_2,\cdots,x_k).
\end{align*}
We treat $P(x_1,x_2,\cdots,x_k)$ as a polynomial of $x_k$, with $x_1,x_2,\cdots,x_{k-1}$ as fixed constants. From \eqref{e:k=1}, there exists a $k$-variable polynomial $\tilde Q_P$, with $\deg\tilde Q_P=\deg P+1$,
\begin{align*}
\sum_{x_k=1}^{y-(x_1+x_2+\cdots+x_{k-1})}P(x_1,x_2,\cdots,x_k)=\tilde Q_P(x_1,x_2,\cdots,x_{k-1},y),
\end{align*}
provided $x_1+x_2+\cdots+x_{k-1}\leq y$. We treat $\tilde Q_P(x_1,x_2,\cdots,x_{k-1},y)$ as a polynomial of $x_1,x_2,\cdots,x_{k-1}$, with $y$ as a fixed constant. 
By induction, there exists a single variable polynomial $Q_P$ with $\deg Q_P=\deg \tilde Q_P+k-1=\deg P+k$, such that
\begin{align*}
\sum_{x_1,x_2,\cdots, x_k\geq 1\atop x_1+x_2+\cdots+x_k\leq y}P(x_1,x_2,\cdots,x_k)=
\sum_{x_1,x_2,\cdots, x_{k-1}\geq 1\atop x_1+x_2+\cdots+x_{k-1}\leq y}\tilde Q_P(x_1,x_2,\cdots,x_{k-1},y)=Q_P(y).
\end{align*}
This finishes the proof of \eqref{e:plytoply}.
\end{proof}

\begin{proof}[Proof of Claim \ref{c:existtQ}]
In the following we assume $\cS_{k_1+k_2+\cdots+k_{s-1}+1}=+$, and $\cS_{k_1+k_2+\cdots+k_s-1}=-$ for any $1\leq s\leq r$. Otherwise, the lefthand side of \eqref{e:existtQ} is identically zero, and thus $Q_{\cK,\cS,\cD}\equiv0$. The statement holds trivially.

We treat $\{j_t\in \bZ_{>0}: \cS_t=-, t\not\in\cD\}$ as fixed constants, and sequentially sum over those $j_t$ with $\cS_t=+$ under the conditions $\cC(\cK,\cS,\cD)$, i.e. \eqref{e:cond1} and \eqref{e:cond2}. 
More precisely, we prove by induction that for any $1\leq s\leq k_1+k_2+\cdots+k_r$, there exists a polynomial $Q^{(s)}_{\cK,\cS,\cD}$ over $\{j_t: \cS_t=+, t\geq s  \text{ }\mathrm{ or }\text{ }\cS_t=-, t\not\in \cD\}$, such that on $\{j_t\in \bZ_{>0}:  \cS_t=+ \text{ }\mathrm{ or }\text{ }\cS_t=-, t\not\in \cD\}$,
\begin{align}\begin{split}\label{e:indasup}
&\phantom{{}={}}\sum_{\{j_t\in \bZ_{>0}: \cS_t=+\} \atop\text{satisfies } \cC(\cK, \cS,\cD)}\tilde Q_{\cS, \cD}((j_t: \cS_t=+ \text{ }\mathrm{ or }\text{ } \cS_t=-, t\not\in \cD))\\
&=\sum_{\{j_t\in \bZ_{>0}:  \cS_t=+,t\geq s\} \atop\text{satisfies } \cC^{(s)}(\cK, \cS,\cD)}Q^{(s)}_{\cK, \cS,\cD}((j_t: \cS_t=+ ,t\geq s\text{ }\mathrm{ or }\text{ }\cS_t=-, t\not\in \cD )),
\end{split}\end{align}
where $\cC^{(s)}(\cK, \cS,\cD)$ is the set of conditions:
for $s\leq t\leq k_1+k_2+\cdots+k_r$ with $\cS_t=+$, if $t=k_1+k_2+\cdots +k_{r'-1}+1$ for some $1\leq r'\leq r$,
\begin{align*}
j_t+\sum_{v: t<v,\atop\cS_v=+, v\notin \cD} j_v+\sum_{v:u<t<v\atop (u,v)\in \cD}j_v=\sum_{v: t<v,\atop \cS_v=-, v\notin\cD}j_v,
\end{align*}
otherwise if $t\not= k_1+k_2+\cdots +k_{r'-1}+1$ for any $1\leq r'\leq r$,
\begin{align*}
j_t+\sum_{v: t<v,\atop\cS_v=+, v\notin \cD} j_v+\sum_{v:u<t<v\atop (u,v)\in \cD}j_v\leq\sum_{v: t<v,\atop \cS_v=-, v\notin\cD}j_v.
\end{align*}
Moreover, the degree of $Q^{(s)}_{\cK,\cS, \cD}$ is
\begin{align}\label{e:Qdeg}
\deg(Q^{(s)}_{\cK,\cS, \cD})=|\{t<s: \cS_t=+\}|+|\{t:\cS_t=0\}|+|\{t: \cS_t=-\}|-\sigma_{\cK}(s-1),
\end{align}
where we make the convention $\sigma_{\cK}(0)=0$.

The statements \eqref{e:indasup} and \eqref{e:Qdeg} hold for $s=1$, with $Q_{\cK,\cS,\cD}^{(1)}=\tilde Q_{\cS,\cD}$, and 
\begin{align*}
\deg(Q_{\cK,\cS,\cD}^{(1)})=\deg(\tilde Q_{\cS,\cD})=|\{t:\cS_t=0\}|+|\{t: \cS_t=-\}|.
\end{align*}
We assume that the statement \eqref{e:indasup} holds for $s$. There are several cases. If $\cS_{s}\neq +$, the statement  \eqref{e:indasup} for $s+1$ holds trivially. Since $(j_t: \cS_t=+,t\geq s \text{ }\mathrm{ or }\text{ }\cS_t=-, t\not\in \cD )=(j_t:  \cS_t=+,t\geq s+1 \text{ }\mathrm{ or }\text{ }\cS_t=-, t\not\in \cD )$ and $\cC^{(s)}(\cK, \cS,\cD)=\cC^{(s+1)}(\cK, \cS,\cD)$, we can take $Q^{(s+1)}_{\cK, \cS,\cD}=Q^{(s)}_{\cK, \cS,\cD}$
\begin{align*}\begin{split}
&\phantom{{}={}}\sum_{\{j_t\in \bZ_{>0}:  \cS_t=+,t\geq s\} \atop\text{satisfies } \cC^{(s)}(\cK, \cS,\cD)}Q^{(s)}_{\cK, \cS,\cD}((j_t:  \cS_t=+ ,t\geq s\text{ }\mathrm{ or }\text{ }\cS_t=-, t\not\in \cD )),\\
&=\sum_{\{j_t\in \bZ_{>0}:  \cS_t=+,t\geq s+1\} \atop\text{satisfies } \cC^{(s+1)}(\cK, \cS,\cD)}Q^{(s+1)}_{\cK, \cS,\cD}((j_t:  \cS_t=+,t\geq s+1 \text{ }\mathrm{ or }\text{ }\cS_t=-, t\not\in \cD )).
\end{split}\end{align*}
Moreover, since $\cS_s\neq +$, we have $s\neq k_1+k_2+\cdots+k_{r'-1}+1$ for all $1\leq r'\leq r$ and $\sigma_{\cK}(s-1)=\sigma_{\cK}(s)$. The degree of $Q_{\cK,\cS,\cD}^{(s+1)}$ is 
\begin{align*}
\deg(Q_{\cK,\cS,\cD}^{(s+1)})=\deg(Q_{\cK,\cS,\cD}^{(s)})=|\{t<s+1: \cS_t=+\}|+|\{t:\cS_t=0\}|+|\{t: \cS_t=-\}|-\sigma_{\cK}(s).
\end{align*}
If $\cS_{s}=+$ and $s\neq k_1+k_2+\cdots+k_{r'-1}+1$ for any $1\leq r'\leq r$, $\cC^{(s)}(\cK,\cS,\cD)$ consists of $\cC^{(s+1)}(\cK,\cS,\cD)$ and the constraint on $j_{s}$
\begin{align}\label{e:conds+1}
1\leq j_{s}\leq -\sum_{v: s<v,\atop\cS_v=+, v\notin \cD} j_v-\sum_{v:u<s<v\atop (u,v)\in \cD}j_v+\sum_{v: s<v,\atop \cS_v=-, v\notin\cD}j_v.
\end{align}
We notice that the constraints $\cC^{(s+1)}(\cK,\cS,\cD)$ ensures that the righthand side of \eqref{e:conds+1} is non-negative.  Therefore we can sum over $j_{s}$, and by Lemma \ref{lem:basic1}, there exists a polynomial
$Q^{(s+1)}_{\cK, \cS,\cD}$
\begin{align*}\begin{split}
&\phantom{{}={}}\sum_{\{j_t\in \bZ_{>0}: \cS_t=+,t\geq s\} \atop\text{satisfies } \cC^{(s)}(\cK, \cS,\cD)}Q^{(s)}_{\cK, \cS,\cD}((j_t:  \cS_t=+,t\geq s \text{ }\mathrm{ or }\text{ }\cS_t=-, t\not\in \cD )),\\
&=\sum_{\{j_t\in \bZ_{>0}:  \cS_t=+,t\geq s+1\} \atop\text{satisfies } \cC^{(s+1)}(\cK, \cS,\cD)}Q^{(s+1)}_{\cK, \cS,\cD}((j_t:  \cS_t=+,t\geq s+1 \text{ }\mathrm{ or }\text{ }\cS_t=-, t\not\in \cD )).
\end{split}\end{align*}
Moreover, we have $s\neq k_1+k_2+\cdots+k_{r'-1}+1$ for all $1\leq r'\leq r$ and thus $\sigma_{\cK}(s-1)=\sigma_{\cK}(s)$. The degree of $Q_{\cK,\cS,\cD}^{(s+1)}$ is 
\begin{align*}
\deg(Q_{\cK,\cS,\cD}^{(s+1)})=\deg(Q_{\cK,\cS,\cD}^{(s)})+1=|\{t<s+1: \cS_t=+\}|+|\{t:\cS_t=0\}|+|\{t: \cS_t=-\}|-\sigma_{\cK}(s).
\end{align*}
The last case is $\cS_{s}=+$ and $s= k_1+k_2+\cdots+k_{r'-1}+1$ for some $1\leq r'\leq r$.
There exists some $1\leq l\leq k_{r'}$, such that $\cS_{s+l}=-$. Let $l$ be the smallest such value. Then we have $\cS_{s}=\cS_{s+1}=\cdots=\cS_{s+l-1}=+$ and $\cC^{(s)}(\cK, \cS,\cD)$ is equivalent to $\cC^{(s+l)}(\cK, \cS,\cD)$ and the constraints on $\{j_t\in \bZ_{>0}:\cS_t=+,s\leq t\leq s+l-1\}$:
\begin{align}\label{e:conds+l}
\sum_{s\leq t\leq s+l-1}j_t=-\sum_{v: s+l-1<v,\atop\cS_v=+, v\notin \cD} j_v-\sum_{v:u<s+l-1<v\atop (u,v)\in \cD}j_v+\sum_{v: s+l-1<v,\atop \cS_v=-, v\notin\cD}j_v.
\end{align}
We notice that $\cS_{s+l}=-$ and the constraints $\cC^{(s+l)}(\cK,\cS,\cD)$ together ensure that the righthand side of \eqref{e:conds+l} is a positive integer.  Therefore we can sum over $\{j_t\in \bZ_{>0}: \cS_t=+,s\leq t\leq s+l-1\}$, and by Lemma \ref{lem:basic1}, there exists a polynomial
$Q^{(s+l)}_{\cK, \cS,\cD}$
\begin{align*}\begin{split}
&\phantom{{}={}}\sum_{\{j_t\in \bZ_{>0}:  \cS_t=+,t\geq s\} \atop\text{satisfies } \cC^{(s)}(\cK, \cS,\cD)}Q^{(s)}_{\cK, \cS,\cD}((j_t:  \cS_t=+,t\geq s \text{ }\mathrm{ or }\text{ }\cS_t=-, t\not\in \cD )),\\
&=\sum_{\{j_t\in \bZ_{>0}:  \cS_t=+,t\geq s+l\} \atop\text{satisfies } \cC^{(s+l)}(\cK, \cS,\cD)}Q^{(s+l)}_{\cK, \cS,\cD}((j_t: \cS_t=+ ,t\geq s+l\text{ }\mathrm{ or }\text{ }\cS_t=-, t\not\in \cD )).
\end{split}\end{align*}
Moreover, $\sigma_{\cK}(s+l-1)=\sigma_{\cK}(s-1)+1$, the degree of $Q_{\cK,\cS,\cD}^{(s+l)}$ is 
\begin{align*}\begin{split}
\deg(Q_{\cK,\cS,\cD}^{(s+l)})
&=\deg(Q_{\cK,\cS,\cD}^{(s)})+|\{s\leq t<s+l:\cS_t=+\}|-1\\
&=|\{t<s+l: \cS_t=+\}|+|\{t:\cS_t=0\}|+|\{t: \cS_t=-\}|-\sigma_{\cK}(s+l-1).
\end{split}\end{align*}
Therefore, by induction \eqref{e:indasup} and \eqref{e:Qdeg} hold for any $s\geq 1$. We can take $Q_{\cK,\cS,\cD}=Q_{\cK,\cS,\cD}^{(k_1+k_2+\cdots+k_r+1)}$. This finishes the proof of Claim \ref{c:existtQ}.
\end{proof}

We recall the definition of the hypergraph $\cG(\cK, \cS, \cD, \cP)$ on the vertex set $\{1,2,\cdots,r\}$, at the beginning of this section. The connected components of $\cG(\cK,\cS, \cD, \cP)$ induce a partition of $\{1,2,\cdots,r\}$. Informally, if $s$ and $t$ are in different connected components of $\cG(\cK,\cS, \cD, \cP)$, then in \eqref{e:sum} $I^{(k_s)}$ and $I^{(k_t)}$ do not interact. Especially, \eqref{e:sum} splits into a product of terms corresponding to different connected components of $\cG(\cK,\cS, \cD, \cP)$. More precisely, say $\cG(\cK,\cS, \cD, \cP)$ induces a partition $\{U_1, U_2,\cdots, U_v\}$ of $\{1,2,\cdots,r\}$, then \eqref{e:sum} equals
\begin{align}\label{e:Usplit}\begin{split}
&\cE(\cK,\cS, \cD, \cP)=\prod_{u=1}^v\cE(\cK|_{U_u}, \cS|_{U_u},\cD|_{U_u}, \cP|_{U_u})=\prod_{u=1}^v
N^{|\{t:(\cS|_{U_u})_t, t\not\in \cD|_{U_u}\}|}
\sum_{\{j_t\in \bZ_{>0}: (\cS|_{U_u})_t=-,t\not\in \cD|_{U_u}\}}\\
& Q_{\cK|_{U_u},\cS|_{U_u}, \cD|_{U_u}}((j_t: (\cS|_{U_u})_t=-, t\not\in \cD|_{U_u}))\prod_{V\in \cP|_{U_u}}\left(\prod_{t\in V}\del p_{j_t}\right)\left.(\ln F_N)\right|_{\bmp=1^N},
\end{split}\end{align}
where $\cK|_{U_u}=(k_s: s\in U_u)$ is the restriction of $\cK$ on $U_u$. 
$\cS|_{U_u}, \cD|_{U_u}, \cP|_{U_u}$, are the restriction of $\cS, \cD, \cP$ on the index set $\{t: \sigma(t)\in U_u\}$, respectively, and they satisfy $\cS|_{U_u}\in \cS(\cK|_{U_u})$, $\cD|_{U_u}\in \cD(\cS|_{U_u})$, and $\cP|_{U_u}\in \cP(\cS|_{U_u}, \cD|_{U_u})$. The polynomial $Q_{\cK|_{U_u}\cS|_{U_u}, \cD|_{U_u}}$ is defined using $\cK|_{U_u}, \cS|_{U_u}, \cD|_{U_u}$.
For any positive integer $r$ and indices $k_1,k_2,\cdots,k_r\geq1$, we define 
\begin{align}\label{e:defFk}
\cF_{(k_1,k_2,\cdots,k_r)}\deq \sum_{\cS\in\cS(\cK), \cD\in \cD(\cS), \cP\in \cP(\cS, \cD)\atop\cG(\cK,\cS, \cD, \cP) \text{ is connected}}\cE(\cK,\cS, \cD, \cP).
\end{align}
Then we have
\begin{align*}
\cF_{(k)}=\left.I^{(k)}F_N\right|_{\bmp=1^N},
\end{align*}
and
\begin{align*}
\cF_{(k_1,k_2)}=\left.I^{(k_1)}I^{(k_2)}F_N\right|_{\bmp=1^N}-\cF_{(k_1)}\cF_{(k_2)}.
\end{align*}

\begin{proposition}\label{p:totalsum}
Given any positive integer $r$ and indices $k_1,k_2,\cdots,k_r\geq1$.  
For any partition $\{U_1, U_2,\cdots, U_v\}$ of $\{1,2,\cdots,r\}$, we have
\begin{align}\label{e:tsum1}
\prod_{u=1}^v \cF_{(k_s: s\in U_u)}=\sum_{\cS\in\cS(\cK), \cD\in \cD(\cS), \cP\in \cP(\cS, \cD)\atop\cG(\cK,\cS, \cD, \cP) \text{ induces } \{U_1,U_2,\cdots, U_v\}}\cE(\cK,\cS,\cD,\cP).
\end{align}
Moreover, 
\begin{align}\begin{split}\label{e:tsum2}
\left.I^{(k_1)}I^{(k_2)}\cdots I^{(k_r)}F_N\right|_{\bmp=1^N}=
\sum_{v\geq 1}\sum_{\{U_1,U_2,\cdots, U_v\}\atop 
\in \cP(\{1,2,\cdots,r\})}\prod_{u=1}^v \cF_{(k_s: s\in U_u)},
\end{split}\end{align}
and
\begin{align}\label{e:tsum3}
\cF_{(k_1,k_2,\cdots,k_r)}=\sum_{\cP\in \cP(\{1,2,\cdots,r\})}(-1)^{|\cP|-1}(|\cP|-1)!\prod_{V\in \cP}\left.\left(\prod_{s\in V}I^{(k_s)}\right)F_N\right|_{\bmp=1^N}.
\end{align}
As a consequence of \eqref{e:tsum3}, $\cF_{(k_1,k_2,\cdots, k_r)}$ only depends on the set $\{k_1,k_2,\cdots, k_r\}$.
\end{proposition}

\begin{proof}
If $\cG(\cK,\cS, \cD, \cP)$  induces the partition $\{U_1,U_2,\cdots, U_v\}$, we denote $\cK_u=\cK|_{U_u}$ for $1\leq u\leq v$.
\eqref{e:sum} splits according to \eqref{e:Usplit},
\begin{align*}\begin{split}
&\phantom{{}={}}\sum_{\cS\in\cS(\cK), \cD\in \cD(\cS), \cP\in \cP(\cS, \cD)\atop\cG(\cK,\cS, \cD, \cP) \text{ induces } \{U_1,U_2,\cdots, U_v\}}\cE(\cK,\cS, \cD, \cP)\\
&=\sum_{\cS\in\cS(\cK), \cD\in \cD(\cS), \cP\in \cP(\cS, \cD)\atop\cG(\cK,\cS, \cD, \cP) \text{ induces } \{U_1,U_2,\cdots, U_v\}}\prod_{u=1}^v\cE(\cK|_{U_u}, \cS|_{U_u},\cD|_{U_u}, \cP|_{U_u})\\
&=\prod_{u=1}^v \sum_{\cS_u\in\cS(\cK_u), \cD_u\in \cD(\cS_u), \cP_u\in \cP(\cS_u, \cD_u)\atop\cG(\cK_u,\cS_u, \cD_u, \cP_u) \text{ is connected}}\cE(\cK_u, \cS_u,\cD_u, \cP_u)=\prod_{u=1}^v\cF_{\cK_u}=\prod_{u=1}^v \cF_{(k_s:s\in U_u)}.
\end{split}\end{align*}
Thus \eqref{e:tsum1} follows. \eqref{e:tsum2} follows from Proposition \ref{t:decompose},
\begin{align*}\begin{split}
\left.I^{(k_1)}I^{(k_2)}\cdots I^{(k_r)}F_N\right|_{\bmp=1^N}
&=\sum_{v\geq 1}\sum_{\{U_1,U_2,\cdots, U_v\}\atop 
\in \cP(\{1,2,\cdots,r\})}
\sum_{\cS\in\cS(\cK), \cD\in \cD(\cS), \cP\in \cP(\cS, \cD)\atop\cG(\cK,\cS, \cD, \cP) \text{ induces } \{U_1,U_2,\cdots, U_v\}}\cE(\cK,\cS, \cD, \cP)\\
&=\sum_{v\geq 1}\sum_{\{U_1,U_2,\cdots, U_v\}\atop 
\in \cP(\{1,2,\cdots,r\})}\prod_{u=1}^v \cF_{(k_s: s\in U_u)},
\end{split}\end{align*}
The relation \eqref{e:tsum2} resembles the relation between cumulants and moments. It follows by induction,
\begin{align}\label{e:copytsum3}
\cF_{(k_1,k_2,\cdots,k_r)}=\sum_{\cP\in \cP(\{1,2,\cdots,r\})}(-1)^{|\cP|-1}(|\cP|-1)!\prod_{V\in \cP}\left.\left(\prod_{s\in V}I^{(k_s)}\right)F_N\right|_{\bmp=1^N}.
\end{align}
Since the operators $I^{(k_1)}, I^{(k_2)},\cdots, I^{(k_r)}$ commute, we conclude from \eqref{e:copytsum3} that $\cF_{(k_1,k_2,\cdots, k_r)}$ only depends on the set $\{k_1,k_2,\cdots, k_r\}$.

\end{proof}

From Theorem \ref{t:moment}, the lefthand side of \eqref{e:tsum2} is related to the joint moments of the Perelomov-Popov measure as defined in \eqref{def:PP}. The relations \eqref{e:tsum2} and \eqref{e:tsum3} resemble the relations between moments and cumulants. It turns out joint cumulants of the random variables 
\begin{align*}
\xi_{PP}^{(k)}[\bmla]
=N\left(\int x^k \rd\mu_{PP} [\bmla]-\bE \int x^k \rd\mu_{PP} [\bmla]\right)
\end{align*}
 are linear combinations of $\cF_{(k_1,k_2,\cdots,k_r)}$.
\begin{proposition}\label{p:cumulant}
For any positive integer $r\geq 2$, and $k_1,k_2,\cdots,k_r\geq 1$, as a formal power series of $1/z_1, 1/z_2,\cdots, 1/z_r$,
\begin{align}\begin{split}\label{e:cumulant}
&\phantom{{}={}}\sum_{k_1,k_2,\cdots,k_r\geq 1}\frac{\kappa_{k_1, k_2, \cdots, k_r}(\xi_{PP}^{(k)}[\bmla],k=1,2,3,\cdots)}{z_1^{k_1+1}z_2^{k_2+1}\cdots z_r^{k_r+1}}\\
&=\prod_{j=1}^{r}(z_j+1)\sum_{k_1,k_2,\cdots,k_r\geq 1}\frac{\cF_{(k_1,k_2,\cdots, k_r)}}{N^{k_1+k_2+\cdots+k_r-r}z_1^{k_1+1}z_2^{k_2+1}\cdots z_r^{k_r+1}}.
\end{split}\end{align}
\end{proposition}
\begin{proof}
The cumulant $\kappa_{k_1, k_2, \cdots, k_r}(\xi_{PP}^{(k)}[\bmla],k=1,2,3,\cdots)$ is given by
\begin{align*}\begin{split}
&\phantom{{}={}}\kappa_{k_1, k_2, \cdots, k_r}(\xi_{PP}^{(k)}[\bmla],k=1,2,3,\cdots)
\\
&=\sum_{\cP\in \cP(\{1,2,\cdots,r\})}(-1)^{|\cP|-1}(|\cP|-1)!\prod_{V\in \cP}\bE\left[\prod_{s\in V}N\int x^{k_s}\rd \mu_{PP}[\bmla]\right]\\
&=\sum_{\cP\in \cP(\{1,2,\cdots,r\})}(-1)^{|\cP|-1}(|\cP|-1)!\prod_{V\in \cP}\prod_{s\in V}\left.\left(\frac{I^{(k_s)}}{N^{k_s-1}}+\frac{I^{(k_s+1)}}{N^{k_s}}\right)F_N(\bmp;\theta)\right|_{\bmp=1^N},
\end{split}\end{align*}
where we used Theorem \ref{t:moment}. Therefore, we have the following generating function of the $r$-th cumulants,
\begin{align*}\begin{split}
&\phantom{{}={}}\sum_{k_1,k_2,\cdots,k_r\geq 1}\frac{\kappa_{k_1, k_2, \cdots, k_r}(\xi_{PP}^{(k)}[\bmla],k=1,2,3,\cdots)}{z_1^{k_1+1}z_2^{k_2+1}\cdots z_r^{k_r+1}}\\
&=\prod_{j=1}^{r}(z_j+1)\sum_{k_1,k_2,\cdots,k_r\geq 1}\sum_{\cP\in \cP(\{1,2,\cdots,r\})}\frac{(-1)^{|\cP|-1}(|\cP|-1)!\prod_{V\in \cP}\left.\left(\prod_{s\in V}I^{(k_s)}F_N\right)\right|_{\bmp=1^N}}{N^{k_1+k_2+\cdots+k_r-r}z_1^{k_1+1}z_2^{k_2+1}\cdots z_r^{k_r+1}}\\
&=\prod_{j=1}^{r}(z_j+1)\sum_{k_1,k_2,\cdots,k_r\geq 1}\frac{\cF_{(k_1,k_2,\cdots, k_r)}}{N^{k_1+k_2+\cdots+k_r-r}z_1^{k_1+1}z_2^{k_2+1}\cdots z_r^{k_r+1}},
\end{split}\end{align*}
where we used \eqref{e:tsum3} in the last line. 
\end{proof}

\subsection{Law of large numbers}\label{subs:LLN}

%

The law of large numbers for the measure $M_N$ follows, if the limit exists
\begin{align}\begin{split}\label{e:limit}
&\phantom{{}={}}\lim_{N\rightarrow \infty}\bE\left[\prod_{j=1}^r\int x^{k_j}\rd \mu_{PP}[\bmla](x) \right]=\lim_{N\rightarrow \infty}\left.\prod_{j=1}^r\left(\frac{I^{(k_j)}}{N^{k_j}}+\frac{I^{(k_j+1)}}{N^{k_j+1}}\right)F_N\right|_{\bmp=1^N}
\end{split}\end{align}
and splits.

\begin{claim}\label{c:LLNup}
Fix a Jack generating function $F_N(\bmp;\theta)$, satisfying Assumption \ref{asup:infinite} and \ref{a:LLN}. For any positive integer $k$, $j_1,j_2,\cdots, j_k$ viewed as symbols, and $\{V_1, V_2,\cdots, V_s\}$ a partition of $\{1,2,\cdots, k\}$, we have that
\begin{align*}
\sum_{j_1,j_2,\cdots, j_k\geq 1}Q(j_1, j_2,\cdots, j_k)
\prod_{u=1}^s \left(\prod_{t\in V_u}\del p_{j_t}\right)(\ln F_N)|_{\bmp=1^N},
\end{align*}
has an $N$-degree at most $s$,
if $|V_u|= 1$ for all $1\leq u\leq s$; and has an $N$-degree less than
$s$,
if $|V_u|> 1$ for some $1\leq u\leq s$.
\end{claim}
\begin{proof}
The claims follow from Assumption \ref{a:LLN}, for any $|V_u|$-variable polynomial $Q_u$,
\begin{align*}
\sum_{\{j_t\in \bZ_{>0}:t\in V_u\}}Q_u((j_t: t\in V_u))\left(\prod_{t\in V_u}\del p_{j_t}\right)(\ln F_N)|_{\bmp=1^N}
\end{align*}
has an $N$-degree $1$ if $|V_u|=1$, and has an $N$-degree less than $1$ if $|V_u|>1$.
\end{proof}

\begin{proposition}\label{p:split}
Fix a Jack generating function $F_N(\bmp;\theta)$, satisfying Assumption \ref{asup:infinite} and \ref{a:LLN}. For any positive integer $r$ and indices $k_1,k_2,\cdots, k_r$, the limit splits
\begin{align}\label{e:split}
\lim_{N\rightarrow \infty}\left.\frac{I^{(k_1)}I^{(k_2)}\cdots I^{(k_r)}}{N^{k_1+k_2+\cdots + k_r}}F_N\right|_{\bmp=1^N}
=\prod_{j=1}^r
\lim_{N\rightarrow \infty}\left.\frac{I^{(k_j)}}{N^{k_j}}F_N\right|_{\bmp=1^N},
\end{align}
and the following limit exists
\begin{align*}
\lim_{N\rightarrow \infty}\left.\frac{I^{(k)}}{N^k}F_N(\bmp;\theta)\right|_{\bmp=1^N}
=\lim_{N\rightarrow \infty}((T_{F_N})^k)_{00}.
\end{align*} 
where the infinite Toeplitz matrix $[(T_{F_N})_{ij}]_{i,j=0}^\infty$ is defined in Section \ref{s:SN}, and is characterized by
\begin{align}\label{e:characterTMN}
(z-T_{F_N})^{-1}_{00}
=\exp\left(-\frac{1}{2\pi \ri}\oint_{|w|=1+\varepsilon/2}\ln\left(z-T_{F_N}(w)\right)\frac{\rd w}{w}\right),
\end{align}
as a formal power series of $1/z$.
\end{proposition}

\begin{proof}
We use the same notation  as in section \ref{s:Tl}. Fix $\cK=(k_1,k_2,\cdots, k_r)$. From Theorem \ref{t:decompose}, we have that 
\begin{align*}
\left.\frac{I^{(k_1)}I^{(k_2)}\cdots I^{(k_r)}}{N^{k_1+k_2+\cdots + k_r}}F_N\right|_{\bmp=1^N}
\end{align*} 
is a sum over triples $(\cS, \cD, \cP)$, where $\cS\in \cS(\cK)$ is a sign pattern, $\cD\in \cD(\cS)$ is a pairing, and $\cP\in \cP(\cS, \cD)$ is a partition,
\begin{align}\label{e:sumcopy}
\frac{N^{|\{t: \cS_t=+, t\not\in \cD\}|}}{N^{k_1+k_2+\cdots+k_r}}\sum_{\{j_t\in \bZ_{>0}: \cS_t=-,t\not\in \cD\}}Q_{\cK,\cS, \cD}((j_t: \cS_t=-, t\not\in \cD))\prod_{V\in \cP}\left(\prod_{t\in V}\del p_{j_t}\right)\left.(\ln F_N)\right|_{\bmp=1^N},
\end{align}
where $Q_{\cK,\cS, \cD}$ is a polynomial over $\{j_t: \cS_t=-, t\not\in \cD\}$. Say $\cP=\{V_1, V_2,\cdots, V_s\}$, then from the Claim \ref{c:LLNup} the $N$-degree of \eqref{e:sumcopy} is at most
\begin{align*}
|\{t: \cS_t=+, t\not\in \cD\}|-(k_1+k_2+\cdots+k_r)+s
=-2|\cD|-\sum_{u=1}^s(|V_u|-1)-|\{t:\cS_t=0\}|.
\end{align*}
More precisely, if $|V_1|=|V_2|=\cdots=|V_s|=1$, then \eqref{e:sumcopy} has an $N$-degree at most $-2|\cD|-|\{t: \cS_t=0\}|\leq 0$.
Otherwise, if for some $1\leq u\leq s$, $|V_u|>1$, then \eqref{e:sumcopy} has an $N$-degree less than
\begin{align*}
-2|\cD|-\sum_{u=1}^s(|V_u|-1)-|\{t:\cS_t=0\}|<0.
\end{align*}
Therefore the terms from \eqref{e:sumcopy} that contribute are those with $\cD=\emptyset$ and $|V_1|=|V_2|=\cdots=|V_s|=1$. In this case $\cG(\cK, \cS, \cD, \cP)$ is a hypergraph without any edges, and by Proposition \ref{p:totalsum} we have
\begin{align*}
\lim_{N\rightarrow \infty}\left.\frac{I^{(k_1)}I^{(k_2)}\cdots I^{(k_r)}}{N^{k_1+k_2+\cdots + k_r}}F_N\right|_{\bmp=1^N}
=\lim_{N\rightarrow \infty}\prod_{j=1}^r \frac{\cF_{(k_j)}}{N^{k_j}}=\prod_{j=1}^r
\lim_{N\rightarrow \infty}\left.\frac{I^{(k_j)}}{N^{k_j}}F_N\right|_{\bmp=1^N}.
\end{align*}
Thus \eqref{e:split} follows.

%

From the discussion above, the leading order terms in $\left.(I^{(k)}/N^k)F_N(\bmp;\theta)\right|_{\bmp=1^N}$ corresponds to triples $(\cS, \cD, \cP)$, where the sign pattern $\cS\in \{-,+\}^{k}$, the pairing $\cD=\emptyset$, and the partition $\cP=\{\{j_t\}: \cS_t=-\}$. We can estimate the following limit
\begin{align*}\begin{split}
&\phantom{{}={}}\lim_{N\rightarrow \infty}\left.\frac{I^{(k)}}{N^k}F_N\right|_{\bmp=1^N}
=\lim_{N\rightarrow\infty}\sum_{i_1,i_2,\cdots,i_{k-1}\geq 0}
\left.\prod_{s: i_s\geq i_{s-1}}\frac{p_{i_s-i_{s-1}}}{N} \prod_{s: i_s< i_{s-1}}\frac{p_{i_s-i_{s-1}}(\ln F_N)}{N}\right|_{\bmp=1^N}\\
&=\lim_{N\rightarrow \infty}\sum_{i_1,i_2,\cdots, i_{k-1}\geq 0}(T_{F_N})_{0i_1}(T_{F_N})_{i_1i_2}\cdots (T_{F_N})_{i_{k-2}i_{k-1}}(T_{F_N})_{i_{k-1}0}=\lim_{N\rightarrow \infty}((T_{F_N})^k)_{00},
\end{split}\end{align*}
where the infinite Toeplitz matrix $[(T_{F_N})_{ij}]_{i,j=0}^\infty$ is defined in Section \ref{s:SN}. The last statement \eqref{e:characterTMN} follows from Lemma \ref{l:invertT}. We recall the symbol of $[(T_{F_N})_{ij}]_{i,j=0}^\infty$ in \eqref{def:TMN}, then by Lemma \ref{l:invertT}
\begin{align*}\begin{split}
(z-T_{F_N})^{-1}_{00}
&=\exp\left(-\frac{1}{2\pi \ri}\oint_{|w|=1+\varepsilon/2}\ln\left(z-T_{F_N}(w)\right)\frac{\rd w}{w}\right),
\end{split}\end{align*}
as a formal power series of $1/z$. 

\end{proof}

\begin{proof}[Proof of Theorem \ref{thm:LLN}]
From Proposition \ref{p:split} and Theorem \ref{t:moment}, we have
\begin{align*}
\lim_{N\rightarrow \infty}\bE\left[\prod_{j=1}^r\int x^{k_j}\rd \mu_{PP}[\bmla](x) \right]
=\prod_{j=1}^r\lim_{N\rightarrow \infty}\bE\left[\int x^{k_j}\rd \mu_{PP}[\bmla](x) \right],
\end{align*}
and the limits exist. Thus the sequence of random measures $\mu_{PP}[\bmla]$ satisfies the law of large numbers in the sense of Definition \ref{def:LLNCLT}. Combining with Lemma \ref{l:change} and \ref{l:CLTtoCLT}, we conclude that the sequence of random measures $\mu[\bmla]$ also satisfies the law of large numbers. By a tightness argument, the sequence of random measures $\mu[\bmla]$ converges as $N\rightarrow \infty$, weakly to a deterministic measure $\mu$ on $\bR$. 

In the following we determine the moments of $\mu$. From Remark \ref{r:change} and \eqref{e:PPid}, as a formal power series of $1/z$, we have
\begin{align}\begin{split}\label{e:formalsum}
e^{m_\mu(z)}
&=\lim_{N\rightarrow\infty}(z+1)(z-T_{F_N})^{-1}_{00}\\
&=\lim_{N\rightarrow\infty}\exp\left(\ln(z+1)-\frac{1}{2\pi \ri}\oint_{|w|=1+\varepsilon/2}\ln\left(z-T_{F_N}(w)\right)\frac{\rd w}{w}\right).
\end{split}\end{align}
We take logarithm on both sides of \eqref{e:formalsum}, and compare the coefficients of $1/z^{k+1}$,
\begin{align}\begin{split}\label{e:kmoment}
&\phantom{{}={}}\int x^k \rd \mu(x)=\lim_{N\rightarrow\infty}\frac{1}{k+1}\left(\frac{1}{2\pi \ri}\oint_{|w|=1+\varepsilon/2} \left(T_{F_N}(w)\right)^{k+1}\frac{\rd w}{w}+(-1)^k\right)\\
&=\lim_{N\rightarrow\infty}
\frac{1}{k+1}\left(\frac{1}{2\pi\ri}\oint_{|w|=1+\varepsilon/2}\left(\frac{w}{\theta}U_{F_N}(w)+\frac{1}{w-1}\right)^{k+1}\frac{\rd w}{w}+(-1)^k\right)
\end{split}\end{align}
The integrand has two poles at $w=0$ and $w=1$. The residual at $w=0$ is $(-1)^{k+1}$, which cancels with the last term $(-1)^k$. Thus we can rewrite  \eqref{e:kmoment} as,
\begin{align*}\begin{split}
\int x^k \rd \mu(x)
&=\lim_{N\rightarrow\infty}\frac{1}{2\pi \ri(k+1)}\int_{|w-1|
=\varepsilon/2}\left(\frac{w}{\theta}U_{F_N}(w)+\frac{1}{w-1}\right)^{k+1}\frac{\rd w}{w}\\
&=\lim_{N\rightarrow\infty}\frac{1}{2\pi \ri(k+1)}\int_{|w|
=\varepsilon/2}\left(\frac{w+1}{\theta}U_{F_N}(w+1)+\frac{1}{w}\right)^{k+1}\frac{\rd w}{w+1}\\
&=[w^{-1}]\frac{1}{(k+1)}\left(\frac{(w+1)}{\theta}U_{\mu}(w+1)+\frac{1}{w}\right)^{k+1}\sum_{a\geq 0}(-w)^a,
\end{split}
\end{align*}
where we used \eqref{e:UMNlimit}, and the claim \eqref{eqn:moments} follows.
\end{proof}

%
%
%
%
%
%

%
%

\subsection{Central limit theorems for one level}\label{subs:CLT}
Before proving Theorem \ref{t:clt}, we need some elementary estimates, which will be repeatedly used in the rest of this section.
\begin{claim}\label{c:CLTup}
Fix a Jack generating function $F_N(\bmp;\theta)$, satisfying Assumption \ref{asup:infinite} and \ref{a:CLT}. For any positive integer $k$, $j_1,j_2,\cdots, j_k$ viewed as symbols, and $\{V_1, V_2,\cdots, V_s\}$ a partition of $\{1,2,\cdots,k\}$, we have
\begin{align*}
\sum_{j_1,j_2,\cdots, j_k\geq 1}Q(j_1, j_2,\cdots, j_k)
\prod_{u=1}^s \left(\prod_{t\in V_u}\del p_{j_t}\right)(\ln F_N)|_{\bmp=1^N}
\end{align*}
has an $N$-degree at most 
\begin{align}\label{e:cltdeg}
\sum_{u=1}^s {\bm1_{|V_u|=1}},
\end{align}
if $|V_u|\leq 2$ for all $1\leq u\leq s$; and has an $N$-degree less than
\eqref{e:cltdeg},
if $|V_u|> 2$ for some $1\leq u\leq s$.
\end{claim}
\begin{proof}
The claims follow from Assumption \ref{a:CLT}, for any $|V_u|$-variable polynomial $Q_u$,
\begin{align*}
\sum_{\{j_t\in \bZ_{>0}:t\in V_u\}}Q_u((j_t: t\in V_u))\left(\prod_{t\in V_u}\del p_{j_{t}}\right)(\ln F_N)|_{\bmp=1^N}
\end{align*}
has an $N$-degree $1$ if $|V_u|=1$, has an $N$-degree $0$ if $|V_u|=2$, and has an $N$-degree less than $0$ if $|V_u|>2$.
\end{proof}

\begin{claim}\label{c:edgeineq}
If the hypergraph $\cG(\cK, \cS, \cD, \cP)$ is connected, then
\begin{align*}
|\cD|+\sum_{V\in \cP}(|V|-1)\geq r-1.
\end{align*}
\end{claim}
\begin{proof}
The hypergraph $\cG(\cK,\cS,\cD, \cP)$ is obtained from the graph $\{1,2,\cdots,r\}$ (which consists of $r$ vertices and no edges) by sequentially adding edges $\{{\sigma_\cK(s)},{\sigma_\cK(t)}\}$ for pairs $(s,t)\in \cD$, and edges $\{{\sigma_\cK(t)}: t\in V\}$ for sets $V\in \cP$. The graph $\{1,2,\cdots,r\}$ has $r$ connected components, the hypergraph $\cG(\cK,\cS,\cD, \cP)$ has only one connected component, and by adding an edge $\{{\sigma_\cK(s)},{\sigma_\cK(t)}\}$ the number of connected components decreases by at most one, by adding an edge $\{{\sigma_\cK(t)}: t\in V\}$, the number of connected components decreases by at most $|V|-1$. The claim follows.

\end{proof}

\begin{proposition}\label{p:crossterm}
For any indices $k_1,k_2\geq 1$, $\cF_{(k_1,k_2)}$ has an $N$-degree at most $k_1+k_2-2$,
\begin{align}\begin{split}\label{e:crossterm}
\lim_{N\rightarrow \infty}\frac{\cF_{(k_1,k_2)}}{N^{k_1+k_2-2}}
=\lim_{N\rightarrow \infty}
&\left(\sum_{k\geq 1}\frac{k}{\theta}\del q_{k} (T^{k_1})_{00}\del q_{-k} (T^{k_2})_{00}\right.\\
&\left.\left.+\sum_{k,l\geq 1}\frac{q_{k,l}(M_N)kl}{\theta^2}\del q_{k}(T^{k_1})_{00}\del q_{l}(T^{k_2})_{00}\right)\right|_{q_j=q_j(F_N)},
\end{split}\end{align}
where the infinite Topelitz matrix $[T_{ij}]_{i,j=0}^\infty$ is defined in Section \ref{s:SN}, given by $T_{ij}=q_{i-j}$. 
 For any positive integer $r\geq 3$, and indices $k_1,k_2,\cdots, k_r\geq 1$, $\cF_{(k_1, k_2, \cdots,k_r)}$ has an $N$-degree less than $k_1+k_2+\cdots+k_r-r$.
\end{proposition}
\begin{proof}
We use the same notations  as in section \ref{s:Tl}. From Proposition \ref{t:decompose}, we have that 
$
\cF_{(k_1,k_2,\cdots,k_r)}
$
is a sum over $\cE(\cK,\cS, \cD, \cP)$ such that the hypergraph $\cG(\cK, \cS, \cD, \cP)$ is connected, where $\cK=(k_1,k_2,\cdots, k_r)$, $\cS\in \cS(\cK)$ is a sign pattern, $\cD\in \cD(\cS)$ is a pairing, and $\cP\in \cP(\cS, \cD)$ is a partition,
\begin{align}\label{e:sumcopy2}
N^{|\{t: \cS_t=+, t\not\in \cD\}|}
\sum_{\{j_t\in \bZ_{>0}: \cS_t=-,t\not\in \cD\}}Q_{\cK,\cS, \cD}((j_t: \cS_t=-, t\not\in \cD))\prod_{V\in \cP}\left.\left(\prod_{t\in V} \del p_{j_t}\right)(\ln F_N)\right|_{\bmp=1^N},
\end{align}
where $Q_{\cK,\cS, \cD}$ is a polynomial over $\{j_t: \cS_t=-, t\not\in \cD\}$. Say $\cP=\{V_1, V_2,\cdots, V_s\}$, then by Claim \ref{c:CLTup}, \eqref{e:sumcopy2} has an $N$-degree at most
\begin{align}\label{e:cltNd}
|\{t: \cS_t=+, t\not\in \cD\}|+\sum_{u=1}^s {\bf1}_{|V_u|=1}=
(k_1+k_2+\cdots+k_r)-|\{t:\cS_t=0\}|-2|\cD|-\sum_{u=1}^s {\bm1_{|V_u|\geq 2}}|V_u|,
\end{align}
if $|V_u|\leq 2$ for all $1\leq u\leq s$. And \eqref{e:sumcopy2} has an $N$-degree less than \eqref{e:cltNd}, if for some $1\leq u\leq s$, $|V_u|> 2$. 
Since $\cG(\cK, \cS, \cD, \cP)$ is connected, either $|\cD|\geq 1$ or $|V_u|\geq 2$ for some $1\leq u\leq s$. By Claim \ref{c:edgeineq}, we have
\begin{align*}
2|\cD|+\sum_{u=1}^s {\bm1_{|V_u|\geq 2}}|V_u|
\geq 1+|\cD|+\sum_{u=1}^s (|V_u|-1)\geq r,
\end{align*}
and the equality holds, only if $\cG(\cK, \cS, \cD, \cP)$ has one edge $\{1,2,\cdots,r\}$, and its other edges are all singleton. We conclude that $\cF_{(k_1,k_2,\cdots, k_r)}$ has an $N$-degree at most $k_1+k_2+\cdots+k_r-r$.

If $r\geq 3$,  either $\eqref{e:cltNd}<k_1+k_2+\cdots+k_r-r$, or $\eqref{e:cltNd}=k_1+k_2+\cdots+k_r-r$ and there exists some $1\leq u\leq s$ such that $\{\sigma_\cK(t): t\in V_u\}=\{1,2,\cdots, r\}$. If this is the case, by Claim \ref{c:CLTup}, we in fact have that \eqref{e:sumcopy2} has an $N$-degree less than $\eqref{e:cltNd}=k_1+k_2+\cdots+k_r-r$. Therefore, if $r\geq 3$, $\cF_{(k_1, k_2, \cdots,k_r)}$ has an $N$-degree less than $k_1+k_2+\cdots+k_r-r$.

If $r=2$, $\eqref{e:cltNd}=0$ only if $\cG(\cK, \cS,\cD, \cP)$ has exactly one edge $\{1,2\}$, and its other edges are all singleton. The edge $\{1,2\}$ comes either from the pairing $\cD$, or the partition $\cP$. In the first case,
$|\{t:\cS_t=0\}|=0$, $|V_u|=1$ for all $1\leq u\leq s$ and $\cD=\{\{t_1t_2\}\}$, where $\sigma_\cK(t_1)=1$ and $\sigma_\cK(t_2)=2$. The sum of such terms are given by
\begin{align}\label{e:dterm}
\left.\sum_{k\geq 1}\frac{k}{\theta}\del q_{k} (T^{k_1})_{00}\del q_{-k} (T^{k_2})_{00}\right|_{q_j=q_j(F_N)}.
\end{align}
In the second case, $|\{t:\cS_t=0\}|=0$, $|\cD|=0$ and exactly one of $V_1, V_2,\cdots, V_s$ equals $\{t_1,t_2\}$ with $\sigma_\cK(t_1)=1$ and $\sigma_\cK(t_2)=2$, and others are singleton. The sum of such terms are given by
\begin{align}\label{e:vterm}
\left.\sum_{k,l\geq 1}\frac{q_{k,l}(F_N)kl}{\theta^2}\del q_{k}(T^{k_1})_{00}\del q_{l}(T^{k_2})_{00}\right|_{q_j=q_j(F_N)}.
\end{align}
The claim \eqref{e:crossterm} follows from combining \eqref{e:dterm} and \eqref{e:vterm}.
\end{proof}

\begin{proof}[Proof of Theorem \ref{t:clt}]
It follows from Proposition \ref{p:cumulant} and \ref{p:crossterm}, 
the covariance $\cov\langle\xi_{PP}^{(k)}[\bmla], \xi_{PP}^{(l)}[\bmla] \rangle$ has an $N$-degree $0$;
for any 
$r\geq 3$ and $k_1,k_2,\cdots, k_r\geq 1$, the $r$-th cumulant $\kappa_{k_1,k_2,\cdots,k_r}(\xi_{PP}^{(k)}[\bmla],k=1,2,3\cdots)$ has an $N$-degree less than $0$. Thus the sequence of random measures $\mu_{PP}[\bmla]$ satisfies the central limit theorems in the sense of Definition \ref{def:LLNCLT}. Combining with Lemma \ref{l:change} and \ref{l:CLTtoCLT}, we conclude that the sequence of random measures $\mu[\bmla]$ also satisfies the central limit theorems. In the following we determine the covariance structure of moments of the sequence of random measures $\mu[\bmla]$.

We take derivative of both sides of \eqref{e:invertT} with respect to $q_i$,
\begin{align*}\begin{split}
\sum_{k\geq 0}\frac{\del q_i (T^k)_{00}}{z^{k+1}}
&=\del{q_i}(z-T)^{-1}_{00}
=\del q_i \exp\left(-\frac{1}{2\pi \ri}\oint_{|w|=1+\varepsilon/2}\ln(z-T(w))\frac{\rd w}{w}\right)\\
&=\left(\frac{1}{2\pi\ri}\oint_{|w|=1+\varepsilon/2}\frac{w^{i}}{z-T(w)}\frac{\rd w}{w} \right) (z-T)^{-1}_{00}.
\end{split}\end{align*}
Therefore, we have
\begin{align}\begin{split}\label{e:gfunction}
&\phantom{{}={}}\sum_{k_1,k_2\geq 1}\left(\sum_{i\geq 1}\frac{i}{\theta}\frac{\del q_{i} (T^{k_1})_{00}\del q_{-i} (T^{k_2})_{00}}{z_1^{k_1+1}z_2^{k_2+1}}+\sum_{i,j\geq 1}\frac{q_{i,j}(F_N)ij}{\theta^2}\frac{\del q_{i}(T^{k_1})_{00}\del q_{j}(T^{k_2})_{00}}{z_1^{k_1+1}z_2^{k_2+1}}\right)\\
&=\sum_{i\geq 1}\frac{i}{\theta}\del{q_{i}}(z_1-T)^{-1}_{00}\del{q_{-i}}(z_2-T)^{-1}_{00}
+\sum_{i,j\geq 1}\frac{q_{i,j}(F_N)ij}{\theta^2}\del{q_{i}}(z_1-T)^{-1}_{00}\del{q_{j}}(z_2-T)^{-1}_{00}\\
&=\left(\frac{1}{(2\pi \ri)^2}
\oint_{|w_1|=1+\varepsilon/3\atop|w_2|=1+\varepsilon/2}
\frac{\sum_{i\geq 1}iw_1^{i}w_2^{-i}/\theta +\sum_{i,j\geq 1}q_{i,j}(F_N)ijw_1^{i}w_2^{j}/\theta^2}{(z_1-T(w_1))(z_2-T(w_2))}\frac{\rd w_1}{w_1} \frac{\rd w_2}{w_2} \right)\\ 
&\phantom{{}={}}\times (z_1-T)^{-1}_{00}(z_2-T)^{-1}_{00}.
\end{split}\end{align}
We take the specialization $q_j=q_j(F_N)$ in \eqref{e:gfunction} and send $N\rightarrow \infty$. From Proposition \ref{p:crossterm}, the lefthand side becomes the generating function of $\lim_{N\rightarrow \infty}\cF_{(k_1,k_2)}/N^{k_1+k_2-2}$. For the righthand side of \eqref{e:gfunction}, on $|w_1|<|w_2|$, we have
\begin{align*}
&\sum_{i\geq 1}\frac{i}{\theta}w_1^{i-1}w_2^{-i-1}=\frac{1}{\theta}\frac{1}{(w_1-w_2)^2}.
\end{align*}
And $\sum_{i,j\geq1}q_{i,j}(F_N)ijw_1^{i-1}w_2^{j-1}=V_{F_N}(w_1,w_2)$ converges uniformly on $\{(w_1,w_2)\in \bC: |w_1|,|w_2|<1+\varepsilon\}$.
Therefore, by plugging $q_i=q_i(F_N)$, we have $T(w)|_{q_j=q_j(F_N)}=T_{F_N}(w)$, and
\begin{align*}\begin{split}
&\phantom{{}={}}\frac{1}{(2\pi \ri)^2}
\oint_{|w_1|=1+\varepsilon/3\atop|w_2|=1+\varepsilon/2}
\frac{\sum_{i\geq 1}iw_1^{i}w_2^{-i}/\theta +\sum_{i,j\geq 1}q_{i,j}(F_N)ijw_1^{i}w_2^{j}/\theta^2}{(z_1-T_{F_N}(w_1))(z_2-T_{F_N}(w_2))}
\frac{\rd w_1}{w_1}\frac{\rd w_2}{w_2}\\
&=\frac{1}{(2\pi \ri)^2}
\oint_{|w_1|=1+\varepsilon/3\atop|w_2|=1+\varepsilon/2}
\left(\frac{1}{\theta}\frac{1}{(w_1-w_2)^{2}} +\frac{V_{F_N}(w_1,w_2)}{\theta^{2}}\right)\frac{\rd w_1\rd w_2}{(z_1-T_{F_N}(w_1))(z_2-T_{F_N}(w_2))}.
\end{split}\end{align*}
Therefore, we get
\begin{align*}\begin{split}
&\lim_{N\rightarrow \infty}\sum_{k_1,k_2\geq 1}\frac{\cF_{(k_1,k_2)}}{N^{k_1+k_2-2}}\frac{1}{z_1^{k_1+1}z_2^{k_2+1}}
=
\lim_{N\rightarrow \infty}
(z_1-T_{F_N})^{-1}_{00}(z_2-T_{F_N})^{-1}_{00}\\
&\left(\frac{1}{(2\pi \ri)^2}
\oint_{|w_1|=\varepsilon/3\atop|w_2|=\varepsilon/2}
\left(\frac{1}{\theta}\frac{1}{(w_1-w_2)^{2}} +\frac{V_{F_N}(w_1+1,w_2+1)}{\theta^{2}}\right)\frac{\rd w_1\rd w_2}{(z_1-T_{F_N}(w_1+1))(z_2-T_{F_N}(w_2+1))}\right).
\end{split}\end{align*}
Thanks to Proposition \ref{p:cumulant} and \eqref{e:formalsum},
\begin{align*}\begin{split}
&\phantom{{}={}}\lim_{N\rightarrow\infty}\cov\langle\Delta m_{PP}[\bmla](z_1),\Delta m_{PP}[\bmla](z_2)\rangle=\lim_{N\rightarrow \infty}\sum_{k_1,k_2\geq 1}\frac{\cov\langle \xi_{PP}^{(k_1)}[\bmla],\xi_{PP}^{(k_2)}[\bmla]\rangle}{z_1^{k_1+1}z_2^{k_2+1}}\\
&=\lim_{N\rightarrow \infty}(z_1+1)(z_2+1)\sum_{k_1,k_2\geq 1}\frac{\cF_{(k_1,k_2)}}{N^{k_1+k_2-2}z_1^{k_1+1}z_2^{k_2+1}}
=e^{-m_\mu(z_1)-m_\mu(z_2)}\\
&\times\left([w_1^{-1}w_2^{-1}]
\left(\sum_{a\geq 1}\frac{a}{\theta}w_1^{a-1}w_2^{-a-1} +\frac{V_{\mu}(w_1+1,w_2+1)}{\theta^{2}}\right)\frac{1}{(z_1-T_{\mu}(w_1+1))(z_2-T_{\mu}(w_2+1))} \right). 
\end{split}\end{align*}
From \eqref{e:Deltam}, the covariance structure of $\xi^{(k)}[\bmla], k=1,2,3,\cdots$ is uniquely determined by the covariance structure of $\xi^{(k)}_{PP}[\bmla], k=1,2,3,\cdots$.
\begin{align*}\begin{split}
&\phantom{{}={}}\lim_{N\rightarrow \infty}\cov\langle\Delta m[\bmla](z_1),\Delta m[\bmla](z_2)\rangle
=e^{-m_\mu(z_1)}e^{-m_\mu(z_2)}\lim_{N\rightarrow \infty}\cov\langle\Delta m_{PP}[\bmla](z_1),\Delta m_{PP}[\bmla](z_2)\rangle\\
&=\sum_{k_1,k_2\geq 1}\frac{1}{z_1^{k_1+1}z_2^{k_2+1}}
\left([w_1^{-1}w_2^{-1}]
\left(\sum_{a\geq 1}\frac{a}{\theta}w_1^{a-1}w_2^{-a-1} +\frac{V_{\mu}(w_1+1,w_2+1)}{\theta^{2}}\right)T_{\mu}^{k_1}(w_1+1)T^{k_2}_{\mu}(w_2+1) \right),
\end{split}\end{align*}
and the claim \eqref{e:cltcov} follows by comparing the coefficients of $1/z_1^{k_1+1}z_2^{k_2+1}$ on both side.
\end{proof}

 \subsection{Central limit theorems for several levels}\label{subs:mulCLT}
 
In this section, let $\{M_N\}_{N\geq 1}$, $\{\fm_N^{(1)}\}_{N\geq 1}, \{\fm_N^{(2)}\}_{N\geq 1}, \{\fm_N^{(3)}\}_{N\geq 1},\cdots$ be a sequence of probability measures on $\bY(N)$, with Jack generating functions $F_N(\bmp;\theta)$,$g_N^{(1)}(\bmp;\theta), g_N^{(2)}(\bmp;\theta), g_N^{(3)}(\bmp;\theta), \cdots $.  We recall the transition probabilities $\fp_N^{(t)}(\bmla,\bmmu)$, the Markov chain $\bmla^{(0)}, \bmla^{(1)}, \bmla^{(2)},\cdots$, the Jack generating functions $H_N^{(t)}(\bmp;\theta)$ from Section \ref{s:mulclt}.
 
By the same argument as in Lemma \ref{l:CLTtoCLT}, $\{\mu_{PP}[\bmla^{(t)}]\}_{0\leq t\leq T}$ satisfies the law of large numbers and the central limit theorems if and only if $\{\mu_{PP}[\bmla^{(t)}]\}_{0\leq t\leq T}$ satisfies the central limit theorems. If this is the case, 
for any $0\leq t\leq T$, $\mu[\bmla^{(t)}]$ converges weakly to $\mu^{(t)}$ and $\mu_{PP}[\bmla^{(t)}]$ converges weakly to $\mu_{PP}^{(t)}$. We denote
\begin{align*}\begin{split}
\Delta m_{PP}[\bmla^{(t)}](z)
&=N\int \frac{\rd \mu_{PP}[\bmla^{(t)}](x)-\bE\rd \mu_{PP}[\bmla^{(t)}](x)}{z-x}\\
\Delta m[\bmla^{(t)}](z)
&=N\int \frac{\rd \mu[\bmla^{(t)}](x)-\bE\rd \mu[\bmla^{(t)}](x)}{z-x}.
\end{split}\end{align*}
The covariance structure of moments of  $\{\mu[\bmla^{(t)}]\}_{0\leq t\leq T}$ is explicitly given by the covariance structure of moments  of  $\{\mu_{PP}[\bmla^{(t)}]\}_{0\leq t\leq T}$,
\begin{align*}
e^{m_{\mu^{(s)}}(z_1)}e^{m_{\mu^{(t)}}(z_2)}\cov\langle\Delta m[\bmla^{(s)}](z_1),\Delta m[\bmla^{(t)}](z_2)\rangle=\cov\langle\Delta m_{PP}[\bmla^{(s)}](z_1),\Delta m_{PP}[\bmla^{(t)}](z_2)\rangle+o(1),
\end{align*}
where $m_{\mu^{(s)}}(z)$ and $m_{\mu^{(t)}}(z)$ are the Stieltjes transforms of $\mu^{(s)}$ and $\mu^{(t)}$ respectively.

\begin{theorem}\label{t:mulmoment}
For any positive integer $r\geq 1$, and times $0\leq t_1\leq t_2\leq \cdots\leq t_r$, the joint moments of the Perelomov-Popov measures $\mu_{PP}[\bmla^{(t_1)}], \mu_{PP}[\bmla^{(t_2)}],\cdots, \mu_{PP}[\bmla^{(t_r)}]$ satisfy
\begin{align}\begin{split}\label{e:mulmoment}
&\bE\left[\prod_{j=1}^r\int x^{k_j}\rd \mu_{PP}[\bmla^{(t_j)}](x) \right]=\left(\frac{I^{(k_r)}}{N^{k_r}}+\frac{I^{(k_r+1)}}{N^{k_r+1}}\right)
\left(\prod_{s=t_{r-1}+1}^{t_r} g_N^{(s)}\right)\left(\frac{I^{(k_{r-1})}}{N^{k_{r-1}}}+\frac{I^{(k_{r-1}+1)}}{N^{k_{r-1}+1}}\right)\\
&
\left(\prod_{s=t_{r-2}+1}^{t_{r-1}} g_N^{(s)}\right)
\cdots \left(\frac{I^{(k_{2})}}{N^{k_{2}}}+\frac{I^{(k_{2}+1)}}{N^{k_{2}+1}}\right)
\left(\prod_{s=t_1+1}^{t_{2}} g_N^{(s)}\right)
\left.\left(\frac{I^{(k_{1})}}{N^{k_{1}}}+\frac{I^{(k_{1}+1)}}{N^{k_{1}+1}}\right) \left(\prod_{s=1}^{t_{1}} g_N^{(s)}\right)
F_N\right|_{\bmp=1^N}
.
\end{split}\end{align}
\end{theorem}
\begin{proof}
By the definition \eqref{e:deffp} of the coefficients $\fp_N^{(t)}(\bmla,\bmmu)$, we have for $\bmla,\bmmu\in \bY(N)$, 
\begin{align*}
g_N^{(t)}(\bmp;\theta)\frac{J_\bmla(\bmp;\theta)}{J_\bmla(1^N;\theta)}
=\sum_{\bmmu\in \Y(N)}\fp_N^{(t)}(\bmla,\bmmu)
\frac{J_\bmmu(\bmp;\theta)}{J_\bmmu(1^N;\theta)}+\Sp(J_\bmla(\bmp;\theta): \bmla\not\in \bY(N)).
\end{align*}
The linear subspace $\Sp(J_\bmla(\bmp;\theta): \bmla\not\in \bY(N))$ is an ideal of the algebra of symmetric functions (in infinitely many variables). Moreover, the Nazarov-Sklyanin Operators preserve the subspace $\Sp(J_\bmla(\bmp;\theta): \bmla\not\in \bY(N))$.
It follows by induction that for any $1\leq u\leq r$,
\begin{align}\begin{split}\label{e:mulmoment2}
&\phantom{{}={}} \left(\frac{I^{(k_{u})}}{N^{k_{u}}}+\frac{I^{(k_{u}+1)}}{N^{k_{u}+1}}\right)
\left(\prod_{s=t_{u-1}+1}^{t_{u}} g_N^{(s)}\right)\cdots
\left(\frac{I^{(k_{1})}}{N^{k_{1}}}+\frac{I^{(k_{1}+1)}}{N^{k_{1}+1}}\right) \left(\prod_{s=1}^{t_{1}} g_N^{(s)}\right)
F_N
\\
&=\sum_{\bmla_0,\bmla_1,\cdots,\bmla_{t_u}\in \bY(N)}M_{N}(\bmla_0)\prod_{t=1}^{t_u}\fp(\bmla_{t-1},\bmla_{t})\prod_{j=1}^u \int x^{k_j}\rd \mu_{PP}[\bmla^{(t_j)}](x)
\frac{J_{\bmla_{t_u}}(\bmp;\theta)}{J_{\bmla_{t_u}}(1^N;\theta)}\\
&+\Sp(J_\bmla(\bmp;\theta): \bmla\not\in \bY(N))
\end{split}\end{align}
The claim \eqref{e:mulmoment} follows by taking $\bmp=1^N$ in \eqref{e:mulmoment2} and $u=r$.
\end{proof}

Fix a positive integer $r$, and indices $k_1,k_2,\cdots,k_r\geq 1$, we recall the $r$-tuple $\cK=(k_1,k_2,\cdots,k_r)$, the set of sign patterns $\cS(\cK)$, the set of pairings $\cD(\cS)$, and the set of conditions $\cC(\cK, \cS,\cD)$ from Section \ref{s:Tl}. For a given $r$-tuple $\cK$, the set of sign patterns $\cS(\cK)$, and the set of pairings $\cD(\cS)$, we define the set of colored partitions $\cP(\cK, \cS, \cD)$, which is a partition $\cP\in \cP(\{t:\cS_t=-,t\not\in \cD\})$, with a map $\sigma_{\cP}: \cP\mapsto \{1,2,\cdots,r\}$ such that for any $V\in \cP$ and $t\in V$, $\sigma_\cK(t)\geq \sigma_{\cP}(V)$. For a given $r$-tuple $\cK$, a sign patterns $\cS\in \cS(\cK)$, a  pairing $\cD\in\cD(\cS)$, and a colored partition $\cP\in \cP(\cK, \cS, \cD)$, we define the multi-edge hypergraph $\cG(\cK,\cS,\cD, \cP)$ on the vertex set $\{1,2,\cdots, r\}$ the same as in Section \ref{s:Tl}.

The same as Theorem \ref{t:decompose}, we have
\begin{proposition}\label{p:muldecompose}
For any positive integer $r$, indices $k_1,k_2,\cdots,k_r\geq 1$, and times $0\leq t_1\leq t_2\leq\cdots\leq t_r$,
\begin{align*}\begin{split}
&\phantom{{}={}}\left. \frac{I^{(k_{r})}}{N^{k_{r}}}
\left(\prod_{s=t_{r-1}+1}^{t_{r}} g_N^{(s)}\right)\cdots
\frac{I^{(k_{1})}}{N^{k_{1}}}\left(\prod_{s=1}^{t_{1}} g_N^{(s)}\right)
F_N\right|_{\bmp=1^N}
\end{split}\end{align*}
is a sum over $\cE(\cT, \cK,\cS,\cD, \cP)$, 
\begin{align}\label{e:mulsum}
N^{|\{t: \cS_t=+, t\not\in \cD\}|}\sum_{\{j_t\in \bZ_{>0}: \cS_t=-,t\not\in \cD\}}Q_{\cK,\cS, \cD}((j_t: \cS_t=-, t\not\in \cD))\prod_{V\in \cP}\left.\left(\prod_{t\in V}\del p_{j_t}\right)(\ln H_N^{(t_{\sigma_\cP(V)})})\right|_{\bmp=1^N},
\end{align}
where $\cT=(t_1,t_2,\cdots,t_r)$, $\cK=(k_1,k_2,\cdots,k_r)$, $\cS\in \cS(\cK)$ is a sign pattern, $\cD\in \cD(\cS)$ is a pairing, $\cP\in \cP(\cK, \cS, \cD)$ is a colored partition,
and 
$Q_{\cK,\cS, \cD}$ is a polynomial over $\{j_t: \cS_t=-, t\not\in \cD\}$ with degree  
\begin{align*}
\deg(Q_{\cK,\cS, \cD})=k_1+k_2+\cdots+k_r-r.
\end{align*}
\end{proposition}

For any positive integer $r$, indices $k_1,k_2,\cdots,k_r\geq 1$, and times $0\leq \tau_1\leq \tau_2\leq\cdots\leq \tau_r$, we take $\cT=(\lfloor N\tau_1\rfloor, \lfloor N\tau_2 \rfloor, \cdots, \lfloor N\tau_r\rfloor)$, and define 
\begin{align*}
\cF_{(k_1,k_2,\cdots,k_r)}^{(\tau_1,\tau_2,\cdots,\tau_r)}\deq \sum_{\cS\in\cS(\cK), \cD\in \cD(\cS), \cP\in \cP(\cK,\cS, \cD)\atop\cG(\cK,\cS, \cD, \cP) \text{ is connected}}\cE(\cT, \cK,\cS, \cD, \cP).
\end{align*}

The same as Proposition \ref{p:cumulant}, we have
\begin{proposition}\label{p:mulcumulant}
For any positive integer $r\geq 2$, indices $k_1,k_2,\cdots,k_r\geq 1$, and times $0\leq t_1\leq t_2\leq\cdots\leq t_r$, as a formal power series of $1/z_1, 1/z_2,\cdots, 1/z_r$,
\begin{align*}\begin{split}
&\phantom{{}={}}\sum_{k_1,k_2,\cdots,k_r\geq 1}\frac{\kappa(\xi_{PP}^{(k_1)}[\bmla^{(\lfloor N\tau_1\rfloor)}],\xi_{PP}^{(k_2)}[\bmla^{(\lfloor N\tau_2\rfloor)}], \cdots, \xi_{PP}^{(k_r)}[\bmla^{(\lfloor N\tau_r\rfloor)}])}{z_1^{k_1+1}z_2^{k_2+1}\cdots z_r^{k_r+1}}\\
&=\prod_{j=1}^{r}(z_j+1)\sum_{k_1,k_2,\cdots,k_r\geq 1}\frac{\cF_{(k_1,k_2,\cdots, k_r)}^{(\tau_1,\tau_2,\cdots,\tau_r)}}{N^{k_1+k_2+\cdots+k_r-r}z_1^{k_1+1}z_2^{k_2+1}\cdots z_r^{k_r+1}}.
\end{split}\end{align*}
\end{proposition}
The same as Proposition \ref{p:crossterm}, we have 
\begin{proposition}\label{p:mulcrossterm}
For any indices $k_1,k_2\geq 1$ and times $0\leq \tau_1\leq \tau_2$, $\cF_{(k_1,k_2)}^{(\tau_1,\tau_2)}$ has an $N$-degree at most $k_1+k_2-2$,
\begin{align*}\begin{split}
\lim_{N\rightarrow \infty}\frac{\cF^{(\tau_1,\tau_2)}_{(k_1,k_2)}}{N^{k_1+k_2-2}}
=\lim_{N\rightarrow \infty}
&\left(\sum_{k\geq 1}\frac{k}{\theta} \left.\del q_{k} (T^{k_2})_{00}\right|_{q_j=q_j(H_N^{(\lfloor N\tau_2\rfloor)})}\left. \del q_{-k} (T^{k_1})_{00}\right|_{q_{j}=q_j(H_N^{(\lfloor N\tau_1\rfloor)})}\right.\\
&\left.+\sum_{k,l\geq 1}\frac{q_{k,l}(H_N^{(\lfloor N\tau_1\rfloor)})kl}{\theta^2}\left.\del q_{k}(T^{k_2})_{00}\right|_{q_j=q_j(H_N^{(\lfloor N\tau_2\rfloor)})}
\left.\del q_{l}(T^{k_1})_{00}\right|_{q_j=q_j(H_N^{(\lfloor N\tau_1\rfloor)})}\right).
\end{split}\end{align*}
 For any positive integer $r\geq 3$, indices $k_1,k_2,\cdots, k_r>0$,  and times $0\leq \tau_1\leq \tau_2\leq \cdots\leq \tau_r$, $\cF^{(\tau_1,\tau_2,\cdots,\tau_r)}_{(k_1, k_2, \cdots,k_r)}$ has an $N$-degree less than $k_1+k_2+\cdots+k_r-r$.
\end{proposition}
Theorem \ref{t:mulclt} follows from combining Proposition \ref{p:mulcumulant} and \ref{p:mulcrossterm}.
 
 \subsection{Necessary conditions for the central limit theorems}\label{subs:invCLT}
 
 In Proposition \ref{p:Fdegreeclt}, we show that the central limit theorems imply certain estimates on the quantities $\cF_{(k_1,k_2,\cdots, k_r)}$ as defined in \eqref{e:defFk}. Since $\cF_{(k_1,k_2,\cdots, k_r)}$ can be expressed as a finite sum of product of terms in the form
\begin{align}\label{e:dpp}
\sum_{i_1,i_2,\cdots, i_r\geq 1}{i_1\choose k_1}{i_2\choose k_2}\cdots {i_r\choose k_r} \left.\frac{\del^r \ln F_N(\bmp;\theta)}{\del p_{i_1}\del p_{i_2}\cdots \del p_{i_r}}\right|_{\bmp=1^N}.
\end{align}
By a delicate induction arguments, the estimates on the quantities $\cF_{(k_1,k_2,\cdots, k_r)}$ transfer to the estimates on \eqref{e:dpp}, and Theorems \ref{t:invclt} follow.


\begin{proposition}\label{p:Fdegreeclt}
Under the assumption of Theorem \ref{t:invclt}, for any indices $k\geq 1$, $\cF_{(k)}$ has an $N$-degree at most $k$. For any indices $k_1,k_2\geq 1$, $\cF_{(k_1,k_2)}$ has an $N$-degree at most $k_1+k_2-2$,
 For any positive integer $r\geq 3$, and indices $k_1,k_2,\cdots, k_r\geq 1$, $\cF_{(k_1, k_2, \cdots,k_r)}$ has an $N$-degree less than $k_1+k_2+\cdots+k_r-r$.
\end{proposition}

\begin{proof}[Proof of Proposition \ref{p:Fdegreeclt}]
Under the assumption of Theorem \ref{t:invclt}, thanks to Lemma \ref{l:CLTtoCLT}, the sequence of measures $\mu_{PP}[\bmla]$ satisfies the central limit theorems in the sense of Definition \ref{def:LLNCLT}. 
By Theorem \ref{t:moment}, the joint moments of the Perelomov-Popov measures satisfy
\begin{align*}
\frac{\cF_{(k)}}{N^k}=\frac{I^{(k)}F_N(\bmp;\theta)|_{\bmp=1^N}}{N^{k}}=\sum_{l=1}^k(-1)^{k-l-1}
\bE\left[\int x^{l}\rd \mu_{PP}(x)\right].
\end{align*}
By our definition \eqref{def:LLNCLT}, the limit exists
\begin{align*}
\lim_{N\rightarrow \infty}\frac{\cF_{(k)}}{N^k}=
\sum_{l=1}^{k-1}(-1)^{k-l-1}
\lim_{N\rightarrow\infty}\bE\left[\int x^{l}\rd \mu_{PP}(x)\right],
\end{align*}
It follows that $\cF_{(k)}$ has an $N$-degree at most $k$. 

Moreover, for any $k_1,k_2\geq 1$, $\kappa_{k_1,k_2}(\xi_{PP}^{(k)}[\bmla],k=1,2,3,\cdots)$ has an $N$-degree $0$, and for any $r\geq 3$, and  $k_1,k_2,\cdots,k_r\geq 1$, $\kappa_{k_1,k_2,\cdots,k_r}(\xi_{PP}^{(k)}[\bmla], k=1,2,3,\cdots)$ has an $N$-degree less than $0$. From Proposition \ref{p:cumulant}, for any $r\geq 2$ and $k_1,k_2,\cdots, k_r\geq 1$, we have
\begin{align*}
\frac{\cF_{(k_1,k_2,\cdots,k_r)}}{N^{k_1+k_2+\cdots+k_r-r}}=\sum_{\{1\leq l_i\leq k_i-1: 1\leq i\leq r\}}(-1)^{k_1+k_2+\cdots+k_r-l_1-l_2-\cdots-l_r-r}
\kappa_{l_1,l_2,\cdots,l_r}(\xi_{PP}^{(k)}[\bmla],k=1,2,3\cdots).
\end{align*}
It follows that for any indices $k_1,k_2\geq 1$, $\cF_{(k_1,k_2)}$ has an $N$-degree at most $k_1+k_2-2$,
 For any positive integer $r\geq 3$, and indices $k_1,k_2,\cdots, k_r\geq 1$, $\cF_{(k_1, k_2, \cdots,k_r)}$ has an $N$-degree less than $k_1+k_2+\cdots+k_r-r$.

\end{proof}

\begin{proposition}\label{p:decompQ}
For any $k\geq 1$ and $\cS\in \{-,+\}^k$ with $s=|\{t: \cS_t=-\}|$, $Q_{(k),\cS,\emptyset}(i_1,i_2,\cdots,i_s)$, defined as in \eqref{e:sum}, is a sum of terms in the following form
\begin{align}\label{e:form}
a_{(\bml)}{i_1\choose l_1}{i_2\choose l_2}\cdots{i_s\choose l_s},
\end{align}
parametrized by $\bml=(l_1,l_2,\cdots, l_s)$ with $l_1\geq l_2\geq \cdots\geq l_s\geq 1$. Moreover, $|\bml|=l_1+l_2+\cdots+l_s\leq k-1$. 
The only term of the form \eqref{e:form} with $s=1$ appears when $\cS=(+,+,\cdots,+,-)$, and is given by
\begin{align*}
\frac{k-1}{\theta}{i\choose k-1}.
\end{align*}
\end{proposition}
\begin{proof}
From Proposition \ref{t:decompose}, $Q_{(k),\cS,\emptyset}$ is a polynomial of degree $k-1$.
Therefore $Q_{(k),\cS,\emptyset}$ is a sum of terms in the form \eqref{e:form}. If $s=|\{t:\cS_t=-\}|=1$, then $|\{t:\cS_t=+\}|=k-1$, and $\cS=(+,+,\cdots,+,-)$. If this is the case $\tilde Q_{\cS,\cD}=j_k/\theta$. By the definition \eqref{e:existtQ} of $Q_{(k),\cS,\emptyset}$
\begin{align}\label{e:specialQ}
Q_{(k),\cS,\emptyset}(j_k)
=\sum_{j_1,j_2,\cdots,j_{k-1}\geq 1\atop j_1+j_2+\cdots+j_{k-1}=j_k}\frac{j_k}{\theta}
={j_k-1\choose k-2}\frac{j_k}{\theta}=\frac{k-1}{\theta}{j_k\choose k-1}.
\end{align}
This finishes the proof of Proposition \ref{p:decompQ}.
\end{proof}

\begin{proof}[Proof of Theorem \ref{t:invclt}]
We use the same notations  as in section \ref{s:Tl}. In the following, we prove by induction that for any positive integer $r\geq 1$ and indices $\bml=(l_1,l_2,\cdots,l_r)$, $l_1\geq l_2\geq \cdots l_{r}\geq 1$,
\begin{align}\label{e:derlln}
\cL_{\bml}(F_N)\deq\sum_{i_1,i_2, \cdots, i_{r}\geq 1}{i_1\choose l_1}{i_2\choose l_2}\cdots {i_{r}\choose l_r}\left.\frac{\del^{r} \ln F_N}{\del p_{i_1}\del p_{i_2}\cdots\del p_{i_{r}}}\right|_{\bmp=1^N},
\end{align}
has an $N$-degree at most $1$ if $r=1$, has an $N$-degree at most $0$ if $r=2$, and has an $N$-degree less than $0$ if $r\geq 3$. In the following we will view $\bml$ as a Young diagram, with partial order such that $\bml'\preceq \bml$ if $|\bml'|=|\bml|$ and $l'_1+l'_2+\cdots l'_i\leq l_1+l_2+\cdots+l_i$ for all $i$. When $|\bml|=0$, the above statement is trivial. We assume that the above statement holds for $|\bml|<k$ and prove the statement for $|\bml|=k$. There are three cases either $r=1$, $r=2$ or $r\geq 3$.

For $r=1$, from Theorem \ref{t:decompose}, we have that $\cF_{(k+1)}/N^{k+1}$
is a sum over $\cE(\cK,\cS, \cD, \cP)$ 
\begin{align}\label{e:sumlln}
\frac{N^{|\{t: \cS_t=+, t\not\in \cD\}|}}{N^{k+1}}\sum_{\{j_t\in \bZ_{>0}: \cS_t=-,t\not\in \cD\}}Q_{\cK,\cS, \cD}((j_t: \cS_t=-, t\not\in \cD))\prod_{V\in \cP}\left.\left(\prod_{t\in V}\del p_{j_t}\right)(\ln F_N)\right|_{\bmp=1^N},
\end{align}
where $\cK=(k+1)$, $\cS\in \cS(\cK)$ is a sign pattern, $\cD\in \cD(\cS)$ is a pairing, $\cP\in \cP(\cS, \cD)$ is a partition, $Q_{\cK,\cS, \cD}$ is a polynomial over $\{j_t: \cS_t=-, t\not\in \cD\}$ with degree $k$. From Theorem \ref{t:decompose}, the polynomial $Q_{\cK,\cS, \cD}$ has the factor $\prod_{t:\cS_t=-, t\not\in \cD}j_t$. If $\cP=\{V_1, V_2,\cdots, V_s\}$, 
the polynomial $Q_{\cK,\cS, \cD}$ decomposes into a sum of product of polynomials in variables $\{j_t: t\in V_u\}$ for $u=1,2,\cdots, s$. 
Therefore the expression \eqref{e:sumlln} is a finite sum of terms in the following form
\begin{align}\label{e:decomposeal}
\frac{N^{|\{t: \cS_t=+, t\not\in \cD\}|}}{N^{k+1}}
a_{(\bml_1,\bml_2,\cdots, \bml_s)}
\cL_{\bml_1}(F_N)\cL_{\bml_2}(F_N)\cdots \cL_{\bml_s}(F_N),
\end{align}
where the constant $a_{(\bml_1,\bml_2,\cdots, \bml_s)}$ is independent of $N$, the length $\ell(\bml_u)=|V_u|$ for $u=1,2,\cdots, s$, and $|\bml_1|+|\bml_2|+\cdots +|\bml_s|\leq k$.
If $s\geq 2$, we have $|\bml_1|,|\bml_2|,\cdots|\bml_s|\leq k-1$, therefore by our induction assumption, the $N$-degree of $\cL_{\bml_u}(F_N)$ is at most $1$ for any $1\leq u\leq s$, and the $N$-degree of \eqref{e:decomposeal} is at most $|\{t:\cS=+,t\not\in \cD\}|+s-(k+1)\leq |\{t:\cS=+,t\not\in \cD\}|+|\{t:\cS=-,t\not\in \cD\}|-(k+1)\leq 0$. In this case \eqref{e:sumlln} has an $N$-degree at most $0$.
If $s=1$, then $\cP=\{\{t:\cS_t=-,t\not\in \cD\}\}$, and \eqref{e:decomposeal} simplifies to
\begin{align}\label{e:decomposeal2} 
\frac{N^{|\{t: \cS_t=+, t\not\in \cD\}|}}{N^{k+1}}
a_{(\bml)}
\cL_{\bml}(F_N),
\end{align}
where $\ell(\bml)=|\{t:\cS_t=-,t\not\in \cD\}|$ and $|\bml|\leq k$.
If $|\bml|\leq k-1$, the same argument as above, by our induction assumption, \eqref{e:decomposeal2} has an $N$-degree at most $0$. Therefore we have
\begin{align}\label{e:decomposeal3}
\eqref{e:sumlln}=\frac{N^{|\{t: \cS_t=+, t\not\in \cD\}|}}{N^{k}}
\sum_{|\bml|=k}a_{(\bml)}
\frac{\cL_{\bml}(F_N)}{N}
+\{\text{terms with $N$-degree at most $0$}\}.
\end{align}
The coefficient $N^{|\{t: \cS_t=+, t\not\in \cD\}|}/N^k=1$ only if $\cD=\emptyset$ and $\cS=(+,+,\cdots,+,-)$. If this is the case, from Proposition \ref{p:decompQ}, we have
\begin{align}\label{e:l=1}
\eqref{e:decomposeal3}=\frac{k}{\theta}\frac{\cL_{(k)}(F_N)}{N}
+
\sum_{|\bml|=k,\bml\prec (k)}a_{(\bml)}
\frac{\cL_{\bml}(F_N)}{N}
+\{\text{terms with $N$-degree at most $0$}\}.
\end{align}

For $r=2$, and indices $k_1\geq k_2\geq 2$, such that $k_1+k_2=k+2$, it follows from Proposition \ref{t:decompose}, we have that $\cF_{(k_1,k_2)}/N^{k_1+k_2-2}$ 
is a sum over $\cE(\cK,\cS, \cD, \cP)$ 
\begin{align}\label{e:sumcopy3}
\frac{N^{|\{t: \cS_t=+, t\not\in \cD\}|}}{N^{k_1+k_2-2}}\sum_{\{j_t\in \bZ_{>0}: \cS_t=-,t\not\in \cD\}}Q_{\cK,\cS, \cD}((j_t: \cS_t=-, t\not\in \cD))\prod_{V\in \cP}\left.\left(\prod_{t\in V}\del p_{j_t}\right)(\ln F_N)\right|_{\bmp=1^N},
\end{align}
where $\cK=(k_1,k_2)$, $\cS\in \cS(\cK)$ is a sign pattern, $\cD\in \cD(\cS)$ is a pairing, $\cP\in \cP(\cS, \cD)$ is a partition, $Q_{\cK,\cS, \cD}$ is a polynomial over $\{j_t: \cS_t=-, t\not\in \cD\}$ with degree $k_1+k_2-2=k$, and the hypergraph $\cG(\cK, \cS, \cD, \cP)$ is connected. If $\cP=\{V_1, V_2,\cdots, V_s\}$, 
the polynomial $Q_{\cK,\cS, \cD}$ decomposes into a sum of product of polynomials in variables $\{j_t: t\in V_u\}$ for $u=1,2,\cdots, s$. 
Therefore the expression \eqref{e:sumcopy3} is a finite sum of terms in the following form
\begin{align}\label{e:decomposeal4}
\frac{N^{|\{t: \cS_t=+, t\not\in \cD\}|}}{N^{k}}
a_{(\bml_1,\bml_2,\cdots, \bml_s)}
\cL_{\bml_1}(F_N)\cL_{\bml_2}(F_N)\cdots \cL_{\bml_s}(F_N),
\end{align}
where the constant $a_{(\bml_1,\bml_2,\cdots, \bml_s)}$ is independent of $N$, the length $\ell(\bml_u)=|V_u|$ for $u=1,2,\cdots, s$, and $|\bml_1|+|\bml_2|+\cdots +|\bml_s|\leq k$.
If $s\geq 2$, we have $|\bml_1|,|\bml_2|,\cdots|\bml_s|\leq k-1$, therefore by our induction assumption, for $1\leq u\leq s$, the $N$-degree of $\cL_{\bml_u}(F_N)$ is at most $1$ if $|V_u|=1$; is at most $0$ if $|V_u|\geq 2$. Therefore the $N$-degree of \eqref{e:decomposeal4} is at most
\begin{align*}
|\{t: \cS_t=+, t\not\in \cD\}|+\sum_{u=1}^s 1_{\{|V_u|=1\}}-k
=2-2|\cD|-\sum_{u=1}^s(|V_u|-1_{\{|V_u|=1\}})\leq 0,
\end{align*}
where we used that the hypergraph $\cG(\cK,\cS, \cD, \cP)$ is connected, so either $|\cD|\geq1$ or $|V_u|\geq 2$ for some $1\leq u\leq s$. Therefore in this case, \eqref{e:sumcopy3} has an $N$-degree at most $0$. If $s=1$, then $\cP=\{\{t:\cS_t=-,t\not\in \cD\}\}$, and \eqref{e:decomposeal4} simplifies to
\begin{align}\label{e:decomposeal5} 
\frac{N^{|\{t: \cS_t=+, t\not\in \cD\}|}}{N^{k}}
a_{(\bml)}
\cL_{\bml}(F_N),
\end{align}
where $\ell(\bml)=|\{t:\cS_t=-,t\not\in \cD\}|$ and $|\bml|\leq k$.
If $|\bml|\leq k-1$, the same argument as above, by our induction assumption, \eqref{e:decomposeal5} has an $N$-degree at most $0$. Therefore we have
\begin{align}\label{e:decomposeal6}
\eqref{e:sumcopy3}=\frac{N^{|\{t: \cS_t=+, t\not\in \cD\}|}}{N^{k}}
\sum_{|\bml|=k}a_{(\bml)}
\cL_{\bml}(F_N)
+\{\text{terms with $N$-degree at most $0$}\}.
\end{align}
The coefficient $N^{|\{t: \cS_t=+, t\not\in \cD\}|}/N^k=1$ only if $\cD=\emptyset$, $\cS=(\cS^{(1)}, \cS^{(2)})$ and $\cS^{(1)}=(+,+,\cdots,+,-)$, $\cS^{(2)}=(+,+,\cdots,+,-)$. If this is the case, the polynomial $Q_{\cK,\cS, \emptyset}$ splits
\begin{align*}
Q_{\cK,\cS, \emptyset}((j_t: \cS_t=-))
=Q_{(k_1),\cS^{(1)}, \emptyset}((j_t: \cS^{(1)}_t=-))Q_{(k_2),\cS^{(2)}, \emptyset}((j_t: \cS^{(2)}_t=-)),
\end{align*}
and Proposition \ref{p:decompQ} implies that
\begin{align}\begin{split}\label{e:l=2}
\eqref{e:decomposeal6}
&=\frac{(k_1-1)(k_2-1)}{\theta^2}\cL_{(k_1-1,k_2-1)}(F_N)\\
&+
\sum_{|\bml|=k,\bml\prec (k_1-1,k_2-1)}a_{(\bml)}
\cL_{\bml}(F_N)
+\{\text{terms with $N$-degree at most $0$}\}.
\end{split}\end{align}

For $r\geq 3$, and indices $k_1\geq k_2\geq \cdots\geq k_r\geq 2$, such that $k_1+k_2+\cdots+k_r=k+r$, it follows from Proposition \ref{t:decompose}, we have that $\cF_{(k_1,k_2,\cdots, k_r)}/N^{k_1+k_2+\cdots+k_r-r}$ 
is a sum over $\cE(\cK,\cS, \cD, \cP)$ 
\begin{align}\label{e:copy10}
\frac{N^{|\{t: \cS_t=+, t\not\in \cD\}|}}{N^{k_1+k_2+\cdots+k_r-r}}\sum_{\{j_t\in \bZ_{>0}: \cS_t=-,t\not\in \cD\}}Q_{\cK,\cS, \cD}((j_t: \cS_t=-, t\not\in \cD))\prod_{V\in \cP}\left.\left(\prod_{t\in V}\del p_{j_t}\right)(\ln F_N)\right|_{\bmp=1^N},
\end{align}
where $\cK=(k_1,k_2,\cdots,k_r)$, $\cS\in \cS(\cK)$ is a sign pattern, $\cD\in \cD(\cS)$ is a pairing, $\cP\in \cP(\cS, \cD)$ is a partition, $Q_{\cK,\cS, \cD}$ is a polynomial over $\{j_t: \cS_t=-, t\not\in \cD\}$ with degree $k_1+k_2+\cdots+k_r-r=k$, and the hypergraph $\cG(\cK, \cS, \cD, \cP)$ is connected. If $\cP=\{V_1, V_2,\cdots, V_s\}$, 
the polynomial $Q_{\cK,\cS, \cD}$ decomposes into a sum of product of polynomials in variables $\{j_t: t\in V_u\}$ for $u=1,2,\cdots, s$. 
Therefore the expression \eqref{e:copy10} is a finite sum of terms in the following form
\begin{align}\label{e:copy11}
\frac{N^{|\{t: \cS_t=+, t\not\in \cD\}|}}{N^{k}}
a_{(\bml_1,\bml_2,\cdots, \bml_s)}
\cL_{\bml_1}(F_N)\cL_{\bml_2}(F_N)\cdots \cL_{\bml_s}(F_N),
\end{align}
where the constant $a_{(\bml_1,\bml_2,\cdots, \bml_s)}$ is independent of $N$, the length $\ell(\bml_u)=|V_u|$ for $u=1,2,\cdots, s$, and $|\bml_1|+|\bml_2|+\cdots +|\bml_s|\leq k$.
If $s\geq 2$, we have $|\bml_1|,|\bml_2|,\cdots|\bml_s|\leq k-1$, therefore by our induction assumption, for $1\leq u\leq s$, the $N$-degree of $\cL_{\bml_u}(F_N)$ is at most $1$ if $|V_u|=1$; is at most $0$ if $|V_u|=2$; is less than $0$ if $|V_u|\geq 3$. 
If $|V_u|\geq 3$ for some $1\leq u\leq s$, then the $N$-degree of \eqref{e:copy11} is less than
\begin{align*}\begin{split}
|\{t: \cS_t=+, t\not\in \cD\}|+\sum_{u=1}^s 1_{\{|V_u|=1\}}-k
&=r-2|\cD|-\sum_{u=1}^s(|V_u|-1_{\{|V_u|=1\}})\\
&\leq r-1-|\cD|-\sum_{u=1}^s(|V_u|-1)\leq 0,
\end{split}\end{align*}
where we used Claim \ref{c:edgeineq}. Otherwise we have $|V_u|\leq 2$ of all $1\leq u\leq s$. Again by Claim \ref{c:edgeineq},
the $N$-degree of \eqref{e:copy11} is at most
\begin{align*}\begin{split}
|\{t: \cS_t=+, t\not\in \cD\}|+\sum_{u=1}^s 1_{\{|V_u|=1\}}-k
&=r-2|\cD|-\sum_{u=1}^s(|V_u|-1_{\{|V_u|=1\}})\\
&\leq r-2\left(|\cD|-\sum_{u=1}^s(|V_u|-1)\right)\leq r-2(r-1)< 0,
\end{split}\end{align*}
Therefore in this case, \eqref{e:copy10} has an $N$-degree less than $0$. If $s=1$, then $\cP=\{\{t:\cS_t=-,t\not\in \cD\}\}$, and \eqref{e:copy11} simplifies to
\begin{align}\label{e:copy12} 
\frac{N^{|\{t: \cS_t=+, t\not\in \cD\}|}}{N^{k}}
a_{(\bml)}
\cL_{\bml}(F_N),
\end{align}
where $\ell(\bml)=|\{t:\cS_t=-,t\not\in \cD\}|$ and $|\bml|\leq k$.
If $|\bml|\leq k-1$, the same argument as above, by our induction assumption, \eqref{e:copy12} has an $N$-degree at most $0$. Therefore we have
\begin{align}\label{e:copy13}
\eqref{e:copy10}=\frac{N^{|\{t: \cS_t=+, t\not\in \cD\}|}}{N^{k}}
\sum_{|\bml|=k}a_{(\bml)}
\cL_{\bml}(F_N)
+\{\text{terms with $N$-degree less than $0$}\}.
\end{align}
The coefficient $N^{|\{t: \cS_t=+, t\not\in \cD\}|}/N^k=1$ only if $\cD=\emptyset$, $\cS=(\cS^{(1)}, \cS^{(2)},\cdots,\cS^{(r)})$ and $\cS^{(i)}=(+,+,\cdots,+,-)$ for $1\leq i\leq r$. If this is the case, the polynomial $Q_{\cK,\cS, \emptyset}$ splits
\begin{align*}
Q_{\cK,\cS, \emptyset}((j_t: \cS_t=-))
=\prod_{i=1}^rQ_{(k_i),\cS^{(i)}, \emptyset}((j_t: \cS^{(i)}_t=-)),
\end{align*}
and Proposition \ref{p:decompQ} implies that
\begin{align}\begin{split}\label{e:l=3}
\eqref{e:copy13}
&=\frac{\prod_{i=1}^r(k_i-1)}{\theta^r}\cL_{(k_1-1,k_2-1,\cdots, k_r-1)}(F_N)\\
&+
\sum_{|\bml|=k,\bml\prec (k_1-1,k_2-1,\cdots,k_r-1)}a_{(\bml)}
\cL_{\bml}(F_N)
+\{\text{terms with $N$-degree less than $0$}\}.
\end{split}\end{align}

We denote the vector $v$, parametrized by Young diagrams $\bml$ with size $|\bml|=k$,
\begin{align*}
\frac{\cL_{(k)}(F_N)}{N}, \{\cL_{\bml}(F_N)\}_{|\bml|=k, \bml\neq (k)}.
\end{align*}
Then \eqref{e:l=1}, \eqref{e:l=2} and \eqref{e:l=3} together can be rewritten as the following linear equation,
\begin{align}\label{e:keylinear}
(A+B)v=o,
\end{align}
where $A$ is an upper triangular matrix, independent of $N$, with diagonal entries given by $A_{\bml\bml}=\prod_{i=1}^{\ell(\bml)} l_i/\theta^{\ell(\bml)}$, $B$ is a matrix with each entry size $\OO(1/N)$, $o$ is a vector parametrized by Young diagrams of size $k$, such that $o_{\bml}$ has an $N$-degree at most $0$ if $\ell(\bml)\leq 2$, and less than $0$ if $\ell(\bml)\geq 3$. Since $A$ is upper triangular, with diagonal entries bounded away from $0$, $A^{-1}$ is also upper triangular. For $N$ large enough, $B$ is a small perturbation, so $A+B$ is invertible and the norm of $(A+B)^{-1}$ is bounded. We can solve for $v$, using \eqref{e:keylinear}
\begin{align*}
v=(A+B)^{-1}o=A^{-1}o-(A+B)^{-1}BA^{-1}o=A^{-1}o+\OO(1/N),
\end{align*}
it follows that $v_{\bml}$ has an $N$-degree at most $0$ if $\ell(\bml)\leq 2$, and less than $0$ if $\ell(\bml)\geq 3$. This finishes the induction. We can conclude that the the Jack generating functions $F_N(\bmp;\theta)$ satisfy Assumption \ref{a:CLT}, and $\{\fc_k\}_{k\geq 1}$, $\{\fd_{k,l}\}_{k,l\geq 1}$ satisfies \eqref{e:invllnc} and \eqref{e:invcltcov}.

\end{proof}

\section{Applications}\label{s:applications}

In this section we discuss some applications of our main theorems. Our theorems lead to the proofs of the law of large numbers and central limit theorems of the Littlewood-Richardson coefficients for zonal polynomials, and for a rich family of nonintersecting random walks, we identify the fluctuations of their height functions as the pullback of the Gaussian free field on the upper half plane. The same as for continuous $\beta$-ensembles, in all our examples, the law of large numbers do not depend on the parameter $\theta$, and the covariance structure in the central limit theorems is also $\theta$ independent up to a factor $\theta$. As applications of our inverse theorems, we derive the asymptotics of the Jack characters at unity. It is worth mentioning that the asymptotics of Jack characters at unity in Theorem \ref{t:jack}, are the same as those of schur characters up to a factor $\theta$.

\subsection{Asymptotics of Jack characters}
In this section we study the asymptotics of Jack characters as an application of the inverse Theorems \ref{t:invclt}.
For any compactly supported probability measure $\mu$, its Stieltjes transform is given by
\begin{align}
m_\mu(z)=\frac{\rd \mu(x)}{z-x}.
\end{align}
Since $\lim_{z\rightarrow \infty} z m_\mu(z)=1$, $m_\mu(z)$ is invertible in a neighborhood of $z=\infty$. We denote its functional inverse by $m_\mu^{-1}(z)$, which is well defined in a punctured neighborhood of $z=0$. The $R$-transform of the measure $\mu$, is defined as
\begin{align}
R_\mu(z)=m_\mu^{-1}(z)-\frac{1}{z},
\end{align}
which is a biholomorphic in a small neighborhood of $z=0$.
%
For $z$ in small neighborhood of $0$, $R_\mu(z)$ is well defined. Integrating $R_\mu(z)$, we set
\begin{align*}
H_\mu(x)=\int_0^{\ln x}(R_\mu(z)+1)\rd z+\ln \left(\frac{\ln x}{x-1}\right),
\end{align*}
which is analytic for $x$ in a small neighborhood of $1$.

In \cite{HJ}, the author studied a $\beta$-analogue of the nonintersecting Poisson random walks, where the empirical density satisfies the law of large numbers and the central limit theorems in the sense of Definition \ref{def:LLNCLT}. The $\beta$-nonintersecting Poisson random walks in \cite{HJ} is a special case of the discrete Markov chain defined in \eqref{e:chain}, with Jack-Plancherel specialization, i.e. the specialization corresponding to $\gamma=1$. Thanks to Theorem \ref{t:invclt}, we can use the results in \cite{HJ} to derive  asymptotics of Jack characters. We recall the quantity $j_\bmla(\theta)$ from \eqref{def:jl2} and the Jack reproducing kernel $H(\rho_1,\rho_2;\theta)$ from \eqref{speckernel}. The following is a variation of \cite[Corollary 1.12]{HJ}
\begin{theorem}\label{t:HJ}
Fix $\fb>0$ large. Let $\{\bmla\in \bY(N)\}_{N\geq 1}$ be a sequence of Young diagrams, such that $\supp \mu[\bmla]\subset[-\fb, \fb]$ and $\mu[\bmla]$ weakly converges to a measure $\mu$ as $N$ goes to infinity,
and $\rho$  be  the Jack-Plancherel specialization corresponding to $\gamma=1$.
Then the sequence of measure $\{M_N=M_{\bmla}\}_{N\geq 1}$,
 \begin{align}\label{e:MNJP}
 M_\bmla(\bmmu)=\frac{1}{H(N\rho, 1^N;\theta)}\frac{j_{\bmla}(\theta)}{j_\bmmu(\theta)}\frac{J_{\bmmu}(1^N;\theta)}{J_\bmla(1^N;\theta)}J_{\bmmu/\bmla}(N\rho;\theta),\quad \bmmu\in \bY(N),
 \end{align}
  satisfies the law of large numbers and the central limit theorems, in the sense of Definition \ref{def:LLNCLT}, with
 \begin{align}\label{e:MNJPU}
 U_\mu(z)=\theta \left(H_{\mu}'(z)+1\right),
 \end{align}
 and
 \begin{align}\begin{split}\label{e:MNJPV}
 V_\mu(z,w)=
 \theta \del_z\del_w\ln \left(1+(z-1)(w-1)\frac{zH_{\mu}'(z)-wH_{\mu}'(z)}{z-w}\right).
 \end{split}\end{align}
\end{theorem}

\begin{remark}
Theorem \ref{t:HJ} is slightly different from \cite[Corollary 1.12]{HJ}, where the author assumed the estimate \cite[Assumption 1.5]{HJ} on the difference between $\mu[\bmla]$ and $\mu$, and used it to derive the second order asymptotics of the mean of the random measures $\mu[\bmmu]$, under $M_{\bmla}(\bmmu)$. This is not needed for the study of the fluctuation of $\mu[\bmmu]$.
\end{remark}

\begin{theorem}\label{t:jack}
Fix $\fb>0$ large. Let $\{\bmla\in \bY(N)\}_{N\geq 1}$ be a sequence of Young diagrams, such that $\supp \mu[\bmla]\subset [-\fb, \fb]$ and $\mu[\bmla]$ weakly converges to a measure $\mu$ as $N$ goes to infinity. Then for any index $i\geq 1$ and exponent $k\geq 1$, we have,
\begin{align}\label{e:firstder1}
\lim_{N\rightarrow \infty}\frac{1}{N}\del_i^k\ln\left.\left(\frac{ J_{\bmla} \left(x_1,x_2,\cdots,x_N; \theta \right) }{J_{\bmla}(1^N; \theta)} \right)\right|_{1^N}= \theta \left.\del_x^k H_\mu(x)\right|_{x=1},
\end{align}
For any distinct indices $i,j$ and any exponents $k,l\geq 1$,
\begin{align}\begin{split}\label{e:secondder2}
&\phantom{{}={}}\lim_{N\rightarrow \infty}\del_i^k\del_j^l\ln\left.\left(\frac{ J_{\bmla} \left(x_1,x_2,\cdots,x_N; \theta \right) }{J_{\bmla}(1^N; \theta)} \right)\right|_{1^N}\\
&= \theta \left.\del_{x_1}^k\del_{x_2}^{l}\ln\left(1+(x_1-1)(x_2-1)\frac{x_1 H_\mu'(x_1)-x_2H_\mu'(x_2)}{x_1-x_2}\right)\right|_{x_1=x_2=1},
\end{split}\end{align}
For any $r\geq 3$ and any indices $i_1,i_2,\cdots, i_r$ such that there are at least three distinct indices among them,
\begin{align}\label{e:thirdder3}
\lim_{N\rightarrow \infty}\del_{i_1}\del_{i_2}\cdots \del_{i_r}\ln\left.\left(\frac{ J_{\bmla} \left(x_1,x_2,\cdots,x_N; \theta \right) }{J_{\bmla}(1^N; \theta)} \right)\right|_{1^N}= 0.
\end{align}
\end{theorem}

\begin{remark}
The limit \eqref{e:firstder1} can be alternatively derived from the integral representations for Jack characters in \cite[Theorems 2.2 and 2.3]{Cuenca}.
Both \eqref{e:secondder2} and \eqref{e:thirdder3} could possibly be alternatively derived from the multivariate formula \cite[Theorem 3.1]{Cuenca}, which the author calls the Pieri integral formula for Jack polynomials. However, applying this formula is a nontrivial task and requires some work.
\end{remark}

\begin{proof}
From Theorem \ref{t:HJ}, the sequence of measures $\{M_N=M_{\bmla}\}_{N\geq 1}$ satisfies the law of large numbers and central limit theorems in the sense of Definition \ref{def:LLNCLT}. Theorem \ref{t:jack} follows from applying Theorem \ref{t:invclt} to the sequence of measure $\{M_N=M_{\bmla}\}_{N\geq 1}$.

Using Proposition $\ref{p:Stochastic}$, the Jack generating function of $M_N=M_{\bmla}$ in \eqref{e:MNJP} is given by
\begin{align*}\begin{split}
F_{N}(\bmp; \theta)
&= \sum_{\bmmu\in\Y(N)}{M_{N}(\bmmu)\frac{J_{\bmmu}(\bmp; \theta)}{J_{\bmmu}(1^N; \theta)}}
= \sum_{\bmmu\in\Y(N)} 
\frac{1}{H(N\rho, 1^N;\theta)}\frac{j_{\bmla}(\theta)}{j_\bmmu(\theta)}\frac{J_{\bmmu}(\bmp;\theta)}{J_\bmla(1^N;\theta)}J_{\bmmu/\bmla}(N\rho;\theta)\\
&= \frac{H(N\rho, \bmp; \theta)}{H(N\rho, 1^N; \theta)}\frac{J_{\bmla}(\bmp;\theta)}{J_{\bmla}(1^N;\theta)}-
\sum_{\bmmu\not\in\Y(N)} 
\frac{1}{H(N\rho, 1^N;\theta)}\frac{j_{\bmla}(\theta)}{j_\bmmu(\theta)}\frac{J_{\bmmu}(\bmp;\theta)}{J_\bmla(1^N;\theta)}J_{\bmmu/\bmla}(N\rho;\theta)\\
&= \frac{H(N\rho, \bmp; \theta)}{H(N\rho, 1^N; \theta)}\frac{J_{\bmla}(\bmp;\theta)}{J_{\bmla}(1^N;\theta)}+\Sp(J_\bmmu(\bmp;\theta): \bmmu\not\in \bY(N)),
\end{split}\end{align*}
%
where where $\Sp(J_\bmeta(\bmp;\theta): \bmeta\not\in \bY(N))$  is the linear span of Jack symmetric functions $J_\bmeta(\bmp;\theta)$ with $\bmeta\not\in \bY(N)$ in $\Sym$. 
Since the Nazarov-Sklyanin Operators preserve the subspace $\Sp(J_\bmla(\bmp;\theta): \bmla\not\in \bY(N))$, and for any $\bmla\not\in \bY(N)$, $J_\bmla(1^N;\theta)=0$, we get
\begin{align*}
\left.I^{(k_1)}I^{(k_2)}\cdots I^{(k_r)}F_N(\bmp;\theta)\right|_{\bmp=1^N}
=\left.I^{(k_1)}I^{(k_2)}\cdots I^{(k_r)}\frac{H(N\rho, \bmp; \theta)}{H(N\rho, 1^N; \theta)}\frac{J_{\bmla}(\bmp;\theta)}{J_{\bmla}(1^N;\theta)}\right|_{\bmp=1^N}.
\end{align*}
Therefore, we can instead use 
\begin{align*}
\tilde F_N(\bmp;\theta)=\frac{H(N\rho, \bmp; \theta)}{H(N\rho, 1^N; \theta)}\frac{J_{\bmla}(\bmp;\theta)}{J_{\bmla}(1^N;\theta)}
=e^{-\theta N^2}e^{\theta Np_1}\frac{J_{\bmla}(\bmp;\theta)}{J_{\bmla}(1^N;\theta)}
\end{align*}
to compute the joint moments of the random measure $\mu[\bmla]$ as defined in \eqref{def:mula}. 
Since $J_{\bmla}(\bmp;\theta)$ is a finite sum of power sum symmetric functions. We conclude that $\tilde F_N(\bmp;\theta)$ satisfies Assumption \ref{asup:infinite}. 

It follows from Theorem \ref{t:invclt} and \ref{thm:equivalent},  for any indices $i$ and any exponent $k\geq 1$, the limit exists
\begin{align}\label{e:dera1}
\fc_{k}=\lim_{N\rightarrow \infty}\left.\frac{\del_i^k \ln \tilde F_N(x_1,x_2,\cdots,x_N;\theta)}{N}\right|_{1^N}=
\theta {\bm 1}_{k=1}+\lim_{N\rightarrow \infty}\frac{1}{N}\del_i^k\ln\left.\left(\frac{ J_{\bmla} \left(x_1,x_2,\cdots,x_N; \theta \right) }{J_{\bmla}(1^N; \theta)} \right)\right|_{1^N},
\end{align}
and is given by
\begin{align}\label{e:dera2}
 \fc_k=\del_z^{k-1}U_\mu(z)|_{z=1}=
\theta({\bm1}_{k=1}+\del_x^k H_\mu(x)|_{x=1}).
\end{align}
where $U_\mu(z)$ is as in \eqref{e:MNJPU}. The claim \eqref{e:firstder1} follows from comparing \eqref{e:dera1} and \eqref{e:dera2}.

It follows from Theorem \ref{t:invclt} and \ref{thm:equivalent}, for  any distinct indices $i,j$ and any exponent $k,l\geq 1$, the limit exists
\begin{align*}
\fd_{k,l}=\lim_{N\rightarrow \infty}\left.\del_i^k\del_j^l \ln \tilde F_N(x_1,x_2,\cdots,x_N;\theta)\right|_{1^N}=
\lim_{N\rightarrow \infty}\del_i^k\del_j^l\ln\left.\left(\frac{ J_{\bmla} \left(x_1,x_2,\cdots,x_N; \theta \right) }{J_{\bmla}(1^N; \theta)} \right)\right|_{1^N},
\end{align*}
and is given by
\begin{align*}\begin{split}
\fd_{k,l}
&=\del_{z}^{k-1}\del_w^{l-1}V_\mu(z,w)|_{z=w=1}\\
&= \theta \left.\del_{x_1}^k\del_{x_2}^{l}\ln\left(1+(x_1-1)(x_2-1)\frac{x_1 H_\mu'(x_1)-x_2H_\mu'(x_2)}{x_1-x_2}\right)\right|_{x_1=x_2=1},
\end{split}\end{align*}
where $V_\mu(z,w)$ is as in \eqref{e:MNJPV}. This finishes the proof of \eqref{e:secondder2}.
For any $r\geq 3$ and any indices $i_1,i_2,\cdots, i_r$ such that there are at least three distinct indices among them, Theorem \ref{t:invclt} and \ref{thm:equivalent} conclude
\begin{align*}\begin{split}
&\phantom{{}={}}\lim_{N\rightarrow \infty}\left.\del_{i_1}\del_{i_2}\cdots\del_{i_r} \ln \tilde F_N(x_1,x_2,\cdots,x_N;\theta)\right|_{1^N}\\
&=\lim_{N\rightarrow \infty}\del_{i_1}\del_{i_2}\cdots \del_{i_r}\ln\left.\left(\frac{ J_{\bmla} \left(x_1,x_2,\cdots,x_N; \theta \right) }{J_{\bmla}(1^N; \theta)} \right)\right|_{1^N}= 0.
\end{split}\end{align*}
This finishes the proof of \eqref{e:thirdder3}.

\end{proof}

%
%

\subsection{Littlewood-Richardson coefficients for the product of zonal polynomials}

In general, for a Gelfand pair $(G, K)$, with zonal spherical functions $\{\Omega_i\}_{i\in I}$. We may identify $G$ with its image in $G\times G$ under the diagonal map $x\mapsto (x,x)$. The zonal spherical functions of $(G\times G, K\times K)$ are $\{\Omega_i \times \Omega_j\}_{i,j\in I}$,  and the pointwise product $\Omega_i\cdot \Omega_j$ is the restriction of $\Omega_i \times \Omega_j$ to $G$. Thus, the restriction $\Omega_i\cdot \Omega_j$ decomposes as a non-negative linear combination of zonal spherical functions of the Gelfand pair $(G,K)$. More precisely, the coefficients $c_{ij}^k$ defined by
\begin{align}\label{e:productrule}
\Omega_i \cdot \Omega_j=\sum_{k\in I}c_{ij}^k \Omega_k,
\end{align}
are real and non-negative.

In the setting, for the Gelfand pair $(GL_N(\bC), U(N))$, its zonal spherical functions correspond to the Schur polynomials, and the coefficients in \eqref{e:productrule} are related to the Littlewood-Richardson coefficients which arise when decomposing a product of two Schur polynomials as a linear combination of other Schur polynomials. 


More generally, we consider the Littlewood-Richardson coefficients for the product of Jack polynomials. More precisely, let $N\in \bZ_{>0}$ be an arbitrary positive integer.  For two Young diagrams $\bmmu,\bmeta\in \bY(N)$, there is a decomposition of the product of two Jack polynomials into a linear combination of Jack polynomials:
\begin{align*}
J_{\bmmu}(x_1,x_2,\cdots, x_N;\theta)J_{\bmeta}(x_1,x_2,\cdots, x_N;\theta)=\sum_{\bmla\in \bY(N)}c_{\bmla}^{\bmmu\bmeta}(\theta)J_{\bmla}(x_1,x_2,\cdots, x_N;\theta).
\end{align*}
The coefficients  $c_{\bmla}^{\bmmu\bmeta}(\theta)$ are called the Littlewood-Richardson coefficients for the product of Jack polynomials. When $\theta=1/2,1,2$, the Jack polynomials specialize to the zonal polynomials, Schur polynomials, and symplectic zonal polynomials, which are zonal spherical functions of the Gelfand pairs $(GL_N(\bR), O(N))$, $(GL_N(\bC), U(N))$ and $(GL_N(\bH), Sp(N))$. Therefore, the coefficients  $c_{\bmla}^{\bmmu\bmeta}(\theta)$ are non-negative. In general, it remains a conjecture of Richard Stanley  \cite{MR1014073} that those coefficients $c_{\bmla}^{\bmmu\bmeta}(\theta)$ are polynomials in $1/\theta$ with nonnegative integer coefficients.   For $\theta=1/2, 1,2$, the Littlewood-Richardson coefficients induce a probability measure on $\bY(N)$:
\begin{align}\label{e:defLRjack}
M_{\bmmu,\bmeta}(\bmla)
=\frac{J_{\bmla}(1^N;\theta)}{J_{\bmmu}(1^N;\theta)J_{\bmeta}(1^N;\theta)}c_{\bmla}^{\bmmu\bmeta}(\theta),\quad \bmla\in \bY(N).
\end{align}
%
%

\begin{theorem}
Fix $\fb>0$ large. Let $\theta\in \{1/2,1,2\}$, and $\{\bmmu,\bmeta\in \bY(N)\}_{N\geq 1}$ be a sequence of Young diagrams, such that $\supp \mu[\bmmu], \supp \mu[\bmeta]\subset [-\fb, \fb]$ and 
\begin{align*}
\lim_{N\rightarrow \infty}\mu[\bmmu]=\mu_1,\quad
\lim_{N\rightarrow \infty}\mu[\bmeta]=\mu_2,
\end{align*}
 then the sequence of measures $\{M_N=M_{\bmmu,\bmeta}\}_{N\geq 1}$, as defined in \eqref{e:defLRjack}, satisfies the law of large numbers and the central limit theorems, with
 \begin{align*}
 U_\mu(z)=\theta \left(H_{\mu_1}'(z)+H_{\mu_2}'(z)\right),
 \end{align*}
 and
 \begin{align*}\begin{split}
 V_\mu(z,w)&=
 \theta \del_z\del_w\ln \left(1+(z-1)(w-1)\frac{zH_{\mu_1}'(z)-wH_{\mu_1}'(w)}{z-w}\right)\\
 &+\theta \del_z\del_w\ln\left(1+(z-1)(w-1)\frac{zH_{\mu_2}'(z)-wH_{\mu_2}'(w)}{z-w}\right).
 \end{split}\end{align*}
\end{theorem}
\begin{proof}
The Jack generating function of the measure $M_N=M_{\bmmu, \bmeta}$ as defined in \eqref{e:defLRjack} is
\begin{align}
\begin{split}
\label{e:LRjack}
F_{N}(\bmp; \theta) 
&=\frac{J_\bmmu(\bmp;\theta)}{J_{\bmmu}(1^N;\theta)}\frac{J_{\bmeta}(\bmp;\theta)}{J_{\bmeta}(1^N;\theta)}
-\sum_{\bmla\not\in \Y(N)}\frac{c_{\bmla}^{\bmmu\bmeta}(\theta)}{J_{\bmmu}(1^N;\theta)J_{\bmeta}(1^N;\theta)}J_{\bmla}(\bmp;\theta)\\
&=\frac{J_\bmmu(\bmp;\theta)}{J_{\bmmu}(1^N;\theta)}\frac{J_{\bmeta}(\bmp;\theta)}{J_{\bmeta}(1^N;\theta)}+\Sp(J_\bmla(\bmp;\theta): \bmla\not\in \bY(N)),
\end{split}\end{align}
where $c_{\bmla}^{\bmmu\bmeta}(\theta)$ are the Littlewood-Richardson coefficients of the product of two Jack symmetric functions:
\begin{align*}
J_{\bmmu}(\bmp;\theta)J_{\bmeta}(\bmp;\theta)=\sum_{\bmla\in \bY}c_{\bmla}^{\bmmu\bmeta}(\theta)J_{\bmla}(\bmp;\theta).
\end{align*}
Since the Nazarov-Sklyanin Operators preserve the subspace $\Sp(J_\bmla(\bmp;\theta): \bmla\not\in \bY(N))$, and for any $\bmla\not\in \bY(N)$, $J_\bmla(1^N;\theta)=0$, we get
\begin{align*}
\left.I^{(k_1)}I^{(k_2)}\cdots I^{(k_r)}F_N(\bmp;\theta)\right|_{\bmp=1^N}
=\left.I^{(k_1)}I^{(k_2)}\cdots I^{(k_r)}\frac{J_\bmmu(\bmp;\theta)}{J_{\bmmu}(1^N;\theta)}\frac{J_{\bmeta}(\bmp;\theta)}{J_{\bmeta}(1^N;\theta)}\right|_{\bmp=1^N}.
\end{align*}
Therefore, we can instead use 
\begin{align*}
\tilde F_N(\bmp;\theta)=\frac{J_\bmmu(\bmp;\theta)}{J_{\bmmu}(1^N;\theta)}\frac{J_{\bmeta}(\bmp;\theta)}{J_{\bmeta}(1^N;\theta)},
\end{align*}
to compute the joint moments of the random measure $\mu[\bmla]$ as defined in \eqref{def:mula}. Since $J_\bmmu(\bmp;\theta), J_{\bmeta}(\bmp;\theta)$ are finite sums of power sum symmetric functions, $\tilde F_N(\bmp;\theta)$ satisfies Assumption \ref{asup:infinite}.
Moreover, thanks to Theorem \ref{t:jack} and \ref{thm:equivalent}, for any $k\geq 0$,
\begin{align*}
 \fc_k=\lim_{N\rightarrow \infty}\left.\frac{\del_i^k \ln \tilde F_N(x_1,x_2,\cdots,x_N; \theta)}{N}\right|_{1^N}=\left.\theta (\del_x ^k H_{\mu_1}(x)+\del_x ^k H_{\mu_2}(x))\right|_{x=1}.
\end{align*}
For any distinct indices $i,j$ and any exponents $k,l\geq 1$,
\begin{align*}\begin{split}
&\phantom{{}={}}\lim_{N\rightarrow \infty}\del_i^k\del_j^l\ln\left.\tilde F_N(x_1,x_2,\cdots, x_N;\theta)\right|_{1^N}\\
&= \theta \left.\del_{x_1}^k\del_{x_2}^{l}\ln\left(1+(x_1-1)(x_2-1)\frac{x_1 H_{\mu_1}'(x_1)-x_2H_{\mu_1}'(x_2)}{x_1-x_2}\right)\right|_{x_1=x_2=1}\\
&+\theta \left.\del_{x_1}^k\del_{x_2}^{l}\ln\left(1+(x_1-1)(x_2-1)\frac{x_1 H_{\mu_2}'(x_1)-x_2H_{\mu_2}'(x_2)}{x_1-x_2}\right)\right|_{x_1=x_2=1}.
\end{split}\end{align*}
For any $r\geq 3$ and any indices $i_1,i_2,\cdots, i_r$ such that there are at least three distinct indices among them,
\begin{align*}
\lim_{N\rightarrow \infty}\del_{i_1}\del_{i_2}\cdots \del_{i_r}\ln\left.\tilde F_N(x_1,x_2,\cdots,x_N;\theta)\right|_{1^N}= 0.
\end{align*}
Thus $\tilde F_N(\bmp;\theta)$ satisfies Assumption \ref{a:LLN} and \ref{a:CLT}. The limit functions \eqref{eqn:UV} are given by
\begin{align*}
\begin{split}
U_\mu(z)
&=
\sum_{k\geq 1}\frac{\fc_k}{(k-1)!}(z-1)^{k-1}
=\sum_{k\geq 1}\frac{\theta (\del_z^kH_{\mu_1}(z)+\del_z^k H_{\mu_2}(z)|_{z=1})}{(k-1)!}(z-1)^{k-1}\\
&=\theta (H_{\mu_1}'(z)+H_{\mu_2}'(z)), 
\end{split}\end{align*}
and
\begin{align*}\begin{split}
V_\mu(z,w)
&=\sum_{k,l\geq 1}\frac{\fd_{k,l}}{(k-1!)(l-1)!}(z-1)^{k-1}(w-1)^{k-1}\\
&=
 \theta \del_z\del_w\ln \left(1+(z-1)(w-1)\frac{zH_{\mu_1}'(z)-wH_{\mu_1}'(w)}{z-w}\right)\\
 &+\theta \del_z\del_w\ln\left(1+(z-1)(w-1)\frac{zH_{\mu_2}'(z)-wH_{\mu_2}'(w)}{z-w}\right).
\end{split}\end{align*}
The claim follows from Theorem \ref{thm:LLN} and Theorem \ref{t:clt}.
\end{proof}

\subsection{Nonintersecting random walks}\label{s:nrw}

Given two Jack positive specializations $\rho_1$ and $\rho_2$.
We define a transition matrix $p_{\bmla\rightarrow \bmmu}(\rho_1,\rho_2;\theta)$ with rows and columns indexed by Young diagrams as follows
\begin{align}\label{e:defp}
p_{\bmla\rightarrow \bmmu}(\rho_1,\rho_2;\theta)=\frac{1}{H(\rho_1,\rho_2;\theta)}\frac{j_{\bmla}(\theta)}{j_{\bmmu}(\theta)}\frac{J_{\bmmu}(\rho_2;\theta)}{J_{\bmla}(\rho_2;\theta)}J_{\bmmu/\bmla}(\rho_1; \theta).
\end{align}
It follows from Proposition \ref{p:Stochastic} that the matrix $p_{\bmla\rightarrow \bmmu}(\rho_1,\rho_2;\theta)$ is a stochastic matrix. Moreover, the transition matrix $p_{\bmla\rightarrow \bmmu}(\rho_1,\rho_2;\theta)$ preserves the Jack measure as in Definition \ref{def:jackmeasure}
\begin{align*}
\sum_{\bmla}\cJ_{\rho_1\rho_2}(\bmla)p_{\bmla\rightarrow \bmmu}(\rho_3,\rho_2;\theta)
=\cJ_{(\rho_1,\rho_3)\rho_2}(\bmmu).
\end{align*}
$p_{\bmla\rightarrow \bmmu}(\rho_1,\rho_2;\theta)$ can be viewed as the transition probability of certain dynamics on the Young diagrams $\bY$. In fact, using Proposition \ref{p:chain2}, we have the following relation
\begin{align*}
\sum_{\bmmu}p_{\bmla\rightarrow \bmmu}(\rho_1,\rho_2;\theta)p_{\bmmu\rightarrow \bmeta}(\rho_1',\rho_2;\theta)
=p_{\bmla\rightarrow \bmeta}((\rho_1,\rho_1'),\rho_2;\theta).
\end{align*}

Let $N\in\bZ_{>0}$ be an arbitrary positive integer.  In this example, we are interested in the case $\rho_2 = (1^N)$, i.e. the pure alpha specialization $\alpha_1=\alpha_2=\cdots=\alpha_N=1, \alpha_{N+1}=\alpha_{N+2}=\cdots=0$. For $H(\rho_1,\rho_2,\theta)$ to be well defined, we introduce the following definition.
\begin{definition}\label{def:stable}
We say a Jack positive specialization $\rho$ is \emph{stable}, if there exists a small positive number $\varepsilon>0$, 
\begin{align*}
\sum_{n\geq 1}\frac{|\rho_n|}{n}|x|^n,
\end{align*}
converges absolutely on $\{z\in\bC:|z|<1+\varepsilon\}$. If $\rho$ is a stable Jack positive specialization, we define
\begin{align*}
W_\rho(z)=\sum_{n\geq 1}\frac{\rho_n z^n}{n}.
\end{align*}
\end{definition}

\begin{proposition}\label{p:pstable}
If $\rho$ is a stable Jack positive specialization, then $H(\rho,\bmp;\theta)$ satisfies Assumption \ref{asup:infinite}, and especially $H(\rho,1^N;\theta)$ is well defined.
\end{proposition}
\begin{proof}
By the Taylor expansion,
\begin{align*}
H(\rho,\bmp;\theta)=\exp\left(\theta \sum_{n\geq 1}\frac{\rho_n p_n}{n}\right)
=\sum_{\bmla\in \bY}a_H(\bmla)p_\bmla.
\end{align*} 
Then we have
\begin{align*}
\sum_{\bmla\in \bY}|a_H(\bmla)|N^{\ell(\bmla)}|x|^{|\bmla|}
\leq\exp\left(\theta N \sum_{n\geq 1}\frac{|\rho_n| |x|^n }{n}\right),
\end{align*}
which converges absolutely on the disk $\{z\in \bC:|z|<1+\varepsilon\}$.
\end{proof}

Let $\rho_1=\rho$ be a stable Jack positive specialization as defined in Definition \ref{def:stable}, and $\rho_2=(1^N)$, the above construction \eqref{e:defp} specializes to a Markov chain $\bmla^{(0)}, \bmla^{(1)},\bmla^{(2)},\cdots \in \bY(N)$, with transition probability,
\begin{align}\label{e:chain}
\bP(\bmla^{(t)}=\bmmu|\bmla^{(t-1)}=\bmla)=p_{\bmla\rightarrow\bmmu}(\rho, 1^N;\theta),\quad t=1,2,3,\cdots.
\end{align}
We define the auxiliary Markov chain $\bmx^{(t)}=(x^{(t)}_1, x^{(t)}_2,\cdots, x^{(t)}_N)$ where
\begin{align}
x^{(t)}_i=\frac{\la_i^{(t)}}{\theta N}-\frac{i-1}{N}, \quad i=1,2,\cdots, N.
\end{align}
for $t=1,2,3,\cdots$. The Markov chain $\bmx^{(0)}, \bmx^{(1)},\bmx^{(2)},\cdots$ can be viewed as a  discrete Markov chain of $N$ nonintersecting particles, whose interaction is non-local and of log-gas type. For some special choices of the Jack positive specialization $\rho$, the Markov chains $\bmx^{(0)}, \bmx^{(1)},\bmx^{(2)},\cdots$ specialize to some well-known nonintersecting random walks.
\begin{itemize}	
	\item If $\theta=1$, and $\rho$ is the Jack-Plancherel specialization  $\gamma=\delta$ for some $0<\delta$, Markov chain $\bmx^{(0)}, \bmx^{(1)},\bmx^{(2)},\cdots$ is the continuous time Poisson random walk with intensity $\delta$, conditioning on nonintersecting.
	\item If $\theta=1$, and $\rho$ is the Pure alpha specialization $\alpha_1=\delta$, $\alpha_2=\alpha_3=\cdots=0$ for some $0< \delta<1$, the Markov chain $\bmx^{(0)}, \bmx^{(1)},\bmx^{(2)},\cdots$ is the discrete time geometric random walk, where at each moment the particle jumps to the right $r$ steps with probability $\delta(1-\delta)^r$, conditioning on nonintersecting. 
	
	\item If $\theta=1$, and $\rho$ is the Pure beta specialization  $\beta_1=\delta$, $\beta_2=\beta_3=\cdots=0$ for some $0< \delta<1$, the Markov chain $\bmx^{(0)}, \bmx^{(1)},\bmx^{(2)},\cdots$ is the discrete time Bernoulli random walk, where at each moment the particle jumps to the right by one with probability $\delta$, conditioning on nonintersecting.  

\end{itemize}
The general cases with $\theta>0$ are the $\beta$-analogues of the above nonintersecting random walks. 

\begin{theorem}\label{t:nrw}
Fix $\fb>0$ large, and a stable Jack positive specialization $\rho$, as defined in Definition \ref{def:stable}. Let $\{ \bmla^{(0)} \in \Y(N) \}_{N \geq 1}$ be a sequence of Young diagrams, such that $\supp \mu[\bmla^{(0)}]\subset [-\fb,\fb]$ and
\begin{align*}
\lim_{N\rightarrow \infty}\mu[\bmla^{(0)}]=\mu^{(0)}.
\end{align*} 
Then the Markov chains $\bmla^{(0)},\bmla^{(1)}, \bmla^{(2)}, \cdots $ as defined in \eqref{e:chain} satisfy the law of large numbers and the central limit theorems in the sense of Definition \ref{def:LLNCLT} and Theorem \ref{t:mulclt}, with
\begin{align*}
U_\mu^{(\tau)}(z)=\theta(H'_{\mu^{(0)}}(z)+\tau W_\rho'(z)),
\end{align*}
and
\begin{align*}
V_\mu^{(\tau)}(z)=
 \theta \del_z\del_w\ln \left(1+(z-1)(w-1)\frac{zH_{\mu^{(0)}}'(z)-wH_{\mu^{(0)}}'(w)}{z-w}\right).
\end{align*}

\end{theorem}

\begin{proof}
In the notations of Section \ref{s:mulclt}, we take $M_N=\delta_{\bmla^{(0)}}$, and $\fm_N^{(1)}=\fm_N^{(2)}=\cdots=\cJ_{\rho,1^N}$ the Jack measure as defined in Definition \ref{def:jackmeasure}. Then their Jack generating functions are given by 
\begin{align*}
F_N(\bmp;\theta)=\frac{J_{\bmla^{(0)}}(\bmp;\theta)}{J_{\bmla^{(0)}}(1^N;\theta)},
\end{align*}
and 
\begin{align*}\begin{split}
&\phantom{{}={}}g_N^{(1)}(\bmp;\theta)=g_N^{(2)}(\bmp;\theta)=\cdots=
\sum_{\bflambda\in\Y(N)}{\cJ_{\rho,1^N}(\bmla)\frac{J_{\bflambda}(\bmp; \theta)}{J_{\bflambda}(1^N; \theta)}}
= \sum_{\bflambda\in\Y(N)}{ \frac{J_{\bflambda}(\rho; \theta)J_{\bflambda}(1^N; \theta)}{j_\bmla(\theta)H(\rho, 1^N; \theta)} \frac{J_{\bflambda}(\bmp; \theta)}{J_{\bflambda}(1^N; \theta)}}\\
&= \sum_{\bflambda\in\Y}\frac{{J_{\bflambda}(\rho; \theta)J_{\bflambda}(\bmp; \theta)}}{j_\bmla(\theta)H(\rho, 1^N; \theta)}
-\sum_{\bmla\not\in \bY(N)}
\frac{J_\bmla(\rho;\theta)}{j_\bmla(\theta)H(\rho,1^N;\theta)}
J_{\bmla}(\bmp;\theta)\\
&= \frac{H(\rho, \bmp; \theta)}{H(\rho, 1^N; \theta)}-\sum_{\bmla\not\in \bY(N)}
\frac{J_\bmla(\rho;\theta)}{j_\bmla(\theta)H(\rho,1^N;\theta)}
J_{\bmla}(\bmp;\theta)\\
&= \frac{H(\rho, \bmp; \theta)}{H(\rho, 1^N; \theta)} +\Sp(J_\bmla(\bmp;\theta): \bmla\not\in \bY(N)).
\end{split}\end{align*}
where $\Sp(J_\bmeta(\bmp;\theta): \bmeta\not\in \bY(N))$  is the linear span of Jack symmetric functions $J_\bmeta(\bmp;\theta)$ with $\bmeta\not\in \bY(N)$ in $\Sym$. We notice that
\begin{align*}
g_N^{(t)}(x_1,x_2,\cdots, x_N;\theta)=\frac{H(\rho, x_1,x_2,\cdots,x_N; \theta)}{H(\rho, 1^N; \theta)},\quad t=1,2,3,\cdots,
\end{align*}
and thus by Proposition \eqref{p:Stochastic}
\begin{align*}
g_N^{(t)}(x_1,x_2,\cdots, x_N;\theta)\frac{J_\bmla(x_1,x_2,\cdots, x_N;\theta)}{J_\bmla(1^N;\theta)}
=\sum_{\bmmu\in \Y(N)}p_{\bmla\rightarrow \bmmu}(\rho,1^N;\theta)
\frac{J_\bmmu(x_1,x_2,\cdots, x_N;\theta)}{J_\bmmu(1^N;\theta)}.
\end{align*}

We notice that $\Sp(J_\bmla(\bmp;\theta): \bmla\not\in \bY(N))$ is an ideal of $\Sym$, the algebra of symmetric functions in infinitely many variables, and the Nazarov-Sklyanin Operators preserve the subspace $\Sp(J_\bmla(\bmp;\theta): \bmla\not\in \bY(N))$, and for any $\bmla\not\in \bY(N)$, $J_\bmla(1^N;\theta)=0$. For any times $0\leq t_1\leq t_2\leq \cdots\leq t_r$, we have
\begin{align*}\begin{split}
&\phantom{{}={}}\left(\frac{I^{(k_r)}}{N^{k_r}}+\frac{I^{(k_r+1)}}{N^{k_r+1}}\right)
\left(\prod_{s=t_{r-1}+1}^{t_r} g_N^{(s)}\right)\cdots
\left.\left(\frac{I^{(k_{1})}}{N^{k_{1}}}+\frac{I^{(k_{1}+1)}}{N^{k_{1}+1}}\right) \left(\prod_{s=1}^{t_{1}} g_N^{(s)}\right)
F_N\right|_{\bmp=1^N}\\
&=\left(\frac{I^{(k_r)}}{N^{k_r}}+\frac{I^{(k_r+1)}}{N^{k_r+1}}\right)
\left(\prod_{s=t_{r-1}+1}^{t_r} \tilde g_N^{(s)}\right)\cdots
\left.\left(\frac{I^{(k_{1})}}{N^{k_{1}}}+\frac{I^{(k_{1}+1)}}{N^{k_{1}+1}}\right) \left(\prod_{s=1}^{t_{1}} \tilde g_N^{(s)}\right)
F_N\right|_{\bmp=1^N},
\end{split}\end{align*}
where 
\begin{align*}
\tilde g_N^{(1)}(\bmp;\theta)=\tilde g_N^{(2)}(\bmp;\theta)=\cdots=\frac{H(\rho, \bmp; \theta)}{H(\rho, 1^N; \theta)}
\end{align*}
Therefore, we can instead use $\tilde g_N^{(1)}(\bmp;\theta), \tilde g_N^{(2)}(\bmp;\theta), \tilde g_N^{(3)}(\bmp;\theta),\cdots$
to compute the joint moments of the random measures $\mu[\bmla^{(0)}], \mu[\bmla^{(1)}], \mu[\bmla^{(2)}], \cdots$ as defined in \eqref{def:mula}. In the following we check
\begin{align*}
\tilde H^{(t)}(\bmp;\theta)=F_N(\bmp;\theta)\prod_{s=1}^{t}\tilde g_N^{(s)}(\bmp;\theta)
=\frac{J_{\bmla^{(0)}}(\bmp;\theta)}{J_{\bmla^{(0)}}(1^N;\theta)}
\frac{H(t\rho, \bmp; \theta)}{H(t\rho, 1^N; \theta)}, 
\end{align*}
for $t=\lfloor N\tau\rfloor$ satisfies Assumption \ref{asup:infinite}, \ref{a:LLN} and \ref{a:CLT}. Since $J_{\bmla^{(0)}}(\bmp;\theta)$ is a finite sum of power sum symmetric functions, and from Proposition \ref{p:pstable}, $H(t\rho,\bmp;\theta)$ satisfies Assumption \ref{asup:infinite} for any $t\geq 1$, we conclude that $\tilde H^{(t)}(\bmp;\theta)$ satisfies Assumption \ref{asup:infinite} for $t=\lfloor N\tau\rfloor$.
Moreover, thanks to Theorem \ref{t:jack} and Theorem \ref{thm:equivalent}, for any $k\geq 0$,
\begin{align*}
 \fc^{(\tau)}_k=\lim_{N\rightarrow \infty}\left.\frac{\del_i^k \ln \tilde H^{(\lfloor N\tau\rfloor)}_N(x_1,x_2,\cdots,x_N; \theta)}{N}\right|_{1^N}=\theta\left(\tau\sum_{i\geq 1}\frac{k!}{i}{i\choose k} \rho_i+\left.\del_x ^k H_{\mu^{(0)}}(x)\right|_{x=1}\right).
\end{align*}
For any distinct indices $i,j$ and any exponents $k,l\geq 1$,
\begin{align*}\begin{split}
\fd_{k,l}^{(\tau)}&=\lim_{N\rightarrow \infty}\del_i^k\del_j^l\ln\left.\tilde H_N^{(\lfloor N\tau\rfloor)}(x_1,x_2,\cdots, x_N;\theta)\right|_{1^N}\\
&= \theta \left.\del_{x_1}^k\del_{x_2}^{l}\ln\left(1+(x_1-1)(x_2-1)\frac{x_1 H_{\mu^{(0)}}'(x_1)-x_2H_{\mu^{(0)}}'(x_2)}{x_1-x_2}\right)\right|_{x_1=x_2=1}.
\end{split}\end{align*}
For any $r\geq 3$ and any indices $i_1,i_2,\cdots, i_r$ such that there are at least three distinct indices among them,
\begin{align*}
\lim_{N\rightarrow \infty}\del_{i_1}\del_{i_2}\cdots \del_{i_r}\ln\left.\tilde H_N^{(\lfloor N\tau\rfloor)}(x_1,x_2,\cdots,x_N;\theta)\right|_{1^N}= 0.
\end{align*}
Thus $\tilde H_N^{(\lfloor N\tau\rfloor)}(\bmp;\theta)$ satisfies Assumptions \ref{a:LLN} and \ref{a:CLT}. The limit functions \eqref{eqn:UV} are given by
\begin{align*}\begin{split}
U^{(\tau)}_\mu(z)=
\sum_{k\geq 1}\frac{\fc_k^{(\tau)}}{(k-1)!}(z-1)^{k-1}
&=\sum_{k\geq 1}\frac{\theta (\tau\del_z^kW_\rho(z)+\del_z^k H_{\mu^{(0)}}(z)|_{z=1})}{(k-1)!}(z-1)^{k-1}\\
&=\theta (\tau W_\rho'(z)+H_{\mu^{(0)}}'(z)), 
\end{split}\end{align*}
and
\begin{align*}\begin{split}
V^{(\tau)}_\mu(z,w)
&=\sum_{k,l\geq 1}\frac{\fd_{k,l}^{(\tau)}}{(k-1!)(l-1)!}(z-1)^{k-1}(w-1)^{k-1}\\
&=
 \theta \del_z\del_w\ln \left(1+(z-1)(w-1)\frac{zH_{\mu^{(0)}}'(z)-wH_{\mu^{(0)}}'(w)}{z-w}\right).
\end{split}\end{align*}
The claim follows from Theorem \ref{thm:LLN} and Theorem \ref{t:mulclt}.
\end{proof}

We denote the limit empirical density 
\begin{align}\label{e:limited}
\lim_{N\rightarrow \infty} \mu[\bmla^{(\lfloor N\tau\rfloor)}]=\mu^{(\tau)}.
\end{align}
It turns out that the limit density $\mu^{(\tau)}$ can be characterized by a Burger's type equation.
\begin{corollary}\label{c:LLN}
Fix $\fb>0$ large, and a stable Jack positive specialization $\rho$ as defined in Definition \ref{def:stable}. Let $\{ \bmla^{(0)} \in \Y(N) \}_{N \geq 1}$ be a sequence of Young diagrams, such that $\supp \mu[\bmla^{(0)}]\subset [-\fb, \fb]$ and
\begin{align*}
\lim_{N\rightarrow \infty}\mu[\bmla^{(0)}]=\mu^{(0)}.
\end{align*} 
 The limit empirical density $\mu^{(\tau)}$, as defined in \eqref{e:limited} is characterized by the following equation,
\begin{align}\label{e:defmtau}
m_{\mu^{(\tau)}}(z+\tau T_\rho(e^{m_{\mu^{(0)}}(z)}))=m_{\mu^{(0)}}(z),
\end{align}
where $m_{\mu^{(0)}}(z)$ and $m_{\mu^{(\tau)}}(z)$ are the Stieltjes transforms of $\mu^{(0)}$ and $\mu^{(\tau)}$ respectively, and  $\rho=(\gamma, \bm \alpha, \bm \beta)$, 
\begin{align*}
T_\rho(z)
=\gamma z+\sum_{i\geq 1}\frac{\alpha_i z}{1-\alpha_i z}+\sum_{i\geq 1}\frac{\beta_i z}{1+\theta \beta_i z}.
\end{align*}
\end{corollary}

%
\begin{proof}
We first prove that the relation \eqref{e:defmtau} defines a probability measure $\mu^{(\tau)}$, and check that the moments of the probability measure $\mu^{(\tau)}$ match with those of $\mu[\bmla^{(\lfloor N\tau\rfloor)}]$, i.e.,
\begin{align}\label{e:equalmoment}
\lim_{N\rightarrow\infty}\int x^k \rd \mu[\bmla^{(\lfloor N\tau\rfloor)}]
=\int x^k \rd \mu^{(\tau)}(x).
\end{align}

Since $\mu^{(0)}$ is a probability measure, with density bounded by $1$ with respect to the Lebesgue measure, we have
\begin{align*}
e^{m_{\mu^{(0)}}(z)}\sim 1+\frac{1}{z},\quad z\rightarrow \infty,
\end{align*}
and for any $z\in \bC_+$,
\begin{align*}
-\pi< \Im[m_{\mu^{(0)}}(z)]< 0, \quad \Im[e^{m_{\mu^{(0)}}(z)}]<0.
\end{align*}
Therefore, we have as $z\rightarrow \infty$,
\begin{align*}
T_\rho(e^{m_{\mu^{(0)}}(z)})-\sum_{i\geq 1}\rho_i
&= \gamma e^{m_{\mu^{(0)}}(z)}+\sum_{i\geq 1}\frac{\alpha_i e^{m_{\mu^{(0)}}(z)}}{1-\alpha_i e^{m_{\mu^{(0)}}(z)}}+\sum_{i\geq 1}\frac{\beta_i e^{m_{\mu^{(0)}}(z)}}{1+\theta \beta_i e^{m_{\mu^{(0)}}(z)}}-\sum_{i\geq 1}\rho_i\\
&\sim\left(\gamma+\frac{\alpha_i}{(1-\alpha_i)^2}+\frac{\beta_i}{(1+\theta\beta_i)^2}\right)\frac{1}{z},
\end{align*}
and for any $z\in \bC_+$,
\begin{align*}
\Im\left[T_\rho(e^{m_{\mu^{(0)}}(z)})-\sum_{i\geq 1}\rho_i\right]
= \gamma \Im[ e^{m_{\mu^{(0)}}(z)}]+\sum_{i\geq 1}\frac{\alpha_i \Im[e^{m_{\mu^{(0)}}(z)}]}{|1-\alpha_i e^{m_{\mu^{(0)}}(z)}|^2}+\sum_{i\geq 1}\frac{\beta_i \Im[e^{m_{\mu^{(0)}}(z)}]}{|1+\theta \beta_i e^{m_{\mu^{(0)}}(z)}|^2}
<0.
\end{align*}
It follows from the classical Akhiezer-Krein theorem \cite[Chapter 3]{MR0184042}, there exists a positive measure $\mu_\rho^{(0)}$ such that
\begin{align}\label{e:defmrho0}
m_{\mu_\rho^{(0)}}(z)=T_\rho(e^{m_{\mu^{(0)}}(z)})-\sum_{i\geq 1}\rho_i,
\end{align}
where $m_{\mu^{(0)}}(z)$ and $m_{\mu_\rho^{(0)}}(z)$ are Stieltjes transforms of $\mu^{(0)}$ and $\mu_\rho^{(0)}$ respectively.   Therefore, we have
\begin{align*}
z+\tau T_\rho(e^{m_{\mu^{(0)}}(z)})=z+\tau\sum_{i\geq 1} \rho_i+\tau m_{\mu_\rho^{(0)}}(z).
\end{align*}
It follows from \cite[Lemma 4]{MR1488333}, that $z+\tau T_\rho(e^{m_{\mu^{(0)}}(z)})$ is conformal from 
\begin{align} \label{e:defOmega}
\Omega_\tau\deq\left\{z\in \bC\setminus \bR: \int \frac{\rd \mu_\rho^{(0)}(x)}{|x-z|^2}<\frac{1}{\tau}\right\},
\end{align}
to $\bC\setminus \bR$, and is a homeomorphism from the closure of $\Omega_\tau\cap \bC_+$ to $\bC_+\cup \bR$, and from the closure of $\Omega_\tau\cap \bC_-$ to $\bC_-\cup \bR$. Therefore there exists a well defined function $m_{\mu^{(\tau)}}(w)$ on $\bC_+$, such that \eqref{e:defmtau} holds. Moreover, $\Im[m_{\mu^{(\tau)}}(w)]<0$ for any $w\in \bC_+$ and $m_{\mu^{(\tau)}}(w)\sim 1/w$ as $w\rightarrow \infty$. Again, from the Akhiezer-Krein theorem \cite[Chapter 3]{MR0184042}, $m_{\mu^{(\tau)}}(w)$ is the Stieltjes transform of a probability measure $\mu^{(\tau)}$.

In the following we prove \eqref{e:equalmoment}. Both $m_{\mu^{(0)}}(z)$ and $m_{\mu^{(\tau)}}(z)$ are invertible in a small neighborhood of $z=\infty$. \eqref{e:defmtau} implies that in a small neighborhood of $z=0$,
\begin{align*}
m_{\mu^{(\tau)}}^{-1}(z)=\tau T_\rho(e^{z})+m_{\mu^{(0)}}^{-1}(z).
\end{align*}
We also notice the for the Jack positive specialization $\rho=(\gamma, \bm\alpha,\bm \beta)$, we have
\begin{align}\label{e:W'T}
W_\rho(z)=\sum_{n\geq 1}\frac{\rho_n z^n}{n}
=\gamma z -\sum_{i\geq 1}\ln(1-\alpha_iz)+\frac{1}{\theta}\sum_{i\geq 1}\ln(1+\theta \beta_i z),\quad zW_\rho'(z)=T_\rho(z).
\end{align}
Therefore, in a small neighborhood of $z=0$,
\begin{align}\begin{split}\label{e:factor1}
&\phantom{{}={}}\frac{1}{z}+\frac{(z+1)U_\mu^{(\tau)}(z+1)}{\theta}
=\frac{1}{z}+(z+1)(\tau W_\rho'(z+1)+H_{\mu^{(0)}}'(z+1))\\
&=\frac{1}{z}+\tau T_{\rho}(z+1)+(z+1)\left(\frac{R_{\mu^{(0)}}(\ln(z+1))+1}{z+1}+\frac{1}{(z+1)\ln(z+1)}-\frac{1}{z}\right)\\
&=\tau T_{\rho}(z+1)+R_{\mu^{(0)}}(\ln(z+1))+\frac{1}{\ln(z+1)}\\
&=\tau T_{\rho}(z+1)+m_{\mu^{(0)}}^{-1}(\ln (z+1))
=m_{\mu^{(\tau)}}^{-1}(\ln (z+1)).
\end{split}\end{align}
We can rewrite the expression \eqref{eqn:moments} as a contour integral and use \eqref{e:factor1},
\begin{align*}\begin{split}
&\phantom{{}={}}\lim_{N\rightarrow\infty}\int x^k \rd \mu[\bmla^{(\lfloor N\tau\rfloor)}]
=\frac{1}{2\pi \ri(k+1)}\oint_{|z|=\varepsilon} \left(\frac{1}{z}+\frac{(z+1)U_\mu^{(\tau)}(z+1)}{\theta}\right)^{k+1}\frac{\rd z}{z+1}\\
&=\frac{1}{2\pi \ri(k+1)}\oint_{|z|=\varepsilon} \left(m_{\mu^{(\tau)}}^{-1}(\ln (z+1))\right)^{k+1}\frac{\rd z}{z+1}
=\frac{1}{2\pi \ri(k+1)}\oint_{|z|=\varepsilon} \left(m_{\mu^{(\tau)}}^{-1}(z)\right)^{k+1}\frac{\rd e^z}{e^z}\\
&=\frac{1}{2\pi \ri(k+1)}\oint_{|z|=1/\varepsilon} z^{k+1}\rd m_{\mu^{(\tau)}}(z)=\int x^k \rd \mu^{(\tau)}(x).
\end{split}\end{align*}
This finishes the proof of Corollary \ref{c:LLN}.
\end{proof}

Next, we identify the fluctuation of the rescaled empirical density $\{N(\mu[\bmla^{(t)}]-\bE[\mu[\bmla^{(t)}]])\}_{t\geq 0}$ with Gaussian Free Field.   
Let $\bH=\{z\in \bC: \Im[z]>0\}$ be the upper half plane. The Gaussian Free Field $\fG$ on $\bH$ with zero boundary conditions, see e.g. \cite{MR2322706}, is a probability measure on a suitable class of generalized functions on $\bH$, and  can be characterized as follows. We take any sequence $\{\phi_k\}_{k\geq 1}$ of compactly supported test functions, the pairings
\begin{align*}
\int_\bH \phi_k(z)\fG(z)|\rd z|^2\deq \fG(\phi_k),\quad k\geq 1,
\end{align*}
form a sequence of mean zero Gaussian random variables with covariance matrix
\begin{align*}
\bE\left[\fG(\phi_k), \fG(\phi_l)\right]
=\int_{\bH}|\rd z_1|^2 |\rd z_2|^2 \phi_k(z_1)\phi_l(z_2) G(z_1,z_2),
\end{align*}
where 
\begin{align*}
G(z,w)=-\frac{1}{2\pi}\ln \left|\frac{z-w}{z-\bar w}\right|.
\end{align*}
is the Green function of the Laplacian on $\bH$ with Dirichlet boundary conditions.  One can make sense of the integrals
$\int
f(z)\fG(z)\rd z$ over finite contours in $\bH$ with continuous functions $f(z)$.

Given a nonintersecting random walk $\bmx^{(0)}, \bmx^{(1)},\bmx^{(2)},\cdots$ as defined in \eqref{e:chain}, we introduce the height function $H_N:  \bR\times \bR_{\geq 0}\mapsto \bZ_{\geq 0}$:
\begin{align*}
H_N(x,\tau)=|\{1\leq i\leq N: x_i^{(\lfloor N\tau\rfloor)}\geq  x\}|.
\end{align*}
We define the map $z\mapsto x(z)+\ri \tau(z)$ from $\bH$ to $\bH$,
\begin{align*}
 (x(z),\tau(z))=\left(\frac{zT_\rho(e^{m_{\mu^{(0)}}(\bar z)})-\bar zT_\rho(e^{m_{\mu^{(0)}}( z)})}{T_\rho(e^{m_{\mu^{(0)}}(\bar z)})-T_\rho(e^{m_{\mu^{(0)}}(z)})}, \frac{z-\bar z}{T_\rho(e^{m_{\mu^{(0)}}(\bar z)})-T_\rho(e^{m_{\mu^{(0)}}(z)})}\right).
\end{align*}
We note that the expressions for $x(z)$ and $\tau(z)$ are invariant with respect to complex conjugate, so $x(z)$ and $\tau(z)$ are indeed real for any $z\in \bH$. Since $x(z)=z+\tau(z)T_\rho(e^{m_{\mu^{(0)}}(z)})\in \bR$, the map $z\mapsto (x(z),\tau(z))$ is in fact a diffeomorphism from $\bH$ to its image. We define the pull back height function on $\bH$,
\begin{align}\label{e:defHn}
H_N(z)\deq H_N(x(z), \tau(z)),\quad z\in \bH.
\end{align}
One might worry that some information is lost in this transformation, as the image of
the map $z\mapsto (x(z),\tau(z))$ is smaller than $\bH$, yet the height function $H_N(x,\tau)$ is actually
frozen outside this image and there are no fluctuations to study. Next corollary states that the height function $H_N(z)$ converges to the Gaussian Free Field on $\bH$ with zero boundary conditions.

\begin{corollary}\label{c:GFF}
Fix $\fb>0$ large, and a stable Jack positive specialization $\rho$ as defined in Definition \ref{def:stable}. Let $\{ \bmla^{(0)} \in \Y(N) \}_{N \geq 1}$ be a sequence of Young diagrams, such that $\supp \mu[\bmla^{(0)}]\subset [-\fb, \fb]$ and
\begin{align*}
\lim_{N\rightarrow \infty}\mu[\bmla^{(0)}]=\mu.
\end{align*} 
 Let $H_N(z)$ be the random height function on $\bH$ as defined in \eqref{e:defHn}, then as $N$ goes to infinity,
\begin{align*}
\sqrt{\pi\theta}\left(H_N(z)-\bE[H_N(z)]\right)\rightarrow \fG(z).
\end{align*}
In more details, the collection of random variables,
\begin{align*}
\sqrt{\pi\theta}\int_\bR x^k (H_N(x, \tau)-\bE[H_N(x,\tau)])\rd x, \quad  k=1,2,3,\cdots,\quad \tau\geq 0,
\end{align*}
converges as $N$ goes to infinity, in the sense of moments to the corresponding moments of Gaussian Free Field,
\begin{align*}
\int_{z\in \bH, \tau(z)=\tau} x^k(z)\fG(z) \rd x(z),\quad k=1,2,3,\cdots,\quad \tau\geq 0.
\end{align*}
\end{corollary}

\begin{proof}
We perform an integration by part, for $k=1,2,3,\cdots$ and $\tau\geq 0$, 
\begin{align*}
\sqrt{\pi\theta}\int_\bR x^k (H_N(x,\tau)-\bE[H_N(x,\tau)])\rd x
=\frac{N\sqrt{\pi\theta}}{k+1}\int_{\bR} x^{k+1}\left(\rd \mu[\bmla^{(\lfloor N\tau\rfloor)}]-\bE[\rd\mu[\bmla^{(\lfloor N\tau\rfloor)}]]\right),
\end{align*}
which converges as $N$ goes to infinity, in the sense of moments, to a collection of Gaussian random variables, by Theorem \ref{t:nrw}. In the following, we identify the covariance structure. 
By Theorem \ref{t:nrw}, for any $\tau\leq\sigma$, the covariance 
\begin{align}\begin{split}\label{e:covcc}
&\phantom{{}={}}\cov\left[\frac{N\sqrt{\pi\theta}}{k+1}\int_{\bR} x^{k+1}\left(\rd \mu[\bmla^{(\lfloor N\tau\rfloor)}]-\bE[\rd\mu[\bmla^{(\lfloor N\tau\rfloor)}]]\right), \frac{N\sqrt{\pi\theta}}{l+1}\int_{\bR} x^{l+1}\left(\rd \mu[\bmla^{(\lfloor N\sigma\rfloor)}]-\bE[\rd\mu[\bmla^{(\lfloor N\sigma\rfloor)}]]\right)\right]\\
&=-\frac{1}{4\pi(k+1)(l+1)}\oint_{|z|=\varepsilon/2\atop |w|=\varepsilon/3}\rd z\rd w
\del_z\del_w\ln \left(\frac{1}{z}+(z+1)H_{\mu^{(0)}}'(z+1)-\frac{1}{w}-(w+1)H_{\mu^{(0)}}'(w+1)\right)
\\
&\times\left(\frac{1}{z}+(z+1)(\tau W_\rho'(z+1)+H_{\mu^{(0)}}'(z+1)) \right)^{k+1}
\left(\frac{1}{w}+(w+1)(\tau W_\rho'(w+1)+H_{\mu^{(0)}}'(w+1))\right)^{l+1},
\end{split}\end{align}
where we used
\begin{align*}\begin{split}
&\phantom{{}={}}\left(\frac{1}{\theta}\frac{1}{(z-w)^2}+\frac{V^{(\tau)}_\mu(z+1,w+1)}{\theta^2}\right)\\
&=\frac{1}{\theta}\frac{1}{(z-w)^2}+ \frac{1}{\theta}\del_z\del_w\ln \left(1+zw\frac{(z+1)H_{\mu^{(0)}}'(z+1)-(w+1)H_{\mu^{(0)}}'(w+1)}{z-w}\right)\\
&=\frac{1}{\theta}\del_z\del_w\ln \left(\frac{1}{z}+(z+1)H_{\mu^{(0)}}'(z+1)-\frac{1}{w}-(w+1)H_{\mu^{(0)}}'(w+1)\right).
\end{split}\end{align*}
We make a change of variables in \eqref{e:covcc},
\begin{align}\begin{split}\label{e:change1}
\tilde z&=\frac{1}{z}+(z+1)H_{\mu^{(0)}}'(z+1)=m^{-1}_{\mu^{(0)}}(\ln(z+1)),\\
\tilde w&=\frac{1}{w}+(w+1)H_{\mu^{(0)}}'(w+1)=m^{-1}_{\mu^{(0)}}(\ln(w+1)),
\end{split}\end{align}
this change of variables is well-defined since we are dealing with analytic functions in a neighborhood of the origin. Moreover, with the change of variables, we have
\begin{align}\begin{split}\label{e:change2}
\frac{1}{z}+(z+1)(\tau W_\rho'(z+1)+H_{\mu^{(0)}}'(z+1))
&=\tilde z+\tau T_\rho(e^{m_{\mu^{(0)}}(\tilde z)}),\\
\frac{1}{w}+(w+1)(\sigma W_\rho'(w+1)+H_{\mu^{(0)}}'(w+1))
&=\tilde w+\sigma T_\rho(e^{m_{\mu^{(0)}}(\tilde w)}),
\end{split}\end{align}
where we used \eqref{e:W'T}.
By plugging \eqref{e:change1} and \eqref{e:change2} into \eqref{e:covcc}, we get
\begin{align}\begin{split}\label{e:covcc2}
&\phantom{{}={}}\cov\left[\frac{N\sqrt{\pi\theta}}{k+1}\int_{\bR} x^{k+1}\left(\rd \mu[\bmla^{(\lfloor N\tau\rfloor)}]-\bE[\rd\mu[\bmla^{(\lfloor N\tau\rfloor)}]]\right), \frac{N\sqrt{\pi\theta}}{l+1}\int_{\bR} x^{l+1}\left(\rd \mu[\bmla^{(\lfloor N\sigma\rfloor)}]-\bE[\rd\mu[\bmla^{(\lfloor N\sigma\rfloor)}]]\right)\right]\\
&=-\frac{1}{4\pi}\oint_{|z|=1/\varepsilon\atop|w|=2/\varepsilon}\frac{\rd z\rd w
\del_z\del_w\ln(z-w)}{(k+1)(l+1)}\left(z+\tau T_\rho(e^{m_{\mu^{(0)}}(z)}) \right)^{k+1}
\left(w+\sigma T_\rho(e^{m_{\mu^{(0)}}(w)}) \right)^{l+1}.
\end{split}\end{align}

We recall the measure $\mu_\rho^{(\tau)}$ from the proof of Corollary \ref{c:LLN}, and  $z+\tau T_\rho(e^{-m_0(z)})$ is conformal from $\Omega_\tau$ as defined in \eqref{e:defOmega} to $\bC\setminus \bR$, and is a homeomorphism from the closure of $\Omega_\tau\cap \bC_+$ to $\bC_+\cup \bR$, and from the closure of $\Omega_\tau \cup \bC_-$ to $\bC_-\cup \bR$.
We denote $\gamma_\tau=\del \Omega_\tau \cap \bC_+$, the boundary of $\Omega_\tau$ in the upper half plane. From \eqref{e:defOmega}, $\gamma_\tau$ is explicitly given by
\begin{align*}
\gamma_\tau \deq\left\{z\in \bC_+: \int \frac{\rd \mu_\rho^{(0)}(x)}{|x-z|^2}=\frac{1}{\tau}\right\}=\{z\in \bC_+: \tau(z)=\tau\},
\end{align*}
and for $z\in \gamma_t$, $ z+\tau T_\rho(e^{m_{\mu^{(0)}}( z)})\in \bR$.
We can deform the contours in \eqref{e:covcc2}, and perform an integration by part
\begin{align}\begin{split}\label{e:covcc3}
&\phantom{{}={}}\cov\left[\frac{N\sqrt{\pi\theta}}{k+1}\int_{\bR} x^{k+1}\left(\rd \mu[\bmla^{(\lfloor N\tau\rfloor)}]-\bE[\rd\mu[\bmla^{(\lfloor N\tau\rfloor)}]]\right), \frac{N\sqrt{\pi\theta}}{l+1}\int_{\bR} x^{l+1}\left(\rd \mu[\bmla^{(\lfloor N\sigma\rfloor)}]-\bE[\rd\mu[\bmla^{(\lfloor N\sigma\rfloor)}]]\right)\right]\\
&=-\frac{1}{4\pi}\oint_{|z|=1/\varepsilon\atop|w|=2/\varepsilon}\frac{\rd z\rd w
\del_z\del_w\ln(z-w)}{(k+1)(l+1)}\left(z+\tau T_\rho(e^{m_{\mu^{(0)}}(z)}) \right)^{k+1}
\left(w+\sigma T_\rho(e^{m_{\mu^{(0)}}(w)}) \right)^{l+1}\\
&=-\frac{1}{4\pi}\oint_{z\in\gamma_\tau\cup\bar \gamma_\tau\atop
w\in\gamma_\sigma\cup\bar \gamma_\sigma}\frac{\rd z\rd w
\del_z\del_w\ln(z-w)}{(k+1)(l+1)}\left(z+\tau T_\rho(e^{m_{\mu^{(0)}}(z)}) \right)^{k+1}
\left(w+\sigma T_\rho(e^{m_{\mu^{(0)}}(w)}) \right)^{l+1}\\
&=-\frac{1}{4\pi}\oint_{z\in\gamma_\tau\cup\bar \gamma_\tau\atop
w\in\gamma_\sigma\cup\bar \gamma_\sigma}
\ln(z-w)\left(z+\tau T_\rho(e^{m_{\mu^{(0)}}(z)}) \right)^{k}
\left(w+\sigma T_\rho(e^{m_{\mu^{(0)}}(w)}) \right)^{l}\\
&\phantom{{}=-\frac{1}{4\pi}\oint_{z\in\gamma_\tau\cup\bar \gamma_\tau\atop
w\in\gamma_\sigma\cup\bar \gamma_\sigma}{}}
\times\rd \left(z+\tau T_\rho(e^{m_{\mu^{(0)}}(z)}) \right)
\rd \left(w+\sigma T_\rho(e^{m_{\mu^{(0)}}(w)}) \right)
\end{split}\end{align}
We notice that on those contours $z\in \gamma_\tau\cup\bar \gamma_\tau$, or $w\in \gamma_\sigma\cup\bar \gamma_\sigma$, $z+\tau T_\rho(e^{m_{\mu^{(0)}}(z)})=x(z)$ and $z+\sigma T_\rho(e^{m_{\mu^{(0)}}(z)})=x(w)$ are real. Using this fact and the equality
\begin{align*}
2\log \left|\frac{z-w}{z-\bar w}\right|=\log \frac{(z-w)(\bar z-\bar w)}{(z-\bar w)(\bar z-w)},
\end{align*}
we can rewrite \eqref{e:covcc3}
\begin{align*}
&\phantom{{}={}}\cov\left[\frac{N\sqrt{\pi\theta}}{k+1}\int_{\bR} x^{k+1}\left(\rd \mu[\bmla^{(\lfloor N\tau\rfloor)}]-\bE[\rd\mu[\bmla^{(\lfloor N\tau\rfloor)}]]\right), \frac{N\sqrt{\pi\theta}}{l+1}\int_{\bR} x^{l+1}\left(\rd \mu[\bmla^{(\lfloor N\sigma\rfloor)}]-\bE[\rd\mu[\bmla^{(\lfloor N\sigma\rfloor)}]]\right)\right]\\
&=\int_{z\in \bH, \tau(z)=\tau}\int_{w\in\bH,\tau(w)=\sigma}G(z,w) x(z)^kx(w)^l\rd x(z)\rd x(w)\\
&=\cov\left[\int_{z\in \bH, \tau(z)=\tau} x(z)^k\fG(z) \rd x(z), \int_{z\in \bH, \tau(z)=\sigma} x(z)^l\fG(z) \rd x(z)\right].
\end{align*}
This finishes the proof of Corollary \ref{c:GFF}.

\end{proof}

\appendix
\section{Basics of Jack Symmetric functions}
\label{s:Jack}

In this section we collect some basic properties of Jack symmetric functions. Our main references are \cite{MR3443860} and \cite{MR1014073}.

Given a Young diagram $\bmla$, a box $\Box\in \bmla$ is a pair of integers, 
\begin{align*}
\Box=(i,j)\in {\bm\lambda}, \text{ if and only if } 1\leq i\leq \ell(\bm\lambda), 1\leq j\leq \lambda_i.
\end{align*} 
We denote $\bmla'$ the transposed diagram  of $\bm\lambda$, defined by 
\begin{align*}
\lambda_j'=|\{i: 1\leq j\leq \lambda_i\}|, \quad 1\leq j\leq \la_1.
\end{align*} 
For a box $\Box=(i,j)\in \bmla$, its arm $a_\Box$ and leg $l_\Box$ are\begin{align*}
a_\Box=\lambda_i-j,\quad l_\Box=\lambda_j'-i.
\end{align*}
We recall the scalar product on $\Sym$ as defined in \eqref{e:scalarP}, and introduce the following notation,
\begin{align}\label{e:jl}
j_{\bmla}(\theta)=\langle J_{\bmla}(\bmx;\theta), J_{\bmla}(\bmx;\theta)\rangle.
\end{align}
The number $j_{\bmla}(\theta)$ is explicitly given by
\begin{align}\label{def:jl2}
j_{\bmla}=\prod_{\Box\in \bmla}\frac{a_{\Box}+\theta l_{\Box}+1} {a_{\Box}+\theta l_{\Box}+\theta}.
\end{align}
See \cite[Theorem 5.8]{MR1014073} for a proof.

\begin{theorem}\label{t:J1N}
We have 
\begin{align*}
J_{\bmla}(1^N;\theta)=\prod_{(i,j)\in \bmla}(N-(i-1)+(j-1)/\theta).
\end{align*}
\end{theorem}
\begin{proof}
See \cite[Theorem 5.4]{MR1014073}.
\end{proof}

\subsection{The Cauchy identity}

One of the most important property of Jack symmetric functions is the following Cauchy identity,
\begin{proposition}\label{p:Cauchy}
Let $\bmx=(x_1,x_2,\cdots)$ and $\bmy=(y_1,y_2,\cdots)$ be two sets of indeterminates. Then
\begin{align*}
\sum_{\bmla} J_{\bmla}(\bmx;\theta)J_{\bmla}(\bmy;\theta)j_{\bmla}^{-1}(\theta)=\prod_{ij}(1-x_iy_j)^{-\theta}.
\end{align*}
\end{proposition}

\begin{proof}
See \cite[Proposition 2.1]{MR1014073}.
\end{proof}
The kernel appears in the Cauchy identity
\begin{align}\label{e:jrk}
H(\bmx,\bmy;\theta)=\exp\left(\theta \sum_{n\geq 1}\frac{p_n(\bmx)p_n(\bmy)}{n}\right)=\prod_{ij}(1-x_iy_j)^{-\theta}
\end{align}
is called the Jack reproducing kernel.

\subsection{Skew Jack symmetric functions}
The skew Jack symmetric functions $J_{\bmla/\bmmu}(\bmx;\theta)$ can be defined as follows. Take two infinite sets of
variables $\bmx$ and $\bmy$ and consider a Jack symmetric function $J_{\bmla}(\bmx,\bmy;\theta)$. In particular, the later is a symmetric
function in the $\bmy$ variables. The coefficients $J_{\bmla/\bmmu}(\bmx;\theta)$ in its decomposition in the linear basis of Jack symmetric functions in the $\bmy$ variables are symmetric functions in the $\bmx$ variables and are referred
to as skew Jack symmetric functions:
\begin{align}\label{e:defskewJ}
J_{\bmla}(\bmx,\bmy;\theta)
=\sum_{\bmmu}J_{\bmla/\bmmu}(\bmx;\theta)J_{\bmmu}(\bmy;\theta).
\end{align}
There's another way to characterize these skew Jack symmetric functions $J_{\bmla/\bmmu}(\bmx;\theta)$.
\begin{proposition}\label{p:innerP}
The skew Jack symmetric functions $J_{\bmla/\bmmu}(\bmx;\theta)$ satisfy
\begin{align}
\langle J_{\bmla/\bmmu}(\bmx;\theta), J_{\bmnu}(\bmx;\theta)\rangle=j_{\bmmu}^{-1}(\theta)\langle J_{\bmla}(\bmx;\theta), J_{\bmmu}(\bmx;\theta)J_{\bmnu}(\bmx;\theta)\rangle. 
\end{align}
\end{proposition}
\begin{proof}
See \cite[Section 4]{MR1014073}.
\end{proof}

\begin{proposition}\label{p:chain2}
The skew Jack symmetric functions $J_{\bmla/\bmnu}$ satisfy
\begin{align}\label{e:chainequal}
J_{\bmla/\bmnu}(\bmx,\bmy;\theta)=\sum_{\bmmu}J_{\bmla/\bmmu}(\bmx;\theta)J_{\bmmu/\bmnu}(\bmy;\theta).
\end{align}
The defining relation \eqref{e:defskewJ} of the skew Jack symmetric functions can be viewed as a special case of \eqref{e:chainequal} with $\bmnu=\emptyset$.
\end{proposition}
\begin{proof}
See \cite[Proposition 4.2]{MR1014073}.
\end{proof}


\begin{proposition}\label{p:Stochastic}
We have the following identity
\begin{align}\label{e:id1}
\sum_{\bmla}j_{\bmla}^{-1}(\theta)J_{\bmla}(\bmx;\theta)J_{\bmla/\bmmu}(\bmy;\theta)
=j_{\bmmu}^{-1}(\theta)J_{\bmmu}(\bmx;\theta)\prod_{ij}(1-x_iy_j)^{-\theta}
\end{align}
\end{proposition}
\begin{proof}
See \cite[Section 4]{MR1014073}.
\end{proof}


\subsection{Jack positive speicialization and Jack measure}

A specialization $\rho$ is an algebra homomorphism from $\Sym$ to the set of complex numbers. A specialization $\rho$ is called Jack-positive if its values on all (skew) Jack symmetric functions with a fixed $\theta$ are real and non-negative, i.e.,
\begin{align*}
J_\bmla(\rho;\theta)=\rho(J_\bmla(\bmx;\theta))\geq 0, 
\end{align*}
for all $\bmla\in \bY$. The set of Jack positive specializations are classified by 
Kerov, Okounkov and Olshanski in
\cite{MR1609628}.
\begin{theorem}
For any fixed $\theta>0$, Jack positive specializations can be parameterized by triplets $(\bm\alpha, \bm\beta,\gamma)$, where $\bm\alpha, \bm\beta$ are sequences of real numbers with 
\begin{align*}
\alpha_1\geq \alpha_2\geq \cdots \geq 0, \quad
\beta_1\geq \beta_2\geq \cdots\geq 0,\quad
\sum_i (\alpha_i+\beta_i)<\infty,
\end{align*}
and $\gamma$ is a non-negative real number. The specialization corresponding to a triplet $(\bm\alpha, \bm \beta, \gamma)$ is given by its values on Newton power sums $p_k$,
\begin{align*}\begin{split}
&p_1\mapsto p_1(\bm\alpha,\bm \beta,\gamma)=\gamma+\sum_{i}(\alpha_i+\beta_i),\\
&p_k\mapsto p_k(\bm\alpha, \bm\beta, \gamma)=\sum_{i}\alpha_i^k +(-\theta)^{k-1} \sum_i \beta_i^k, \quad k\geq 2.
\end{split}\end{align*}
\end{theorem}

\begin{definition}
Given two specializations $\rho,\rho'$, we define their union $(\rho,\rho')$ through the formula:
\begin{align*}
p_k(\rho,\rho')=p_k(\rho)+p_k(\rho'),\quad k\geq 1.
\end{align*}
\end{definition}

Jack measures are examples of probability measures with closed-form product formulas for their Jack generating functions, which is a natural generalization of Okounkov's Schur measures \cite{MR1856553}.
They were first defined in \cite{MR3418747}, as degenerations of the more general Macdonald measures of Borodin-Corwin \cite{MR3152785}.

\begin{definition}\label{def:jackmeasure}
Let $\rho_1, \rho_2 : \Sym \rightarrow \R$ be two Jack-positive specializations.
The \textit{Jack measure associated to $(\rho_1, \rho_2)$} is the probability measure $\cJ_{\rho_1 \rho_2}$ on $\Y$ given by
\begin{equation*}\label{e:defMrr}
\cJ_{\rho_1 \rho_2}(\bflambda) := \frac{ J_{\bflambda}(\rho_1; \theta)J_{\bflambda}(\rho_2; \theta)}{j_{\bflambda}(\theta) H(\rho_1, \rho_2; \theta)},\quad \bflambda\in\Y.
\end{equation*}
In the formula above, $H(\bfx, \bfy; \theta)$ is the Jack reproducing kernel $(\ref{e:jrk})$, and so
\begin{equation}\label{speckernel}
H(\rho_1, \rho_2; \theta) = \exp\left( \theta\cdot\sum_{n\geq 1}{\frac{p_n(\rho_1)p_n(\rho_2)}{n}} \right).
\end{equation}
\end{definition}

\section{$N$-finite Jack generating functions}
\label{s:finiteJGF}
In a series of papers \cite{MR3361772,BG1,BG2} by Bufetov and Gorin, Schur generating functions are introduced to study the global behavior of stochastic discrete particle systems. The Schur generating functions in \cite{MR3361772,BG1,BG2} are defined using Schur polynomials in $N$ variables. However, the Jack generating functions \eqref{e:gener} in this paper are defined using Jack symmetric functions in infinitely many variables. In this section, we define Jack generating functions using Jack polynomials in $N$ variables, and prove that under Assumption \ref{asup:infinite}, our Assumptions \ref{a:LLN} and \ref{a:CLT} are equivalent to their counterparts in \cite{BG1}.

Before we define Jack generating functions using Jack polynomials in $N$ variables, we need more notations. For each $N\in\bZ_{>0}$, the algebra of symmetric polynomials in $N$ variables $x_1,x_2,\cdots,x_N$ is denoted $\Sym(N)$. 
A homogeneous basis of $\Sym(N)$ is the monomial symmetric polynomials $\{m_{\bflambda}(x_1, \ldots, x_N) : \bflambda \in \Y(N) \}$. 
The polynomial $m_{\bflambda}(x_1, \ldots, x_N)$ can be obtained from $m_{\bmla}(\bmx)$ by taking $x_{N+1}=x_{N+2}=\cdots=0$. 

Let  $N\in\bZ_{>0}$ and $M_N$ be a probability measure on $\Y(N)$.
The \textit{$N$-finite Jack generating function of $M_N$} is defined as
\begin{equation}\label{finitegener}
F_N(x_1, \ldots, x_N; \theta) := \sum_{\bflambda\in\Y(N)}
M_N(\bflambda)\frac{J_{\bflambda}(x_1, \ldots, x_N; \theta)}{J_{\bflambda}(1^N; \theta)}.
\end{equation}
 We observe that the sum in $(\ref{finitegener})$ is infinite; thus $F_N(x_1, \ldots, x_N; \theta)$ is not an element of $\Sym(N)$. We obtain the $N$-finite Jack generating function \eqref{finitegener} from Jack generating function \eqref{e:gener} by taking $x_{N+1}=x_{N+2}=\cdots=x_N=0$. Let $G(x_1,x_2,\cdots,x_N)$ be a symmetric polynomial, and $\bmeta\in \bY(N)$, we write $\del_\bmeta G(x_1,x_2,\cdots, x_N)|_{1^N}=\del_{i_1}^{\eta_1}\del_{i_2}^{\eta_2}\cdots \del_{i_{\ell(\bmeta)}}^{\eta_{\ell(\bmeta)}}G(x_1,x_2,\cdots, x_N)|_{1^N}$, where $i_1,i_2,\cdots, i_{\ell(\bmeta)}$ are distinct indices. We write $\bmla\subset \bmeta$, if there exist indices $i_1<i_2<\cdots<i_{\ell(\bmla)}\leq \ell(\bmeta)$ such that $\bmla=(\eta_{i_1},\eta_{i_2},\cdots, \eta_{i_{\ell(\bmla)}})$.

\begin{theorem}\label{thm:equivalent}
Let $\{M_N\}_{N\geq 1}$ be a sequence of probability measures on $\bY(N)$ with Jack generating function $F_N(\bmp;\theta)$ as defined in \eqref{e:gener}. We assume that the functions $F_N(\bmp;\theta)$ satisfy Assumption \ref{asup:infinite}. Then 
 Assumption \ref{a:LLN} is equivalent to that, the $N$-finite Jack generating function $F_N(x_1,x_2,\cdots,x_N;\theta)$ as defined in \eqref{finitegener} satisfies
\begin{enumerate}
\item For any index $i\geq 1$ and any exponent $k\geq 1$,
\begin{align}\label{e:deri1}
\lim_{N\rightarrow \infty}\left.\frac{\del_i^k \ln F_N(x_1,x_2,\cdots,x_N;\theta)}{N}\right|_{1^N}= \fc_k.
\end{align}
\item 
For any $r\geq 2$ and any indices $i_1,i_2,\cdots, i_r$ such that there are at least two distinct indices among them,
\begin{align}\label{e:deri2}
\lim_{N\rightarrow \infty}\left.\frac{\del_{i_1}\del_{i_2}\cdots\del_{i_r} \ln F_N(x_1,x_2,\cdots,x_N;\theta)}{N}\right|_{1^N}= 0.
\end{align}
\end{enumerate}
And Assumption \ref{a:CLT} is equivalent to that the Jack generating function $F_N(x_1,x_2,\cdots,x_N;\theta)$ as defined in \eqref{finitegener} satisfies
\begin{enumerate}
\item For any index $i\geq 1$ and any exponent $k\geq 1$,
\begin{align}\label{e:deri1copy}
\lim_{N\rightarrow \infty}\left.\frac{\del_i^k \ln F_N(x_1,x_2,\cdots,x_N;\theta)}{N}\right|_{1^N}= \fc_k.
\end{align}
\item 
For  any distinct indices $i,j$ and any exponent $k,l\geq 1$,
\begin{align}\label{e:deri3}
\lim_{N\rightarrow \infty}\left.\del_i^k\del_j^l \ln F_N(x_1,x_2,\cdots,x_N;\theta)\right|_{1^N}= \fd_{k,l}.
\end{align}
\item
For any $r\geq 3$ and any indices $i_1,i_2,\cdots, i_r$ such that there are at least three distinct indices among them,
\begin{align}\label{e:deri4}
\lim_{N\rightarrow \infty}\left.\del_{i_1}\del_{i_2}\cdots\del_{i_r} \ln F_N(x_1,x_2,\cdots,x_N;\theta)\right|_{1^N}= 0.
\end{align}
\end{enumerate}
\end{theorem}

\begin{proposition}\label{p:changeDer}
Let $G(\bmp)=\sum_{\bmla\in \bY}a(\bmla)p_\bmla\in \widehat\Sym$ satisfy Assumption \ref{asup:infinite} with $G(\bmp)|_{\bmp=1^N}=1$. Then for any fixed $r\geq 1$, the infinite series
\begin{align}\label{e:summable}
\sum_{j_1,j_2,\cdots,j_r\geq 1}\left|\left.\frac{\del^{r}\ln G(\bmp)}{\del p_{j_1}\del p_{j_2}\cdots \del p_{j_{r}}}\right|_{\bmp=1^N}\right||x|^{j_1+j_2+\cdots+j_r},
\end{align}
converges absolutely on $\{z\in\bC: |z|<1+\varepsilon\}$.

Let $G(x_1,x_2,\cdots, x_N)$ be obtained from $G(\bmp)$ by taking $p_n=x_1^n+x_2^n+\cdots+x_N^n$ for $n\geq 1$. Then $G(x_1,x_2,\cdots, x_N)$ is analytic on
\begin{align}\label{e:domain}
\{(z_1,z_2,\cdots,z_N)\in \bC^{N}: |z_i|<1+\varepsilon,\quad i=1,2,\cdots,N\}.
\end{align}
We recall the coefficients $c_{\bmla\bmeta}$ and $\tilde c_{\bmla\bmeta}$ from \eqref{def:c} and \eqref{def:tc}. For any partition $\bmeta\in \bY(N)$, we have 
\begin{align}\begin{split}\label{eq:derG1}
&\phantom{{}={}}
\del_{\bmeta}\ln G(x_1,x_2,\cdots, x_N)|_{1^N}\\&={\prod_{i=1}^{\ell(\bmeta)}\eta_i!}\sum_{\bmmu\in\bY:
|\bmmu|=|\bmeta|\atop\ell(\bmmu)\geq\ell(\bmeta)}\frac{c_{\bmmu\bmeta}}{\prod_{i\geq 1}m_i(\bmmu)!}\sum_{\la_1,\la_2,\cdots\la_{\ell(\bmmu)}\geq 1}  \prod_{i=1}^{\ell(\bmmu)}{\la_i\choose\mu_{i}} \left.\frac{\del^{\ell(\bmmu)}\ln G(\bmp)}{\del p_{\la_1}\del p_{\la_2}\cdots \del p_{\la_{\ell(\bmmu)}}}\right|_{\bmp=1^N}.
\end{split}\end{align}
For any partition $\bmmu\in \bY(N)$, we have
\begin{align}\begin{split}\label{eq:derG2}
&\phantom{{}={}}\sum_{\la_1,\la_2,\cdots\la_{\ell(\bmmu)}\geq 1}  \prod_{i=1}^{\ell(\bmmu)}{\la_i\choose\mu_{i}} \left.\frac{\del^{\ell(\bmmu)}\ln G(\bmp)}{\del p_{\la_1}\del p_{\la_2}\cdots \del p_{\la_{\ell(\bmmu)}}}\right|_{\bmp=1^N}\\
&=\prod_{i=1}^{\mu_1}m_i(\bmmu)!\sum_{\bmeta\in \bY(N):|\bmeta|=|\bmmu|\atop \ell(\bmeta)\geq\ell(\bmmu)}\tilde c_{\bmeta\bmmu}\frac{\del_{\bmeta}\ln G(x_1,x_2,\cdots,x_N)|_{1^N}}{\prod_{i=1}^{\ell(\bmeta)}\eta_i!}.
\end{split}\end{align}
\end{proposition}

\begin{proof}
We denote $\Delta p_n=p_n-N$ and $\Delta x_n=x_n-1$ for $n=1,2,3,\cdots$. We can rewrite $G(\bmp)$ as
\begin{align}\label{e:Gmp}
G(\bmp)=1+\sum_{\bmla\in \bY\setminus \emptyset}c(\bmla) (\Delta p)_\bmla, \quad (\Delta p)_\bmla=\Delta p_{\la_1}\Delta p_{\la_2}\cdots\Delta p_{\la_{\ell(\bmla)}},
\end{align}
where  for any $\bmla\in \bY\setminus \emptyset$,
\begin{align*}
c(\bmla)=\sum_{\bmmu\in \bY:\bmla\subset \bmmu}a(\bmmu)N^{\ell(\bmmu)-\ell(\bmla)}.
\end{align*}
Formally, we have
\begin{align}\label{e:defcla}
\ln G(\bmp)=\sum_{k\geq 1}\frac{(-1)^k}{k}\left(\sum_{\bmla\in \bY\setminus \emptyset}c(\bmla) (\Delta p)_\bmla\right)^k
=\sum_{\bmla\in \bY}\tilde c(\bmla) (\Delta p)_\bmla,
\end{align}
where 
\begin{align}\label{e:defcla2}
\tilde c(\bmla)=\sum_{k\geq 1}\frac{(-1)^k}{k}\sum_{\bmla_1,\bmla_2,\cdots,\bmla_k\in \bY\setminus \emptyset,\atop\bmla_1\cup\bmla_2\cup\cdots\cup \bmla_k=\bmla}c(\bmla_1)c(\bmla_2)\cdots c(\bmla_k)=\frac{1}{\prod_{i\geq 1}m_i(\bmla)!}\left.\frac{\del^{\ell(\bmla)}\ln G(\bmp)}{\del{p_{\la_1}}\del{p_{\la_2}}\cdots \del{p_{\la_{\ell(\bmla)}}}}\right|_{\bmp=1^N}.
\end{align}
For any $r\geq 1$, and $1\leq |x|<1+\varepsilon$ we have
\begin{align*}\begin{split}
&\phantom{{}={}}\sum_{j_1,j_2,\cdots,j_r\geq 1}\left|\left.\frac{\del^{r}\ln G(\bmp)}{\del p_{j_1}\del p_{j_2}\cdots \del p_{j_{r}}}\right|_{\bmp=1^N}\right||x|^{j_1+j_2+\cdots+j_r}
= r!\sum_{\bmla\in \bY:\ell(\bmla)=r}|\tilde c(\bmla)||x|^{|\bmla|}\\
&\leq  r!\sum_{k\geq 1}\frac{1}{k}\sum_{\bmla_1,\bmla_2,\cdots,\bmla_k\in\bY\setminus\emptyset,\atop\ell(\bmla_1)+\ell(\bmla_2)+\cdots+\ell(\bmla_k)=r}\prod_{i=1}^k |c(\bmla_i)||x|^{|\bmla_i|}\\
&\leq  r!\sum_{k\geq 1}\frac{1}{k}\sum_{\bmla_1,\bmla_2,\cdots,\bmla_k\in\bY\setminus\emptyset,\atop\ell(\bmla_1)+\ell(\bmla_2)+\cdots+\ell(\bmla_k)=r}N^{-r}\prod_{i=1}^k\sum_{\bmmu_i\in\bY:\bmla_i\subset \bmmu_i}|a(\bmmu_i)|N^{\ell(\bmmu_i)}|x|^{|\bmla_i|}\\
&\leq r!N^{-r}\sum_{k\geq 1}\frac{1}{k}\sum_{\bmmu_1,\bmmu_2,\cdots,\bmmu_k\in\bY}\prod_{i=1}^k|a(\bmmu_i)|N^{\ell(\bmmu_i)}\sum_{\bmla_i\in\bY\setminus\emptyset:\bmla_i\subset\bmmu_i\atop \ell(\bmla_1)+\ell(\bmla_2)+\cdots+\ell(\bmla_k)=r}|x|^{|\bmla_1|+|\bmla_2|+\cdots+|\bmla_k|}\\
&\leq r!N^{-r}\sum_{r_1,r_2,\cdots,r_k\geq 1, \atop r_1+r_2+\cdots+r_k=r}\frac{1}{k}\prod_{i=1}^k\sum_{\bmmu_i\in \bY}|a(\bmmu_i)|N^{\ell(\bmmu_i)}{\ell(\bmmu_i)\choose r_i}|x|^{|\bmmu_i|},
\end{split}\end{align*}
which converges absolutely by Assumption \ref{asup:infinite}. This finishes the proof of claim \eqref{e:summable}.

From \eqref{def:c}, we have $p_\bmla=\sum_{\bmmu\in \bY}c_{\bmla\bmmu}m_{\bmmu}$. If we take $x_{N+1}=x_{N+2}=\cdots=0$, we get
$p_\bmla(x_1,x_2,\cdots,x_N) = \sum_{\bmmu\in \bY(N)}c_{\bmla\bmmu}m_{\bmmu}(x_1,x_2,\cdots, x_N)$.
Thus we can write $G(x_1,x_2,\cdots, x_N)$ as
\begin{align}\label{e:TaylorG}
G(x_1,x_2,\cdots, x_N)=\sum_{\bmla\in \bY}\sum_{\bmmu\in \bY(N)}a(\bmla)c_{\bmla\bmmu}m_{\bmmu}(x_1,x_2,\cdots, x_N)
=\sum_{\bmmu\in\bY(N)}\tilde a(\bmmu)m_{\bmmu}(x_1,x_2,\cdots, x_N),
\end{align}
where 
\begin{align*}
\tilde a(\bmmu)=\sum_{\bmla\in\bY}a(\bmla)c_{\bmla\bmmu}.
\end{align*}
We view \eqref{e:TaylorG} as a formal power series of $G(x_1,x_2,\cdots,x_N)$, it converges absolutely on \eqref{e:domain}. In fact, since $c_{\bmla\bmmu}\geq 0$ for any $\bmla,\bmmu\in \bY$, we have
\begin{align*}\begin{split}
&\phantom{{}={}}\sum_{\bmmu\in\bY(N)}|\tilde a(\bmmu)|m_{\bmmu}(|x|,|x|,\cdots, |x|)
\leq\sum_{\bmmu\in \bY(N)}\sum_{\bmla\in \bY}|a(\bmla)|c_{\bmla\bmmu}m_{\bmmu}(|x|,|x|,\cdots, |x|)\\
&=\sum_{\bmla\in\bY}|a(\bmla)|p_{\bmla}(|x|,|x|,\cdots, |x|)
=\sum_{\bmla\in\bY}|a(\bmla)|N^{\ell(\bmla)}|x|^{|\bmla|},
\end{split}\end{align*}
converges absolutely on $\{z\in \bC: |z|\leq 1+\varepsilon\}$.  Since the coefficients of the polynomials $m_\bmmu(x_1,x_2,\cdots,x_N)$ are all nonnegative, we conclude $G(x_1,x_2,\cdots, x_N)$ is analytic on
\begin{align*}
\{(z_1,z_2,\cdots,z_N)\in \bC^{N}: |z_i|<1+\varepsilon,\quad i=1,2,\cdots,N\}.
\end{align*}

The function $G(x_1,x_2,\cdots,x_N)$ is obtained from $G(\bmp)$ by setting $p_n=x_1^n+x_2^n+\cdots+x_N^n$ for $n\geq 1$. Since we are interested in the derivative of $\ln G(x_1,x_2,\cdots, x_N)$ at $x_1=x_2=\cdots=x_N=1$, we further take 
\begin{align*}
\Delta p_n =p_n-N=\sum_{i=1}^N (1+\Delta x_i)^n-N=\sum_{k\geq 1}{n\choose k}\sum_{i=1}^N(\Delta x_i)^k=\sum_{k\geq 1}{n\choose k}\tilde p_k, \quad n=1,2,3,\cdots,
\end{align*}
where 
\begin{align*}
\tilde p_k=p_k(\Delta x_1, \Delta x_2,\cdots, \Delta x_N), \quad {\tilde p}_{\bmla}=\prod_{i=1}^{\ell(\bmla)}\tilde p_{\la_i}.
\end{align*}
In this way we can rewrite \eqref{e:defcla} as 
\begin{align}\begin{split}\label{e:Gmp2}
\ln G(x_1,x_2,\cdots,x_N)
&=\sum_{\bmla\in \bY}\tilde c(\bmla) \prod_{i=1}^{\ell(\bmla)}\sum_{\mu=1}^{\la_i}{\la_i\choose \mu_i}\tilde p_{\mu_i}\\
&=\sum_{\bmla\in \bY}\tilde c(\bmla) \sum_{\bmmu\in \bY:\ell(\bmmu)=\ell(\bmla)}\frac{1}{\prod_{i\geq 1}m_i(\bmmu)!}\left(\sum_{\sigma\in S_{\ell(\bmla)}}\prod_{i=1}^{\ell(\bmla)}{\la_i\choose\mu_{\sigma(i)}}\right)\tilde p_{\bmmu}.
\end{split}\end{align}
Since $\tilde p_{\bmla}=p_\bmla(\Delta x_1,\Delta x_2,\cdots, \Delta x_N)=\sum_{\bmmu\in \bY(N)}c_{\bmla\bmmu}m_{\bmmu}(\Delta x_1, \Delta x_2,\cdots, \Delta x_N)$, we can rewrite \eqref{e:Gmp2} in terms of the monomial symmetric polynomials,
\begin{align*}\begin{split}
&\phantom{{}={}}\ln G(x_1,x_2,\cdots,x_N)
=\sum_{\bmla\in \bY}\tilde c(\bmla) \sum_{\bmmu\in \bY:\ell(\bmmu)=\ell(\bmla)}\frac{1}{\prod_{i\geq 1}m_i(\bmmu)!}\left(\sum_{\sigma\in S_{\ell(\bmla)}}\prod_{i=1}^{\ell(\bmla)}{\la_i\choose\mu_{\sigma(i)}}\right)\tilde p_{\bmmu}\\
&=\sum_{\bmla\in\bY}\tilde c(\bmla) \sum_{\bmmu\in\bY:\ell(\bmmu)=\ell(\bmla)}\frac{1}{\prod_{i\geq 1}m_i(\bmmu)!}\left(\sum_{\sigma\in S_{\ell(\bmla)}}\prod_{i=1}^{\ell(\bmla)}{\la_i\choose\mu_{\sigma(i)}}\right)\sum_{\bmeta\in \bY(N)}c_{\bmmu\bmeta}m_\bmeta(\Delta x_1, \Delta x_2,\cdots, \Delta x_N)\\
&=\sum_{\bmeta\in \bY(N)}\sum_{\bmmu\in\bY}\frac{c_{\bmmu\bmeta}}{\prod_{i\geq 1}m_i(\bmmu)!}\sum_{\bmla\in\bY:\ell(\bmmu)=\ell(\bmla)} \tilde c(\bmla) \left(\sum_{\sigma\in S_{\ell(\bmla)}}\prod_{i=1}^{\ell(\bmla)}{\la_i\choose\mu_{\sigma(i)}}\right)m_\bmeta(\Delta x_1, \Delta x_2,\cdots, \Delta x_N).
\end{split}\end{align*}
Thus the claim \eqref{eq:derG1} follows, for any $\bmeta\in \bY(N)$,
\begin{align}\label{e:derG3}
\frac{\del_{\bmeta}\ln G(x_1,x_2,\cdots, x_N)|_{1^N}}{\prod_{i=1}^{\ell(\bmeta)}\eta_i!}=\sum_{\bmmu\in \bY}\frac{c_{\bmmu\bmeta}}{\prod_{i\geq 1}m_i(\bmmu)!}\sum_{\la_1,\la_2,\cdots\la_{\ell(\bmmu)}\geq 1}  \prod_{i=1}^{\ell(\bmmu)}{\la_i\choose\mu_{i}} \left.\frac{\del^{\ell(\bmmu)}\ln G(\bmp)}{\del p_{\la_1}\del p_{\la_2}\cdots \del p_{\la_{\ell(\bmmu)}}}\right|_{\bmp=1^N},
\end{align}
where we used \eqref{e:defcla} and \eqref{e:defcla2}. This finishes the proof of \eqref{eq:derG1}, since $c_{\bmmu\bmeta}$ is nonzero for partitions $\bmmu$ with $|\bmmu|=|\bmeta|$ and $\ell(\bmmu)\geq \ell(\bmeta)$.

In the following we prove \eqref{eq:derG2}, which follows from inverting \eqref{e:derG3}. The monomial symmetric functions can be written in terms of the power-sum symmetric functions, $m_\bmla=\sum_{\bmmu}\tilde c_{\bmla\bmmu} p_{\bmmu}$. Then 
$
\sum_{\bmeta}c_{\bmmu\bmeta}\tilde c_{\bmeta\bmla}=\delta_{\bmmu\bmla}.
$
Moreover, we notice that $c_{\bmmu\bmeta}\neq 0$ only if $\ell(\bmeta)\leq \ell(\bmmu)$. Especially, if $\bmmu\in \bY(N)$, we have
$
\sum_{\bmeta\in \bY(N)}c_{\bmmu\bmeta}\tilde c_{\bmeta\bmmu}=1.
$
Thus, we can multiply $\tilde c_{\bmeta\bmmu}$ on both sides of \eqref{e:derG3}, and sum over $\bmeta\in \bY(N)$,
\begin{align*}
\sum_{\la_1,\la_2,\cdots\la_{\ell(\bmmu)}\geq 1}  \prod_{i=1}^{\ell(\bmmu)}{\la_i\choose\mu_{i}} \left.\frac{\del^{\ell(\bmmu)}\ln G(\bmp)}{\del p_{\la_1}\del p_{\la_2}\cdots \del p_{\la_{\ell(\bmmu)}}}\right|_{\bmp=1^N}=\prod_{i\geq 1}m_i(\bmmu)!\sum_{\bmeta\in \bY(N)}\tilde c_{\bmeta\bmmu}\frac{\del_{\bmeta}\ln G(x_1,x_2,\cdots,x_N)|_{1^N}}{\prod_{i=1}^{\ell(\bmeta)}\eta_i!}.
\end{align*}
This finishes the proof of \eqref{eq:derG2}, since for any fixed $\bmmu\in \bY(N)$, $\tilde c_{\bmeta\bmmu}$ is nonzero for partitions $\bmeta$ with $|\bmeta|=|\bmmu|$ and $\ell(\bmeta)\geq \ell(\bmmu)$.

\end{proof}

\begin{proof}[Proof of Theorem \ref{thm:equivalent}]
We first prove that Assumption \ref{a:LLN} is equivalent to \eqref{e:deri1} and \eqref{e:deri2}.
We assume that Assumption \ref{a:LLN} holds. We take $\bmeta=(k)$ in \eqref{eq:derG1} and let $N$ go to infinity,
\begin{align}\begin{split}\label{e:comptck}
&\phantom{{}={}}\lim_{N\rightarrow \infty}\left.\frac{\del_i^k \ln F_N(x_1, x_2,\cdots, x_N;\theta)}{N}\right|_{1^N}\\
&=\lim_{N\rightarrow \infty}k!\sum_{{\bmmu\in\bY(N):
|\bmmu|=k\atop\ell(\bmmu)\geq1}}\frac{c_{\bmmu (k)}}{\prod_{i\geq 1}m_i(\bmmu)!}\frac{1}{N}\sum_{\la_1,\la_2,\cdots\la_{\ell(\bmmu)}\geq 1}  \prod_{i=1}^{\ell(\bmmu)}{\la_i\choose\mu_{i}} \left.\frac{\del^{\ell(\bmmu)}\ln F_N(\bmp;\theta)}{\del p_{\la_1}\del p_{\la_2}\cdots \del p_{\la_{\ell(\bmmu)}}}\right|_{\bmp=1^N}.
\end{split}\end{align}
From Assumption \eqref{a:LLN}, the only term on the righthand side of \eqref{e:comptck} which will contribute are those with $\ell(\bmmu)=1$. The only such term is $\bmmu=(k)$, and we have
\begin{align}\label{e:derlnG-1}
\lim_{N\rightarrow \infty}\left.\frac{\del_i^k \ln F_N(x_1, x_2,\cdots, x_N;\theta)}{N}\right|_{1^N}
&=\lim_{N\rightarrow \infty}\frac{k!}{N}\sum_{\la\geq 1} {\la\choose k} \left.\frac{\del \ln F_N(\bmp;\theta)}{\del p_{\la}}\right|_{\bmp=1^N}=\fc_k.
\end{align}
This finishes the proof of \eqref{e:deri1}. If $\ell(\bmeta)\geq 2$, similarly by taking $N$ to infinity in \eqref{eq:derG1}, we get
\begin{align}\begin{split}\label{e:derlnG0}
&\phantom{{}={}}
\lim_{N\rightarrow\infty}\frac{\del_{\bmeta}\ln F_N(x_1,x_2,\cdots, x_N;\theta)|_{1^N}}{N}
\\&=\lim_{N\rightarrow\infty}\frac{1}{N}{\prod_{i=1}^{\ell(\bmeta)}\eta_i!}\sum_{\bmmu\in\bY:
|\bmmu|=|\bmeta|\atop\ell(\bmmu)\geq\ell(\bmeta)}\frac{c_{\bmmu\bmeta}}{\prod_{i\geq 1}m_i(\bmmu)!}\sum_{\la_1,\la_2,\cdots\la_{\ell(\bmmu)}\geq 1}  \prod_{i=1}^{\ell(\bmmu)}{\la_i\choose\mu_{i}} \left.\frac{\del^{\ell(\bmmu)}\ln F_N(\bmp;\theta)}{\del p_{\la_1}\del p_{\la_2}\cdots \del p_{\la_{\ell(\bmmu)}}}\right|_{\bmp=1^N}.
\end{split}\end{align}
All the terms on the righthand side of \eqref{e:derlnG0} correspond to some $\bmmu$ with $\ell(\bmmu)\geq 2$. Thus by Assumption \ref{a:LLN}, the righthand side of \eqref{e:derlnG0} vanishes, and this finishes the proof of \eqref{e:deri2}.

We assume that \eqref{e:deri1} and \eqref{e:deri2} holds. We take $\bmmu=(k)$ and let $N$ go to infinity in \eqref{eq:derG2}, 
\begin{align}\begin{split}\label{e:derlnG1}
&\phantom{{}={}}\lim_{N\rightarrow \infty}\frac{k!}{N}\sum_{\la\geq 1}  {\la\choose k} \left.\frac{\del \ln F_N(\bmp;\theta)}{\del p_{\la}}\right|_{\bmp=1^N}\\
&=\lim_{N\rightarrow\infty}\frac{k!}{N}\prod_{i\geq 1}m_i(\bmmu)!\sum_{\bmeta\in \bY(N):|\bmeta|=k\atop \ell(\bmeta)\geq 1}\tilde c_{\bmeta\bmmu}\frac{\del_{\bmeta}\ln F_N(x_1,x_2,\cdots,x_N;\theta)|_{1^N}}{\prod_{i=1}^{\ell(\bmeta)}\eta_i!}.
\end{split}\end{align}
From \eqref{e:deri1} and \eqref{e:deri2}, the only terms on the righthand side of \eqref{e:derlnG1} which will contribute are those with $\ell(\bmeta)=1$. The only such term is $\bmeta=(k)$, and we have
\begin{align}\label{e:derlnG-2}
\lim_{N\rightarrow \infty}\frac{k!}{N}\sum_{\la\geq 1} {\la\choose k} \left.\frac{\del \ln F_N(\bmp;\theta)}{\del p_{\la}}\right|_{\bmp=1^N}=\lim_{N\rightarrow \infty}\left.\frac{\del_i^k \ln F_N(x_1, x_2,\cdots, x_N;\theta)}{N}\right|_{1^N}
=\fc_k.
\end{align}
This finishes the proof of \eqref{e:LLN1}. If $\ell(\bmmu)\geq 2$, similarly by taking $N$ to infinity in \eqref{eq:derG2}, 
\begin{align}\begin{split}\label{e:derlnG2}
&\phantom{{}={}}\lim_{N\rightarrow\infty}\frac{1}{N}\sum_{\la_1,\la_2,\cdots\la_{\ell(\bmmu)}\geq 1}  \prod_{i=1}^{\ell(\bmmu)}{\la_i\choose\mu_{i}} \left.\frac{\del^{\ell(\bmmu)}\ln F_N(\bmp;\theta)}{\del p_{\la_1}\del p_{\la_2}\cdots \del p_{\la_{\ell(\bmmu)}}}\right|_{\bmp=1^N}\\
&=\lim_{N\rightarrow\infty}\frac{1}{N}\prod_{i\geq 1}m_i(\bmmu)!\sum_{\bmeta\in \bY(N):|\bmeta|=|\bmmu|\atop \ell(\bmeta)\geq\ell(\bmmu)}\tilde c_{\bmeta\bmmu}\frac{\del_{\bmeta}\ln F_N(x_1,x_2,\cdots,x_N;\theta)|_{1^N}}{\prod_{i=1}^{\ell(\bmeta)}\eta_i!}.
\end{split}\end{align}
All the terms on the righthand side of \eqref{e:derlnG2} correspond to some $\bmeta$ with $\ell(\bmeta)\geq 2$. Thus by \eqref{e:deri2}, the righthand side of \eqref{e:derlnG2} vanishes, and this finishes the proof of \eqref{e:LLN3}.

In the following we prove that Assumption \ref{a:CLT} is equivalent to \eqref{e:deri1copy}, \eqref{e:deri3}, and \eqref{e:deri4}. We assume that Assumption \ref{a:CLT} holds. \eqref{e:deri1copy} follows from \eqref{e:derlnG-1}. For \eqref{e:deri3}, we take $\bmeta=(k,l)$ in \eqref{eq:derG1} and let $N$ go to infinity,
\begin{align}\begin{split}\label{e:compdkl}
&\phantom{{}={}}\lim_{N\rightarrow \infty}\left.\del_i^k\del_j^l \ln F_N(x_1,x_2,\cdots, x_N;\theta)\right|_{1^N}\\
&=\lim_{N\rightarrow \infty}k!l!\sum_{\bmmu\in \bY(N): |\bmmu|=k+l,\atop \ell(\bmmu)\geq 2}\frac{c_{\bmmu (k,l)}}{\prod_{i\geq 1}m_i(\bmmu)!}\sum_{\la_1,\la_2,\cdots\la_{\ell(\bmmu)}\geq 1}  \prod_{i=1}^{\ell(\bmmu)}{\la_i\choose\mu_{i}} \left.\frac{\del^{\ell(\bmmu)}\ln F_N(\bmp;\theta)}{\del p_{\la_1}\del p_{\la_2}\cdots \del p_{\la_{\ell(\bmmu)}}}\right|_{\bmp=1^N}.
\end{split}\end{align}
From Assumption \eqref{a:CLT}, the only terms on the righthand side of \eqref{e:compdkl} which will contribute are those with $\ell(\bmmu)=2$. The only such term is $\bmmu=(k,l)$. If $k\neq l$, $c_{\bmmu(k,l)}=1$; if $k=l$, $c_{\bmmu(k,l)}=2$. In both cases $c_{\bmmu(k,l)}=\prod_{i\geq 1} m_i(\bmmu)!$. Thus we have
\begin{align*}\begin{split}
\lim_{N\rightarrow \infty}\left.\del_i^k\del_j^l \ln F_N(x_1,x_2,\cdots, x_N;\theta)\right|_{1^N}=\lim_{N\rightarrow \infty}k!l!\sum_{\la_1,\la_2\geq 1}  {\la_1\choose k}{\la_2\choose l} \left.\frac{\del^2 \ln F_N(\bmp;\theta)}{\del p_{\la_1}\del p_{\la_2}}\right|_{\bmp=1^N}=\fd_{k,l}.
\end{split}\end{align*}
This finishes the proof of \eqref{e:deri3}. If $\ell(\bmeta)\geq 3$, similarly by taking $N$ to infinity in \eqref{eq:derG1}, we get
\begin{align}\begin{split}\label{e:derlnG3}
&\phantom{{}={}}
\lim_{N\rightarrow\infty}\del_{\bmeta}\ln F_N(x_1,x_2,\cdots, x_N)|_{1^N}
\\&=\lim_{N\rightarrow\infty}{\prod_{i=1}^{\ell(\bmeta)}\eta_i!}\sum_{\bmmu\in\bY(N):
|\bmmu|=|\bmeta|\atop\ell(\bmmu)\geq\ell(\bmeta)}\frac{c_{\bmmu\bmeta}}{\prod_{i\geq 1}m_i(\bmmu)!}\sum_{\la_1,\la_2,\cdots\la_{\ell(\bmmu)}\geq 1}  \prod_{i=1}^{\ell(\bmmu)}{\la_i\choose\mu_{i}} \left.\frac{\del^{\ell(\bmmu)}\ln F_N(\bmp;\theta)}{\del p_{\la_1}\del p_{\la_2}\cdots \del p_{\la_{\ell(\bmmu)}}}\right|_{\bmp=1^N}.
\end{split}\end{align}
All the terms on the righthand side of \eqref{e:derlnG3} correspond to some $\bmmu$ with $\ell(\bmmu)\geq 3$. Thus by Assumption \ref{a:LLN}, the righthand side of \eqref{e:derlnG3} vanishes, and this finishes the proof of \eqref{e:deri4}.

We assume that \eqref{e:deri1copy}, \eqref{e:deri3}, and \eqref{e:deri4} hold. \eqref{e:CLT1} follows from \eqref{e:derlnG-2}. For \eqref{e:CLT2}, we take $\bmmu=(k,l)$ in \eqref{eq:derG2},
\begin{align}\begin{split}\label{e:derlnG4}
&\phantom{{}={}}\lim_{N\rightarrow\infty}k!l!\sum_{\la_1,\la_2\geq 1}  {\la_1\choose k} {\la_2\choose l}\left.\frac{\del^{2}\ln F_N(\bmp;\theta)}{\del p_{\la_1}\del p_{\la_2}}\right|_{\bmp=1^N}\\
&=\lim_{N\rightarrow\infty}k!l!\prod_{i\geq 1}m_i(\bmmu)!\sum_{\bmeta\in \bY(N):|\bmeta|=k+l\atop \ell(\bmeta)\geq2}\tilde c_{\bmeta\bmmu}\frac{\del_{\bmeta}\ln F_N(x_1,x_2,\cdots,x_N;\theta)|_{1^N}}{\prod_{i=1}^{\ell(\bmeta)}\eta_i!}.
\end{split}\end{align}
From \eqref{e:deri4}, the only terms on the righthand side of \eqref{e:compdkl} which will contribute are those with $\ell(\bmeta)=2$. The only such term is $\bmeta=(k,l)$. If $k\neq l$, $\tilde c_{\bmeta(k,l)}=1$; if $k=l$, $\tilde c_{\bmeta(k,l)}=1/2$. In both cases $\tilde c_{\bmeta(k,l)}\prod_{i\geq 1} m_i(\bmmu)!=1$. Thus we have
\begin{align*}\begin{split}
\lim_{N\rightarrow \infty}k!l!\sum_{\la_1,\la_2\geq 1}  {\la_1\choose k}{\la_2\choose l} \left.\frac{\del^2 \ln F_N(\bmp;\theta)}{\del p_{\la_1}\del p_{\la_2}}\right|_{\bmp=1^N}=\lim_{N\rightarrow \infty}\left.\del_i^k\del_j^l \ln F_N(x_1,x_2,\cdots, x_N;\theta)\right|_{1^N}=\fd_{k,l}.
\end{split}\end{align*}
This finishes the proof of \eqref{e:CLT2}.
If $\ell(\bmmu)\geq 3$, similarly by taking $N$ to infinity in \eqref{eq:derG2}, 
\begin{align}\begin{split}\label{e:derlnG5}
&\phantom{{}={}}\lim_{N\rightarrow\infty}\sum_{\la_1,\la_2,\cdots\la_{\ell(\bmmu)}\geq 1}  \prod_{i=1}^{\ell(\bmmu)}{\la_i\choose\mu_{i}} \left.\frac{\del^{\ell(\bmmu)}\ln F_N(\bmp;\theta)}{\del p_{\la_1}\del p_{\la_2}\cdots \del p_{\la_{\ell(\bmmu)}}}\right|_{\bmp=1^N}\\
&=\lim_{N\rightarrow\infty}\prod_{i\geq 1}m_i(\bmmu)!\sum_{\bmeta\in \bY(N):|\bmeta|=|\bmmu|\atop \ell(\bmeta)\geq\ell(\bmmu)}\tilde c_{\bmeta\bmmu}\frac{\del_{\bmeta}\ln F_N(x_1,x_2,\cdots,x_N;\theta)|_{1^N}}{\prod_{i=1}^{\ell(\bmeta)}\eta_i!}.
\end{split}\end{align}
All the terms on the righthand side of \eqref{e:derlnG5} correspond to some $\bmeta$ with $\ell(\bmeta)\geq 3$. Thus by \eqref{e:deri4}, the righthand side of \eqref{e:derlnG5} vanishes, and this finishes the proof of \eqref{e:CLT5}.

\end{proof}

\bibliography{References}{}
\bibliographystyle{plain}

\end{document}